\documentclass[11pt]{article}
\pdfoutput=1
\usepackage{fullpage}

\usepackage[utf8]{inputenc}

\usepackage{mathtools}
\usepackage{dsfont}
\usepackage{accents}
\usepackage{hyperref}
\usepackage{graphicx}
\usepackage{bm}
\usepackage{mathdots}
\usepackage{braket}
\usepackage{xstring}
\usepackage{euscript}
\usepackage[authoryear]{natbib}
\setcitestyle{authoryear,open={(},close={)}}

\usepackage{enumitem}
\usepackage[autolanguage]{numprint}

\usepackage{eurosym}
\usepackage{footnote}
\usepackage{amsfonts}
\usepackage{mathbbol}

\DeclareSymbolFontAlphabet{\bboldmathbb}{bbold}%
\DeclareSymbolFontAlphabet{\mathbb}{AMSb}

\DeclarePairedDelimiterX{\Iintv}[1]{\llbracket}{\rrbracket}{\iintvargs{#1}}

\newcommand{\DP}{Dynamic Programming}
\newcommand{\pomdp}{\textsc{Pomdp}}
\newcommand{\dpomdp}{\textsc{Det-Pomdp}}

\newcommand{\Mdpomdp}{Separated \dpomdp}
\newcommand{\mdpomdp}{separated \dpomdp}

\newcommand{\mdp}{\textsc{Mdp}}

\newcommand{\MdpomdpFunctionSet}{$\np{\cemetery}$-Separated Mapping Set}

\newcommand{\cemeterySeparated}{$\np{\cemetery}$-separated}

\newcommand{\cemeterySeparation}{$\np{\cemetery}$-separation}

\newcommand{\separatedMappingSet}{separated mapping set}

\newcommand{\fullpomdp}{Partially Observed Markov Decision Process}
\newcommand{\fulldpomdp}{Deterministic Partially Observed Markov Decision Process}
\newcommand{\fullmdpomdp}{Separated Deterministic Partially Observed Markov Decision Process}
\newcommand{\fullmdp}{Markov Decision Process}

\newcommand{\DPA}{\textsc{Dp} Algorithm}
\newcommand{\SARSOP}{\textsc{Sarsop}}

\ifdefined\graph \renewcommand{\graph}{\mathcal{G}} \else \newcommand{\graph}{\mathcal{G}}\fi

\newcommand{\timeindex}{t}

\newcommand{\timeset}{\mathbb{T}} 
\ifdefined\horizon \renewcommand{\horizon}{T} \else \newcommand{\horizon}{T}\fi
\newcommand{\timesetNoHorizon}{\timeset \setminus \na{\horizon}}

\newcommand{\dataTuple}{\mathcal{D}}

\newcommand{\cost}{c}

\newcommand{\states}{x}
\newcommand{\nextstates}{y}
\newcommand{\statesva}{\bm{X}}
\newcommand{\controls}{u}
\newcommand{\controlsva}{\bm{U}}
\newcommand{\admcontrolset}{\mathbb{U}^{\mathrm{ad}}} 

\ifdefined\dynamics \renewcommand{\dynamics}{\varphi} \else \newcommand{\dynamics}{\varphi} \fi

\newcommand{\observer}{o}
\newcommand{\observerva}{\bm{O}}
\newcommand{\noisesobserver}{v}

\newcommand{\observerfunct}{\beta}
\newcommand{\belief}{b}

\newcommand{\valuefct}{V}
\newcommand{\knownproba}{b}

\newcommand{\observationNoisesKernelTime}[1]{\PP_{\noisesobserver, \timeindex}}

\ifdefined\policy \renewcommand{\policy}{\pi} \else\newcommand{\policy}{\pi}\fi
\newcommand{\probvalue}{\mathcal{V}}

\newcommand{\beliefspace}{\mathbb{B}}

\newcommand{\OO}{\mathbb{O}}

\newcommand{\reachablebeliefspace}{\mathcal{R}^\dataTuple}
\newcommand{\BBR}{\reachablebeliefspace}

\newcommand{\beliefdynamics}{\theta} 

\newcommand{\ProbaObservation}{Q}

\newcommand{\beliefadmcontrolset}{\mathbb{U}^{\mathrm{b, ad}}}

\ifdefined\AA \renewcommand{\AA}{\mathbb{A}} \else \newcommand{\AA}{\mathbb{A}}\fi
\newcommand{\GG}{\mathbb{G}}

\newcommand{\pushforwardtransition}{\phi}

\newcommand{\restrictionsetsPushForward}[2]{\setsPushForward_{{#1} \to {#2}}}
\newcommand{\restrictionFunctionSubsetSet}[3]{#1_{{#2} \to {#3}}}

\newcommand{\barXX}{\overline{\XX}}
\newcommand{\cemeteryBelief}{\delta_{\partial}}
\newcommand{\cemetery}{\partial}

\newcommand{\measureOnBarXX}{\nu}
\newcommand{\measureOnBarXXprime}{\nu'}

\newcommand{\Renormalization}{{\cal N}}

\newcommand{\observerdiscretization}{m}
\newcommand{\inistatediscretization}{n}
\newcommand{\controlsdiscretization}{d}

\newcommand{\supp}{\mathrm{supp}}
\newcommand{\image}{\mathrm{Im}}
\newcommand{\modulo}{\mathrm{mod}}

\newcommand{\costfunct}{\mathcal{L}}
\newcommand{\finalcost}{\mathcal{K}}

\newcommand{\Mapforward}[2]{{#1}_{\overrightarrow{#2}}}
\newcommand{\Mapbackward}[2]{{#1}_{\overleftarrow{#2}}}
\newcommand{\Mappings}[2]{\mathbb{L}\np{#1;#2}}

\newcommand{\BackwardSet}[1]{\mathbb{G}_{\overleftarrow{#1}}}
\newcommand{\ForwardMappings}{$\bp{\overrightarrow{\PRIMAL}}$-mappings}
\newcommand{\MForwardMappings}{$\bp{\mathbb{M},\overrightarrow{\PRIMAL}}$-mappings}

\newcommand{\BackwardMappings}{$\bp{\overleftarrow{\PRIMAL}}$-mappings}
\newcommand{\MBackwardMappings}{$\bp{\mathbb{M},\overleftarrow{\PRIMAL}}$-mappings}
\newcommand{\MeBackwardMappings}[1]{$\bp{#1,\overleftarrow{\PRIMAL}}$-mappings}
\newcommand{\MkBackwardMappings}[1]{$\bp{\mathbb{M}_{#1},\overleftarrow{\PRIMAL}}$-mappings}

\newcommand{\generalFunction}{g}
\newcommand{\generalFunctionSet}{\GG}

\newcommand{\generalFunctionBisSet}{\bboldmathbb{\Psi}}

\def\DUALP{\PRIMAL}
\def\DualP{\Primal}
\def\dualP{\primal}
\def\kindex{k}

\def\DualV{V}
\def\dualV{v}

\newcommand{\setsBeliefDynamics}{\Theta^{\dataTuple}}
\newcommand{\setsBeliefDynamicsUpInd}[1]{\Theta^{\dataTuple, #1}}
\newcommand{\setsPushForward}{\Phi^{\dataTuple}}
\newcommand{\setsPushForwardUpInd}[1]{\Phi^{\dataTuple, #1}}

\newcommand{\intervalStates}{\observerfunct^{-1}}

\usepackage{dcolumn}
\newcolumntype{d}[1]{D{.}{.}{#1}}

\usepackage{amsthm}
\usepackage{amsmath}
\usepackage{amssymb}
\usepackage{stmaryrd}

\usepackage[colors]{optsys}

\date{\today}

\date{March 30, 2023}

\title{%
  Complexity Bounds for Deterministic \\ Partially Observed Markov Decision Processes}

\author{
  Cyrille Vessaire\footnote{CERMICS, Ecole des Ponts, Marne-la-Vall\'ee, France},
  \and Jean-Philippe Chancelier\footnotemark[1],
  \and Michel {De Lara}\footnotemark[1],
  \and Pierre Carpentier\footnote{UMA, ENSTA Paris, Institut Polytechnique de Paris,
    Palaiseau, France},
  \and Alejandro {Rodr\'iguez-Mart\'inez}\footnote{IAM, TotalEnergies SE, Pau, France}
}

\begin{document}

\maketitle

\begin{abstract} 
  \fullpomdp es (\pomdp) share the structure of  \fullmdp s (\mdp) --- with stages,  states, actions, 
  probability transitions, rewards --- but for the notion of solutions.
  In a \pomdp, observation mappings provide partial and/or imperfect knowledge
  of the state, and a policy maps observations (and not states like in a \mdp)
  towards actions. Theroretically, a \pomdp\ can be solved by \DP\ (DP), but with an
  information state made of probability distributions over the original state,
  hence DP suffers from the curse of dimensionality, even in the finite case.
  This is why, authors like \citep{littman_thesis} and~\citep{bonet_deterministic_pomdp}
  have studied the subclass of so-called \fulldpomdp es (\dpomdp), 
  where transitions and observations mappings are deterministic.
  In this paper, we improve on Littman’s complexity bounds.
  We then
  introduce and study a more restricted class, \Mdpomdp s, and give some new complexity
  bounds for this class.
\end{abstract}

\section{Introduction}
\label{sect:intro_dpomdp}

\fullmdp es (\mdp s) form a versatile framework used to model a wide range of optimization
problems. More precisely, the formalism of \mdp s is adapted to optimize discrete time
controlled dynamical systems under stochasticity. It is popular in both optimal control
and machine learning community, as it can be used to model complex real-life problems (see
the survey \citep{white_survey_1993} for common applications). Moreover, it provides the
mathematical foundations for Reinforcement Learning (see \citep{Sutton1998}), and
algorithms such as Policy Iteration and \DP\ can efficiently solve \mdp s.

The \mdp\ model consists of sets of states, actions, time
steps, rewards, and transition probabilities. When in a given state and at a given time,
the decision-maker’s action generates a reward and determines the state at the next time
step according to the transition probability function.

However, \mdp s assume that, when making an action, the decision-maker knows the state (as
solution policies map states towards actions). By contrast, in a \fullpomdp\ (\pomdp),
observation mappings provide partial and/or imperfect knowledge of the state, and a policy
maps observations towards actions. An extensive literature exists on \pomdp s, most of
which focuses on the infinite horizon case. \pomdp s can be applied to numerous fields,
from medical models (as in~\citep{steimle_multi-model_2021}) to robotics (as
in~\citep{pajarinen_robotic_2017}) to name a few. Algorithms based on \DP\ (see
\citep{bellman_1957}) have been designed to exploit specific structures in \pomdp s in
order to solve this difficult class of problems. They do so by first reformulating the
problem through the use of beliefs (probability distributions over the state space), as
in~\citep{sondik_pomdp}. One such algorithm is \SARSOP, described
in~\citep{sarsop_kurniawati_2008}.
\pomdp s are often untractable in the general
case as \DP\ suffers from the curse of dimensionality. Indeed, working with beliefs
implies working on the space of distributions over the state space, which is, by nature,
an infinite continuous space.

Different attemps have been made to handle 
the curse of dimensionality, in specific cases of \pomdp s.
{The case of MDP with unknown transition probabilities is addressed in the literature.
References include~\cite{Burnetas-Katehakis-1997}, who explored optimal
adaptive policies for Markov Decision Processes, providing foundational insights
into decision processes under uncertainty, and~\cite{Cowan-et-all-2018,Cowan-et-all-2019}
on accelerating the computation of UCB and related
indices for reinforcement learning that offer important perspectives on
optimization and computational efficiency.}
Regarding our work, 
we consider the subclass where transitions and
observations mappings are deterministic, named \fulldpomdp es (\dpomdp). That subclass of
problems has been studied by~\citep{littman_thesis} and~\citep{bonet_deterministic_pomdp}.
It was first considered as a limit case of \pomdp s by Littman, mainly used to illustrate
the complexity of \pomdp s when considering as few sources of uncertainties as possible.
For Bonet, \dpomdp s became of interest after some applications were found. He presented
examples in \cite[\S2]{bonet_deterministic_pomdp}, such as the navigation of a robot in a
partially observed terrain.

In this paper, we introduce and study a restricted subclass of \dpomdp s, that we call
\emph{\Mdpomdp s}. With this new class, we are able to push back the curse of dimensionality.

\medskip

The paper is organized as follows.
In Sect.~\ref{subsect:formulation_dpomdp}, we present the general formulation of \dpomdp.
In Sect.~\ref{sect:dp_dpomdp_complexity},
we present \DP\ on beliefs for \dpomdp s with constraints, and we give new complexity bounds.
In Sect.~\ref{sect:mdpomdp}, we introduce a subclass of \dpomdp, \Mdpomdp,
and we give new improved complexity bounds.
In Sect.~\ref{sect:dpomdp_illustration}, we illustrate the numerical solving of \Mdpomdp\ with an example:
emptying a tank containing water when considering partial observation of the level of
water in the tank.
Finally, in Appendix~\ref{sect:technical_lemmatas_pushforward}, we
present technical lemmata and considerations on pushforward measures,
and in Appendix~\ref{sect:complements-det-pomdp} we present complements on \Mdpomdp s.

We now detail our main contributions.
In Sect.~\ref{sect:dp_dpomdp_complexity}, we improve \cite{littman_thesis} bound on the
cardinality of the set of reachable beliefs 
for \dpomdp s (Theorem~\ref{th:reachable_belief_bounds}).
This new bound comes
from a new representation of the belief dynamics in \dpomdp s using the notion of
\emph{pushforward measure} (Lemma~\ref{lem:tau_as_pushforward}).
In Sect.~\ref{sect:mdpomdp}, we introduce a subclass of \dpomdp s, \Mdpomdp s. As shown in
Theorem~\ref{th:bound_belief_space_mdpomdp}, the interest of \Mdpomdp s is that they
further push back the curse of dimensionality for \DP\ with beliefs.
Moreover, this last
bound is tight (Proposition~\ref{prop:reaching_the_bound}).

\section{Formulation of \fulldpomdp es}
\label{subsect:formulation_dpomdp}

A \dpomdp\ is a particular case of \pomdp s, itself an extension of \fullmdp es (\mdp s).
Backgrounds on \mdp s can be found in \cite{puterman_1994}, whereas backgrounds on
\pomdp s can be found in \cite{bertsekas_shreve}.
As with \mdp s, the model consists of 
stages (times, time steps), states, controls (also called actions), and
probability transitions.
At each stage, the decision-maker (also called the agent)
chooses a given action, which generates a random reward depending on both
current stage and state. The state then transits to its next random value. However, in the
case of \dpomdp s (and \pomdp s), the decision-maker has only partial knowledge of the state
of the dynamical system. Instead, he has access to functions of the state and controls:
the \emph{observations}. For \dpomdp s, the transitions and observations are given by
deterministic evolution and observation mappings. Moreover, the initial state is not
known beyond an initial probability distribution.

First, we present the ingredients of a
\dpomdp. Second, we present the formulation of a \dpomdp\ optimization problem.
We use the notation \( \ic{j,k}=\na{j, j+1,\ldots,k-1,k} \) for any pair of
natural numbers such that \( j \leq k \).
We call \emph{pair} \( \na{a,b} \) a subset made of 1 ($a=b$) or 2 ($a \neq b$)
elements.
We call \emph{couple} or \emph{ordered pair} \( \np{a,b} \) an element of a
Cartesian product. 

\subsubsubsection{Ingredients of a \dpomdp}
A \dpomdp\ is defined by the tuple
\begin{equation}
\dataTuple = \bp{\timeset, \UU, \OO, \XX,
  \nseqa{\costfunct_{\timeindex}}{\timeindex \in \timeset},
  \nseqa{\dynamics_{\timeindex}}{\timeindex \in \timesetNoHorizon},
  \nseqa{\admcontrolset_{\timeindex}}{\timeindex \in \timesetNoHorizon},
  \nseqa{\observerfunct_{\timeindex}}{\timeindex \in \timeset}}
\eqfinv
\label{eq:def_data_tuple_dpomdp}
\end{equation}
which we now detail\footnote{For simplicity, we assume that the sets $\UU$, $\OO$ and $\XX$ are not
  indexed by time}.

The set $\timeset = \ic{0, \horizon}$ is the set of \emph{stages} (\emph{times},
\emph{time-steps}), where the positive integer
$\horizon \in \NN \setminus \na{0}$ is colloquially known as the \emph{horizon}.
The set~$\UU$ is the set of \emph{controls} the decision-maker can choose from.
The set $\OO$ is the set of \emph{observations} available to the decision-maker.
The set $\XX$ is the set of \emph{states}.
The collection $\nseqa{\costfunct_{\timeindex}}{\timeindex \in \timeset}$ is
made of \emph{instantaneous cost functions}
$\costfunct_{\timeindex}: \XX \times \UU \to \RR \cup \na{+\infty}$,
for all time $\timeindex \in \timesetNoHorizon$,
with the special \emph{final cost function}~$\costfunct_{\horizon}$ denoted by
$\finalcost: \XX \to \RR \cup \na{+\infty}$.
The collection $\nseqa{\dynamics_{\timeindex}}{\timeindex \in
  \timesetNoHorizon}$ is made of \emph{dynamics} (evolution mappings or
transitions), that is, mappings\footnote{%
Adopting usage in mathematics, we follow Serge Lang and use ``function'' only to
    refer to mappings in which the codomain is a set of numbers (i.e. a subset
    of~$\RR$ or $\CC$, or their possible extensions with $\pm \infty$),
    and reserve the term mapping for more general codomains.}
$\dynamics_{\timeindex}: \XX \times \UU \to \XX$, for all time $\timeindex \in \timesetNoHorizon$.
The collection $\nseqa{\admcontrolset_{\timeindex}}{\timeindex \in
  \timesetNoHorizon}$ is made of \emph{admissibility constraints}: for all time $\timeindex \in \timesetNoHorizon$,
$\admcontrolset_{\timeindex}: \XX\rightrightarrows \UU$ is a set-valued mapping
from $\XX$ to $\UU$, that is, for all state $\states \in \XX$, the admissible
controls at time $\timeindex$ are given by 
$\admcontrolset_{\timeindex}\np{\states} \subset \UU$.
The collection $\nseqa{\observerfunct_{\timeindex}}{\timeindex \in \timeset}$ is
made of \emph{observation mappings}: the initial observation mapping is 
$\observerfunct_{0}: \XX \to \OO$ whereas, for all time
$\timeindex \in \timeset \setminus \na{0}$, the observation mappings
are $\observerfunct_{\timeindex}: \XX \times \UU \to \OO$.

When considering \dpomdp, we initialize the initial state with a probability
distribution. We hence need to introduce a probability space as the tuple $\dataTuple$
does not contain any stochastic element.
Let $\Omega$ be the set of possible outcomes and $\PP$ a
probability measure on $\Omega$, such that 
$\forall \omega \in \Omega$, $\PP\np{\omega} > 0$ (hence $\Omega$ is countable). We denote
by $\espe$ the mathematical expectation operator.

In this paper, we only consider \dpomdp s which satisfy the following finite sets
assumption.
\begin{assumption}[Finite sets assumption]
  The sets of stages~$\timeset$, of states~$\XX$, of controls~$\UU$, of
  observations~$\OO$ and of possible outcomes~$\Omega$ have finite cardinality. 
  \label{assumpt:pomdp_finite_sets}
\end{assumption}
As a consequence, the horizon is finite: $\horizon < +\infty$.

For a finite set~$\YY$, 
the cardinality of~$\YY$ is denoted by~$\cardinal{\YY}$,
and the set of probability distributions over~$\YY$ by~$\Delta\np{\YY}$.
Moreover, for any nonnegative measure $\mu$ on~$\YY$, we define
the \emph{support} of the measure~$\mu$ by 
\begin{equation}
  \supp \np{\mu} = \bset{ y \in \YY}{ \mu(\na{y}) > 0} \subset \YY
  \eqfinp
  \label{eq:def_support_measure}
\end{equation}

\subsubsubsection{Formulation of a \dpomdp\ optimization problem}

A finite-horizon \dpomdp~optimization problem is formulated,
for any initial belief $\belief_0 \in \Delta\np{\XX}$, by 
\begin{subequations}
  \begin{align}
    \probvalue^{\star} \np{\knownproba_0} =
    \min_{\statesva, \observerva, \controlsva}\;
    &
      \espe \Bc{\sum_{\timeindex = 0}^{\horizon-1} \costfunct_{\timeindex}(
      \statesva_{\timeindex}, \controlsva_{\timeindex}) +
      \finalcost(\statesva_{\horizon})}
      \label{eq:dpomdp_gen_obj}
    \\
    s.t. ~~
    & \PP_{\statesva_{0}} = \knownproba_0
    \eqfinv
        \label{eq:dpomdp_gen_initialisation}\\
    & \statesva_{\timeindex+1} = \dynamics_{\timeindex}(\statesva_{\timeindex},
      \controlsva_{\timeindex})
      \eqsepv \forall \timeindex \in \timeset \setminus \na{\horizon} \eqfinv
      \label{eq:dpomdp_gen_dynamics_def}
    \\
    & \observerva_{0} = \observerfunct_{0}(\statesva_{0})
    \eqfinv
      \label{eq:dpomdp_gen_ini_observation_def}
    \\
    & \observerva_{\timeindex+1} = \observerfunct_{\timeindex+1}(\statesva_{\timeindex+1},
      \controlsva_{\timeindex}) \eqsepv
      \forall \timeindex \in \timeset \setminus \na{\horizon} \eqfinv
      \label{eq:dpomdp_gen_observation_def}
    \\
    & \controlsva_{\timeindex} \in \admcontrolset_{\timeindex}(\statesva_{\timeindex})
      \eqsepv \forall \timeindex \in \timeset \setminus \na{\horizon}
    \eqfinv
      \label{eq:dpomdp_gen_control_adm}
    \\
    & \sigma(\controlsva_{\timeindex}) \subset \sigma(\observerva_{0}, \dots,
      \observerva_{\timeindex}, \controlsva_{0}, \dots, \controlsva_{\timeindex - 1})
      \eqsepv \forall \timeindex \in \timeset \setminus \na{\horizon}
      \label{eq:dpomdp_gen_nonanticipativity}
      \eqfinp
  \end{align}
  \label{eq:dpomdp_general_formulation}
\end{subequations}
In Problem~\eqref{eq:dpomdp_general_formulation}, there are three processes
$\statesva = \ba{\statesva_{\timeindex}}_{\timeindex \in \timeset}$,
$\controlsva = \ba{\controlsva_{\timeindex}}_{\timeindex \in \timeset \setminus
  \na{\horizon}}$ and
$\observerva = \ba{\observerva_{\timeindex}}_{\timeindex \in \timeset}$. For all time
$\timeindex \in \timeset$, $\statesva_{\timeindex}: \Omega \to \XX$ and
$\observerva_{\timeindex}: \Omega \to \OO$ are random variables representing respectively
the state and the observation variables of the system at time $\timeindex$, and for all
time $\timeindex \in \timesetNoHorizon$, $\controlsva_{\timeindex}: \Omega \to \UU$ is a
random variable representing the control at time $\timeindex$.

The optimization criterion of Problem~\eqref{eq:dpomdp_general_formulation} is given by
Equation~\eqref{eq:dpomdp_gen_obj}.
We now detail the constraints of the optimization
Problem~\eqref{eq:dpomdp_general_formulation}.
First, Equation~\eqref{eq:dpomdp_gen_initialisation} is the \emph{initialization}
constraint. As the initial state is not fully known, we instead use the probability
distribution $\belief_0 \in \Delta\np{\XX}$ of the initial state of the system
for the initialization.
Second, Equation~\eqref{eq:dpomdp_gen_dynamics_def} is called the \emph{state evolution}
equation of the system. It is defined thanks to the dynamics which describe the evolution
of the states of the controlled dynamical
system.
Third, Equations~\eqref{eq:dpomdp_gen_ini_observation_def}
and~\eqref{eq:dpomdp_gen_observation_def} define the \emph{observations} of the system
available at each time step.
Fourth, Equation~\eqref{eq:dpomdp_gen_control_adm} is called the \emph{admissibility
  constraints equation}: it defines which controls can be applied at each time step.
Note that the proper formulation of the admissibility constraints should contain an added
quantification, ``$\forall \omega \in \Omega$'', which we omit in this paper as
the set~$\Omega$ is finite
and the probability~$\PP$ has full support ($\PP\np{\omega} > 0$ for all
$\omega \in \Omega$).
Equation~\eqref{eq:dpomdp_gen_nonanticipativity} is the
\emph{nonanticipativity} constraint: it defines the information available to the decision
maker before choosing a control at each time step.
As all sets $\Omega$, $\XX$, $\UU$ and $\OO$ are assumed to be finite by
Assumption~\ref{assumpt:pomdp_finite_sets}, all mappings with domain $\Omega$ are random
variables and Equation~\eqref{eq:dpomdp_gen_obj} is well defined because
$\costfunct_{\timeindex}$ and $\finalcost$ takes their values in $\RR \cup \na{+\infty}$,
hence the optimization Problem~\eqref{eq:dpomdp_general_formulation} is well defined.

\section{Complexity analysis of \DP\ for \dpomdp s}
\label{sect:dp_dpomdp_complexity}

In \S\ref{sect:dp_for_dpomdp}, we present \DP\ for \dpomdp s. Then, in
\S\ref{subsect:def_bbr_and_dp} we study its complexity, in the sense of the number of
``operations'' necessary to solve Problem~\eqref{eq:dpomdp_general_formulation}. In
\S\ref{sect:belief_dynamics_as_pushforward}, we present a new representation of
transitions for beliefs with pushforward measures, that will be used to prove the complexity
results. 

\subsection[\DP\ for \dpomdp]{\DP\ for \dpomdp}
\label{sect:dp_for_dpomdp}

We now present \DP\ Equations with beliefs for
Problem~\eqref{eq:dpomdp_general_formulation}. As a \dpomdp\ is a \pomdp, all the results
and numerical methods that apply to \pomdp s are carried over to \dpomdp s. Notably, it is
possible to write \DP\ equations for a finite horizon problem associated with a \pomdp. To
do so, it is classical to formulate a belief-\mdp\ where the state is a probability
distribution over the state space, called belief (see \citep{bertsekas_shreve} for details
on the assumptions for general \pomdp s). Here, we detail this approach for the
specific \dpomdp\ case, and we slightlly contribute by tackling cases with
explicit admissibility constraints on the controls.

First, in \S\ref{subsubsect:belief_in_dpomdp}, we formally define sets and mappings
which are necessary for the formulation of the belief-\mdp. Second, in
\S\ref{subsect:dp_for_dpomdp}, we present the \DP\ equations for the resulting
belief-\mdp.

\subsubsection{Beliefs in \dpomdp}
\label{subsubsect:belief_in_dpomdp}
First, we present the set of beliefs. 
Second, we present the mappings necessary for the formulation of the
belief-\mdp, notably the beliefs dynamics.

\subsubsubsection{Sets for the beliefs}

The dynamic programming equation for \dpomdp s is formulated using new
information states in the set~$\Delta(\XX)$ --- that is, the probability
distributions over the ``initial'' state space~$\XX$ --- which are called
beliefs.
However, the beliefs dynamics, as described later in
Equation~\eqref{eq:beliefdynamics}, may lead to a null measure over the space $\XX$ when
considering some combination of observations and controls which are in contradiction with
each other. As we want to be able to compose belief dynamics, we combine $\Delta(\XX)$ and
the null measure over $\XX$ as follows.

We introduce an extra element, denoted by~$\cemetery$ ($\cemetery \notin \XX$),
and the \emph{extended state set~$\barXX$}
\begin{equation}
  \overline{\XX} = \XX \cup \na{\cemetery}
  \eqfinv 
  \label{eq:def_bar_XX}
\end{equation}
obtained as the union of the original set~$\XX$ with~$\cemetery$.
We denote by $\BB$ the subset of $\Delta{(\barXX)}$ defined by
\begin{equation}
  \label{eq:def_BB}
  \BB = \Delta(\XX) \cup \na{\cemeteryBelief}
  \eqfinv
\end{equation}
where we identify the set $\Delta\np{\XX}$ with the set
$\nset{\mu \in \Delta\np{\barXX}}{\supp \np{\mu} \subset \XX}$ and where
$\cemeteryBelief \in \Delta\np{\barXX}$ is the discrete probability measure on $\barXX$
concentrated on $\cemetery$, that is $\cemeteryBelief(\na{\cemetery})=1$. The null measure
over $\XX$ is thus ``replaced'' by the probability $\cemeteryBelief$ over $\barXX$ whose
support is $\na{\cemetery}$. We call the probability measure~$\cemeteryBelief$ the
\emph{cemetery belief} as we will see in Equation~\eqref{eq:beliefdynamics} that the
belief dynamics, when reaching the belief state~$\cemeteryBelief$, remains in
$\cemeteryBelief$ forever. A probability measure
$\measureOnBarXX\in \Delta\np{\overline{\PRIMAL}}$ will be represented, in some equations,
by the ordered pair
$\bp{\measureOnBarXX_{\vert_{\PRIMAL}}, \measureOnBarXX \np{\cemetery}}$, where
$\measureOnBarXX_{\vert_{\PRIMAL}}$ is a nonnegative measure on the set $\XX$ and
$\measureOnBarXX \np{\cemetery} \in \RR_{+}$.

Now that the set of beliefs $\BB$ is defined, we present the beliefs dynamics.

\subsubsubsection{Beliefs dynamics}
In order to define the beliefs dynamics, we introduce, for each $\timeindex \in \timeset
\setminus \na{\horizon}$, a mapping
$\ProbaObservation_{\timeindex+1}: \BB \times \UU \times \OO \to \nc{0, 1}$ and
a function $\beliefdynamics_{\timeindex}: \BB \times \UU \times \OO \to \BB$. They are defined using
partial mappings, defined as follows.

Let $\AA$, $\DD$, $\FF$ and $\GG$ be sets. Let
$g: \AA \times \DD \to \FF, (a, d) \mapsto g(a, d)$ be a mapping.
For any fixed value~$d \in \DD$, we denote by $g^{d}$ the mapping
\begin{equation}
  g^{d}: \AA \to \FF \eqsepv a \mapsto g(a, d)
  \eqfinv
  \label{eq:fixed_second_component_g}
\end{equation}
i.e. the mapping $g \np{\cdot, d}$ obtained from $g$ by setting (``freezing'')
its second variable to the value~$d$. When considering mappings with~$n$ inputs, we extend this
notation to the last $n-1$ inputs using a Cartesian product over the last $n-1$ sets.
For example, in the case $n=3$, we consider $g: \AA \times \DD \times \FF \to
\GG$, and we denote by
$g^{\np{d, f}} = g\np{\cdot, d, f}$ the mapping
$g^{\np{d, f}}: \AA \to \GG, a \mapsto g(a, d, f)$.

The function $\ProbaObservation_{\timeindex+1}: \BB \times \UU \times \OO \to \nc{0, 1}$ gives the probability of observing
$\observer$ at time $\timeindex+1$,
when applying control~$\controls$ with knowledge of the current state given by
the belief~$\belief$ at time $\timeindex$,
and is given by
\begin{equation}
  \forall \timeindex \in \timesetNoHorizon \eqsepv
  \ProbaObservation_{\timeindex+1}:
  (\belief, \controls, \observer) \mapsto \belief \bp{ \np{
  \observerfunct_{\timeindex+1}^{\controls} \circ \dynamics_{\timeindex}^{\controls}}^{-1}
  \np{\observer}} \in  \nc{0, 1}
\eqfinv
\label{eq:def_proba_observation}
\end{equation}
where $\dynamics_{\timeindex}^{\controls}\np{\cdot}$ and
$\observerfunct_{\timeindex}^{\controls}\np{\cdot}$ are partial mappings that follow the
notation defined in Equation~\eqref{eq:fixed_second_component_g} from the mappings defined
in Equations~\eqref{eq:dpomdp_gen_dynamics_def}, \eqref{eq:dpomdp_gen_ini_observation_def}
and~\eqref{eq:dpomdp_gen_observation_def}
\begin{equation*}
  \forall \controls \in \UU,\quad
  \dynamics_{\timeindex}^{\controls}: \XX \to \XX \eqsepv
  \states \mapsto \dynamics_{\timeindex}
  \np{\states, \controls}
  \eqsepv
  \text{ and }
  \forall \controls \in \UU,\quad
  \observerfunct_{\timeindex}^{\controls}: \XX \to \OO
  \eqsepv \states \mapsto
  \observerfunct_{\timeindex}\np{\states, \controls}
  \eqfinv
\end{equation*}
and where
$\belief\bp{ \np{ \observerfunct_{\timeindex+1}^{\controls} \circ
    \dynamics_{\timeindex}^{\controls}}^{-1} \np{\observer}}$ is the probability of the
set $\np{\observerfunct_{\timeindex+1}^{\controls} \circ
  \dynamics_{\timeindex}^{\controls}}^{-1} \np{\observer}$ under the probability
distribution~$\belief$. 
Note that we always have that
\begin{equation}
  \label{eq:Qforcemetery}
  \ProbaObservation_{\timeindex+1}(\cemeteryBelief, \controls, \observer) =
  \cemeteryBelief \bp{ \np{
  \observerfunct_{\timeindex+1}^{\controls} \circ \dynamics_{\timeindex}^{\controls}}^{-1}
  \np{\observer}}
  = 0
  \eqfinv
\end{equation}
as
$\np{\observerfunct_{\timeindex+1}^{\controls} \circ
  \dynamics_{\timeindex}^{\controls}}^{-1} \np{\observer}$ is always a (possibly empty)
subset of $\XX$ and thus has a null intersection with $\na{\cemetery}$.

For all time $\timeindex \in \timesetNoHorizon$, the mapping
$\beliefdynamics_{\timeindex}: \BB \times \UU \times \OO \to \BB$ gives the
evolution of the beliefs,
when applying control~$\controls$ with knowledge of the current state given by
the belief~$\belief$
and observing~$\observer$ at time $\timeindex+1$, and is given by
\begin{subequations}
  \begin{align}
    \forall \nextstates \in \XX \eqsepv
  \beliefdynamics_{\timeindex}(\belief, \controls, \observer) \np{\nextstates}
  &=
    \begin{cases}
      \displaystyle
      \frac{\belief\bp{\np{\dynamics_{\timeindex}^{\controls}}^{-1} \np{\nextstates}}}{
      \ProbaObservation_{\timeindex+1}\np{\belief, \controls, \observer}}
      &\text{ if } \ProbaObservation_{\timeindex+1}\np{\belief, \controls, \observer}
        \neq 0,
        \text{ and }
        \nextstates \in
        \bp{\observerfunct_{\timeindex+1}^{\controls}}^{-1}(\observer) \eqfinv\\
      0
      &\text{ otherwise,}
    \end{cases}
    \label{eq:beliefdynamics_states} \\
  \beliefdynamics_{\timeindex}(\belief, \controls, \observer) \np{\cemetery}
    &= 1 - \beliefdynamics_{\timeindex}(\belief, \controls, \observer) \np{\XX}
  \eqfinp
  \label{eq:beliefdynamics_cemetery}
\end{align}
\label{eq:beliefdynamics}
\end{subequations}
Hence, $\cemeteryBelief$ is used as a last resort belief, which appears when it is not
possible to observe $\observer$ after applying control $\control$ to any state of the
support of belief $\belief$.
Thus, $\cemeteryBelief$ is used to ensure that the mappings
$\beliefdynamics_{\timeindex}$ are well defined for all beliefs, controls and
observations.


The above tools make it possible to express a \DP\ algorithm to solve a \dpomdp\
optimization problem given by Problem~\eqref{eq:dpomdp_general_formulation}, as we can use
them to properly define a belief-\mdp\ which is amenable to \DP.

\subsubsection[\DP\ equations for \dpomdp]{\DP\ equations for \dpomdp}
\label{subsect:dp_for_dpomdp}
In the case of \pomdp\ (without constraints on the controls), \DP~equations with
beliefs as new states were first given in \citep{astrom_optimal_1965}.
More general cases (still without explicit constraints on the controls) are treated
in~\cite[Chapter 10]{bertsekas_shreve} and in~\cite[Chapter 4]{Bertsekas:2000}.
\DP\ Equations for \dpomdp\ can be obtained as a special case of \DP\ for \pomdp.
In the case where there are no constraints on the controls, they are
given
in~\citep{littman_thesis} using the expression of the beliefs dynamics
$\na{\beliefdynamics_{\timeindex}}_{\timeindex \in \timeset \setminus \na{\horizon}}$
presented in Equation~\eqref{eq:beliefdynamics}.
In \citep{bertsekas_shreve} the proof that beliefs are \emph{statistics sufficient for
  controls} was made for \pomdp s without any admissibility constraint. We thus cannot
directly apply this result on Problem~\eqref{eq:dpomdp_general_formulation}, as it
contains Constraint~\eqref{eq:dpomdp_gen_control_adm}.
We extend the classical results by \citep{bertsekas_shreve} in
Proposition~\ref{prop:dynamics_belief_and_bellman} in order to tackle such constraints. We
identify an admissibility set for beliefs of the form
$\beliefadmcontrolset\np{\belief} = \bigcap_{\states \in \supp\np{\belief}}
\admcontrolset(\states)$. Note that we use an upper index b to distinguish admissibility
sets for beliefs from admissibility sets for states.
Also note that, as far as we know, the first \DP\ equations using
such sets $\beliefadmcontrolset\np{\belief}$ were given in \cite[\S5]{bonet_pomdp_1998}
with no explicit proof.

\begin{proposition}
  \label{prop:dynamics_belief_and_bellman}
  Consider a \dpomdp\ optimization problem given by
  Problem~\eqref{eq:dpomdp_general_formulation} which satisfies the finite sets
  Assumption~\ref{assumpt:pomdp_finite_sets}.
  Let
  $\BB = \Delta(\XX) \cup \left\{ \cemeteryBelief \right\}$, as defined in
  Equation~\eqref{eq:def_BB} and consider the sequence of value functions
  $\nseqp{\valuefct_{\timeindex}: \BB \to \RR \cup \na{+\infty}}{\timeindex\in \timeset}$
  defined by the following backward induction. First, for all $\timeindex \in \timeset$,
  we have that $\valuefct_{\timeindex}(\cemeteryBelief)=0$. Second, we have that
  \begin{subequations}
    \begin{align}
      \valuefct_{\horizon}: \belief \in \Delta(\XX) \mapsto
      & \sum_{\states \in \XX} \belief(\states)
        \finalcost(\states)
        \eqfinv
        \label{eq:Bellman_DPOMDP_general_final}
      \\
      \valuefct_{\timeindex}:
      \belief\in \Delta(\XX) \mapsto
      & \min_{\controls \in
        \beliefadmcontrolset_{\timeindex}(\belief)
        }
        \Bp{
        \sum_{\states \in \XX} \belief(\states)
      \costfunct_{\timeindex}( \states, \controls)
        + \sum_{\observer \in \OO}
        \ProbaObservation_{\timeindex+1}\np{\belief, \controls, \observer}
        \valuefct_{\timeindex+1}
        \bp{
        \beliefdynamics_{\timeindex}\np{ \belief, \controls, \observer}
        }
        }
        \eqfinv
        \label{eq:Bellman_DPOMDP_general_bis}
      \intertext{the set~$\beliefadmcontrolset_{\timeindex}(\belief)$ being defined as}
      &
      \beliefadmcontrolset_{\timeindex}(\belief) = \bigcap_{\states \in
        \supp\np{\belief}} \admcontrolset_{\timeindex}(\states) \nonumber \eqfinp
    \end{align}
    \label{eq:bellman_dpomdp}
  \end{subequations}

  Then, the optimal value of Problem~\eqref{eq:dpomdp_general_formulation} and the value
  of the function $\valuefct_0$ at the initial belief $\belief_0$ are
  equal, that is, $\valuefct_0(\knownproba_0) =
  \probvalue^{\star}(\knownproba_0)$. 
  Moreover, a policy $\policy =(\policy_0, \dots ,
  \policy_{\horizon-1})$, defined by a sequence of measurable mappings
  $\policy_{\timeindex}:\BB \to
  \UU$, which minimizes the right-hand side of
  Equation~\eqref{eq:Bellman_DPOMDP_general_bis} for each $\belief$ and
  $\timeindex$ is an optimal policy of Problem~\eqref{eq:dpomdp_general_formulation}: the
  controls given by $\controls_{\timeindex} =
  \policy_{\timeindex}(\belief_{\timeindex})$ (where
  $\belief_{\timeindex}$ is computed thanks to the recursion $\belief_{\timeindex+1} =
  \beliefdynamics_{\timeindex}(\belief_{\timeindex}, \controls_{\timeindex},
  \observer_{\timeindex+1})$, with $\belief_{0} =
  \knownproba_{0}$) are optimal controls of Problem~\eqref{eq:dpomdp_general_formulation}.

\end{proposition}

\begin{proof}
  We present a sketch of proof of Proposition~\ref{prop:dynamics_belief_and_bellman}.
  \begin{enumerate}
  \item We rewrite Problem~\eqref{eq:dpomdp_general_formulation} as an equivalent
  problem, without Constraint~\eqref{eq:dpomdp_gen_control_adm} by adding indicator
  functions of the constraints to the instantaneous costs. The equivalent problem then
  follows the framework of~\citep{bertsekas_shreve}.
\item 
  We can apply the results of~\citep{bertsekas_shreve} to the reformulated problem
  and obtain
  associated \DP\ equations.
  \item The \DP\ equations which solve the equivalent problem are equivalent to
  Equations~\eqref{eq:bellman_dpomdp} presented in
  Proposition~\ref{prop:dynamics_belief_and_bellman}, thus concluding that
  Equation~\eqref{eq:bellman_dpomdp} gives the solution of
  Problem~\eqref{eq:dpomdp_general_formulation} as formulated in
  Proposition~\ref{prop:dynamics_belief_and_bellman}. This step is a bit technical, but is
  otherwise straightforward and does not present any major difficulty.
  \end{enumerate}
  The detailed proof can be found in the PhD thesis
  \citep[Chapter 5, \S A.3, p.120-125]{vessaire_thesis}.
\end{proof}

Now that we have presented \DP\ equations on beliefs, we present the complexity of \DP.

\subsection{\DP\ complexity for \dpomdp s}
\label{subsect:def_bbr_and_dp}
According to Proposition~\ref{prop:dynamics_belief_and_bellman}, we can solve
Problem~\eqref{eq:dpomdp_general_formulation} by computing $\valuefct_{0}\np{\belief_0}$
by means of Equations~\eqref{eq:bellman_dpomdp}. Solving \DP\
equations~\eqref{eq:bellman_dpomdp} implies that we are able to numerically evaluate at
each time $\timeindex \in \timeset$ the value functions
$\valuefct_{\timeindex}$, not necessarily for every belief but, at least, for each
\emph{reachable} belief starting from~$\belief_0$. Thus, we introduce the subsets of reachable
beliefs starting from $\belief_0$.
We start by formally defining the set of reachable beliefs, before we present our first
complexity result on \DP\ for \dpomdp.

The set of reachable beliefs $\BBR_{\ic{1,\horizon}}$
is defined as follows. Note that we use the upper
index $\dataTuple$ to recall that we consider the set of reachable beliefs of a \dpomdp\
defined by the data tuple $\dataTuple$, in Equation~\eqref{eq:def_data_tuple_dpomdp},

\begin{definition}
  \label{def:reachable_space}
  Let $\belief_0 \in \Delta\np{\XX}$ be given and consider the sequence
  $\nseqa{\BBR_{\timeindex}}{\timeindex \in \timeset}$ of subsets of the set of beliefs
  $\BB = \Delta(\XX) \cup \left\{ \cemeteryBelief \right\}$ defined by the induction
  \begin{align}
    \BBR_{0}(\belief_0)
    = \na{\belief_0}
    \quad \text{and}\quad\forall \timeindex \in \timesetNoHorizon \eqsepv
    \BBR_{\timeindex+1}(\belief_0)
    = \beliefdynamics_{\timeindex}\bp{
    \BBR_{\timeindex}(\belief_0), \UU, \OO}
    \eqsepv
    \label{eq:def_BBR_t}
  \end{align}
  where $\beliefdynamics_{\timeindex}$ is defined in
  Equation~\eqref{eq:beliefdynamics}. For any $t\in \timeset$, the subset
  $\BBR_{\timeindex}(\belief_0) \subset \beliefspace$ is called the set of \emph{reachable
    beliefs at time $t$} starting from initial belief $\belief_0$.

  Moreover, we denote by $\BBR_{\ic{\timeindex,\timeindex'}} \np{\belief_0}$ the union, for
  $\timeindex''$ in the time interval $\ic{\timeindex,\timeindex'}$,
  $\timeindex < \timeindex'$, of the reachable beliefs at time $\timeindex''$ starting
  from the initial belief $\belief_0 \in \Delta\np{\XX}$, that is,
  \begin{equation}
    \label{eq:def_BBR_t_tprime}
    \forall \np{\timeindex, \timeindex'} \in \timeset^2
    \eqsepv
    \timeindex < \timeindex'
    \eqsepv
    \BBR_{\ic{\timeindex, \timeindex'}} \np{\belief_0} =
    \bigcup_{\timeindex'' = \timeindex}^{\timeindex'} \BBR_{\timeindex''} \np{\belief_0}
    \eqfinp
  \end{equation}
  The set $\BBR_{\ic{1,\horizon}}\np{\belief_0}$ is called the
  \emph{set of reachable beliefs} from the initial belief~$\belief_0$.

\end{definition}

Note that, under Assumption~\ref{assumpt:pomdp_finite_sets}, the set
$\BBR_{\ic{1,\horizon}}\np{\belief_0}$ is finite.

We now present a classical complexity result for \DP\ algorithm (which we call \DPA\ in
the rest of this paper).

\begin{proposition}
  \label{prop:alg_DP_belief_solves_dpomdp}
  Consider a \dpomdp\ optimization problem given by
  Problem~\eqref{eq:dpomdp_general_formulation} which satisfies the finite sets
  Assumption~\ref{assumpt:pomdp_finite_sets}.
  Let $\belief_0 \in \Delta\np{\XX}$.
 Then, the \DPA\ recalled in Proposition~\ref{prop:alg_DP_belief_solves_dpomdp}
solves (numerically) 
  Problem~\eqref{eq:dpomdp_general_formulation} with complexity 
  $O\np{\cardinal{\timeset} \cardinal{\BBR_{\ic{1,\horizon}}\np{\belief_0}} \cardinal{\UU}
    \cardinal{\OO}}$, where the set of reachable beliefs
  $\BBR_{\ic{1,\horizon}}\np{\belief_0}$ is defined in
  Equation~\eqref{eq:def_BBR_t_tprime}.
\end{proposition}

\begin{proof}
  First, as we consider that Assumption~\ref{assumpt:pomdp_finite_sets} holds, note that
  $\BBR_{\ic{1,\horizon}}\np{\belief_0}$ is finite and we can apply
  Proposition~\ref{prop:dynamics_belief_and_bellman} on
  Problem~\eqref{eq:dpomdp_general_formulation}. We hence solve
  Problem~\eqref{eq:dpomdp_general_formulation} by computing value functions given by
  Equations~\eqref{eq:bellman_dpomdp}.

  For a given time $\timeindex \in \timesetNoHorizon$ and reachable belief
  $\belief \in \BBR_{\timeindex} \np{\belief_0}$, we compute the value function
  $\valuefct_{\timeindex}$ by evaluating the next value for each control
  $\controls \in \UU$ and each resulting observations. We hence need
  $\sum_{\timeindex \in \timeset} \cardinal {\BBR_{\timeindex} \np{\belief_0}}
  \cardinal{\UU} \cardinal{\OO}$ operations to solve
  Problem~\eqref{eq:dpomdp_general_formulation}. Then, since for all time
  $\timeindex \in \timeset\backslash\na{0}$,
  $\BBR_{\timeindex} \np{\belief_0} \subset \BBR_{\ic{1,\horizon}}\np{\belief_0}$ (see
  Equation~\eqref{eq:def_BBR_t_tprime}), we have for all time
  $\timeindex \in \timeset\backslash\na{0}$,
  $\cardinal {\BBR_{\timeindex} \np{\belief_0}} \leq
  \cardinal{\BBR_{\ic{1,\horizon}}\np{\belief_0}}$ . It remains to consider the case
  $t=0$. We have that $\BBR_{0}\np{\belief_0} = \na{\belief_0}$ and thus
  $\cardinal{\BBR_{0}\np{\belief_0}}=1$. Moreover,
  $\BBR_{\ic{1,\horizon}} \np{\belief_0}\neq \emptyset$ since there is always at least one
  belief in $\BBR_{1}\np{\belief_0}$, as for a given control $\controls \in \UU$ and an
  observation $\observer \in \OO$,
  $\beliefdynamics_{0}\np{\belief_0, \controls, \observer} \in \BBR_{1}\np{\belief_0}
  \subset \BBR_{\ic{1,\horizon}}\np{\belief_0}$. Hence
  $\cardinal{\BBR_{0}\np{\belief_0}}=1 \leq
  \cardinal{\BBR_{\ic{1,\horizon}}\np{\belief_0}}$.
  
  We have obtained that
  $\sum_{\timeindex \in \timeset} \cardinal {\BBR_{\timeindex} \np{\belief_0}}
  \cardinal{\UU} \cardinal{\OO} \leq \cardinal{\timeset}
  \cardinal{\BBR_{\ic{1,\horizon}}\np{\belief_0}} \cardinal{\UU} \cardinal{\OO}$, and thus
  we can solve Problem~\eqref{eq:dpomdp_general_formulation} in
  $O\np{\cardinal{\timeset} \cardinal{\BBR_{\ic{1,\horizon}}\np{\belief_0}} \cardinal{\UU}
    \cardinal{\OO}}$ operations.
\end{proof}

In order to apply Proposition~\ref{prop:alg_DP_belief_solves_dpomdp} on
Problem~\eqref{eq:dpomdp_general_formulation} and to get complexity bounds on the
\DPA, we now study the set of reachable beliefs
$\BBR_{\ic{1,\horizon}}\np{\belief_0}$, more specifically, we give bounds on its
cardinality.

\begin{theorem}
  \label{th:reachable_belief_bounds}
  Consider a \dpomdp\ optimization problem given by
  Problem~\eqref{eq:dpomdp_general_formulation} which satisfies the finite sets
  Assumption~\ref{assumpt:pomdp_finite_sets}, and such that $\cardinal{\UU}>1$.
  For all initial belief $\belief_0 \in \Delta(\XX)$, the cardinality of the set of
  reachable beliefs starting from $\belief_0$, defined in
  Equation~\eqref{eq:def_BBR_t_tprime}, satisfies the following bound
  \begin{equation}
    \label{eq:bounds_dpomdp}
    \bcardinal{\BBR_{\ic{1,\horizon}}(\belief_0)} \leq
    \min \left( \left( 1 + \cardinal{\XX} \right)^{\cardinal{\supp \np{\belief_0}}},
      1+ \cardinal{\supp\np{\belief_0}} \cardinal{\UU}^{\cardinal{\timeset}} \right) \eqfinp
  \end{equation}
\end{theorem}

\begin{proof}
  A sketch of proof is postponed to \S\ref{sect:belief_dynamics_as_pushforward}, as it
  relies on a new representation of the belief dynamics presented in
  \S\ref{sect:belief_dynamics_as_pushforward}. The complete proof can be found in Appendix
  \S\ref{subsect:bounds_cardinal_sets_pushfowards}.
\end{proof}

The bound on the cardinality of the
set~$\BBR_{\ic{1,\timeindex}}\np{\belief_0}$,
obtained in Theorem~\ref{th:reachable_belief_bounds},
improves on a previous result, that we now recall.
  Littman presents in~\cite[Lemma 6.1]{littman_thesis} a bound on the set of reachable
  beliefs starting from belief $\belief_0 \in \Delta(\XX)$:
  \begin{equation}
    \forall \timeindex \in \timeset
    \eqsepv
    \bcardinal{\BBR_{\ic{0,\timeindex}}(\belief_0)}
    \leq \left( 1 + \cardinal{\XX} \right)^{\cardinal{\XX}} \eqfinp
    \label{eq:littman-bound}
  \end{equation}
Equation~\eqref{eq:bounds_dpomdp} is an improvement on the
bound given in Equation~\eqref{eq:littman-bound} which takes into account the support of
the initial belief $\belief_0$: indeed, as $\belief_0 \in \Delta\np{\XX}$ and
$\cardinal{\supp \np{\belief_0}} \le \cardinal{\XX}$,
Equation~\eqref{eq:bounds_dpomdp} is tighter than
Equation~\eqref{eq:littman-bound}.

Using Equation~\eqref{eq:bounds_dpomdp}, we obtain that the number of reachable beliefs of
a \dpomdp\ is finite even when considering the case of an infinite horizon. Indeed, the
first inequality in Equation~\eqref{eq:bounds_dpomdp} is well defined even in the infinite
horizon case.

A direct consequence of Proposition~\ref{prop:alg_DP_belief_solves_dpomdp} and
Theorem~\ref{th:reachable_belief_bounds} is that the complexity of the \DPA\ is
$O \bp{\cardinal{\BBR_{\ic{1,\horizon}} \np{\belief_0}} \cardinal{\timeset} \cardinal{\UU}
  \cardinal{\OO}}$, i.e. in
$O \left( \min \left( \left( 1 + \cardinal{\XX} \right)^{\cardinal{\supp \np{\belief_0}}},
    1 + \cardinal{\supp\np{\belief_0}} \cardinal{\UU}^{\cardinal{\timeset}} \right)
  \cardinal{\timeset} \cardinal{\UU} \cardinal{\OO} \right)$.

\subsection{Belief dynamics as pushforward measures}
\label{sect:belief_dynamics_as_pushforward}

Here, we expose another representation of the beliefs evolution mappings
$\nseqa{\beliefdynamics_{\timeindex}}{\timeindex \in \timeset\setminus\na{\horizon}}$
defined in Equation~\eqref{eq:beliefdynamics}, used in the proof of
Theorem~\ref{th:reachable_belief_bounds}. First, we recall the notion of \emph{pushforward
  measures} when considering finite sets. Second, we introduce the mappings necessary for
the new representation. We then present in Lemma~\ref{lem:tau_as_pushforward} the
representation of the belief dynamics as pushforward measures.

\begin{definition}
  \label{def:pushforward}
  Consider two finite sets $\AA$ and $\DD$ and a mapping $h: \AA \to \DD$. The
  \emph{pushforward measure} (or the \emph{image-measure}) of a probability measure
  $\mu \in \Delta\np{\AA}$ on the set~$\AA$ by the mapping $h$ is the probability measure
  $h_{\star}\mu \in \Delta\np{\DD}$ on the set $\DD$ defined by
  \begin{equation}
    \label{eq:def_push_forward}
    \np{h_{\star}\mu} (d) =
    \mu\bp{ h^{-1}(d)} = \sum_{a \in \AA, h(a)=d}  \mu \np{a}
    \eqsepv    \forall d \in \DD 
    \eqfinp
  \end{equation}
  We also denote by $h_\star$ the mapping from $\Delta\np{\AA}$ to $\Delta\np{\DD}$ such
  that $h_{\star}\np{\mu} = h_\star \mu$.
\end{definition}

Before presenting Lemma~\ref{lem:tau_as_pushforward}, we introduce the two mappings
$\pushforwardtransition^{\controls, \observer}_{\timeindex}$ and $\Renormalization$.
For each ordered pair
$\np{\controls, \observer} \in \UU \times \OO$, and each
$\timeindex \in \timeset \setminus \na{\horizon}$, we denote by
$\pushforwardtransition^{\controls, \observer}_{\timeindex}$ the self-mapping on the
extended state set $\overline{\XX}= \XX \cup \na{\cemetery}$ in~\eqref{eq:def_bar_XX}), defined by:
\begin{equation}
  \label{eq:def_Cal_T_dpomdp}
  \pushforwardtransition^{\controls,\observer}_{\timeindex}:  \overline{\XX}
  \to \overline{\XX}
  \eqsepv
  \overline{\states}
  \mapsto
  \begin{cases}
    \dynamics_{\timeindex}^\controls \np{\overline{\states}}
    &\text{if}\quad \overline{\states} \neq \cemetery \text{ and }
      \dynamics_{\timeindex}^\controls \np{\overline{\states}} \in
    \bp{\observerfunct_{\timeindex+1}^{\controls}}^{-1}\np{\observer} \eqfinv
    \\
    \cemetery &\text{otherwise.}
  \end{cases}
\end{equation}
The mapping $\pushforwardtransition^{\controls,\observer}_{\timeindex}$ hence applies the
dynamics $\dynamics_{\timeindex}$, as defined in
Problem~\eqref{eq:dpomdp_general_formulation}, given control $\controls$, and only keeps
the resulting state if it is consistent with observation $\observer$.
Meanwhile, the \emph{renormalization mapping}
$\Renormalization:\Delta\np{\overline{\PRIMAL}} \to \Delta\np{\overline{\PRIMAL}}$ is defined by
\begin{align}
  \Renormalization:
  \measureOnBarXX \in \Delta\np{\overline{\PRIMAL}}
  &\mapsto
    \begin{cases}
      \bp{\frac{1}{\measureOnBarXX\np{\PRIMAL}}\measureOnBarXX_{\vert_{\PRIMAL}},0}
      &\text{if } \measureOnBarXX\np{\PRIMAL}\not=0
      \eqsepv
      \\
      \delta_{\cemetery}
      &\text{if } \measureOnBarXX\np{\PRIMAL}=0
      \eqfinp
    \end{cases}
    \label{eq:definition_renormalisation}
\end{align}

We now express the belief dynamics as pushforward measures.

\begin{lemma}
  \label{lem:tau_as_pushforward}
  Let $\np{\controls, \observer} \in \UU \times \OO$ be given, and let
  $\timeindex \in \timeset \setminus \na{\horizon}$. The beliefs evolution mapping at
  time $\timeindex$, $\beliefdynamics_{\timeindex}$, defined in
  Equation~\eqref{eq:beliefdynamics} satisfies
  \begin{equation}
    \label{eq:equality_dynamics_belief_pushforward}
    \beliefdynamics_{\timeindex}\np{\belief, \controls, \observer} = \Renormalization
    \circ \np{\pushforwardtransition^{\controls,\observer}_{\timeindex}}_{\star} \np{\belief}
    \eqsepv
    \forall \belief \in \BB
    \eqfinv
  \end{equation}
  where
  $\np{\pushforwardtransition^{\controls,\observer}_{\timeindex}}_{\star} \np{\belief}$ is
  the pushforward measure of belief $\belief$ by
  $\pushforwardtransition^{\controls,\observer}_{\timeindex}$, as defined in~\eqref{eq:def_push_forward}.
\end{lemma}

\begin{proof}
  The proof is detailed in Appendix~\ref{sect:technical_lemmatas_pushforward}
  (page~\pageref{par:lemmata}).
\end{proof}
The meaning of Lemma~\ref{lem:tau_as_pushforward} is illustrated in
Figure~\ref{fig:illustration_belief_dynamics_as_pushforward}.
This new representation is of interest as, for all time $\timeindex \in \timesetNoHorizon$,
the composition of belief dynamics $\beliefdynamics_{\timeindex}$ is given by the
pushforward measure of the composition of mappings
$\pushforwardtransition^{\controls, \observer}_{\timeindex}$ for the relevant ordered pairs
$\np{\controls, \observer} \in \UU \times \OO$. Indeed, when considering a composition of
belief dynamics, we can factorize the renormalization mapping $\Renormalization$. We thus
apply the renormalization mapping~$\Renormalization$ to the composition of the pushforward
measures, which is the pushforward measure of the composition of mappings
$\pushforwardtransition^{\controls, \observer}_{\timeindex}$.
There is therefore an equivalence between studying the composition for time
$\timeindex \in \timesetNoHorizon$ of the belief dynamics $\beliefdynamics_{\timeindex}$
and the composition, for the relevant ordered pairs
$\np{\controls, \observer} \in \UU \times \OO$, of the mappings
$\pushforwardtransition^{\controls, \observer}_{\timeindex}$. Notably, we use this
representation to bound the cardinality of the set of reachable beliefs
(see Definition~\ref{def:reachable_space}), and thus study
the complexity of \DP\ for \dpomdp.
To do so, we introduce notations for sets and mappings.

\subsubsubsection{Notation for sets and mappings}
For any given sets $\DUAL$ and $\VV$, we denote by
$\Mappings{\YY}{\VV} = \VV^{\YY}$ the set of mappings from $\YY$ to $\VV$.
\begin{subequations}
  \begin{itemize}
  \item For all $\generalFunctionSet \subset \Mappings{\YY}{\VV}$, and
    $\Dual \subset \DUAL$, $B \subset \Delta\np{\DUAL}$, we
    introduce the notations
    $\generalFunctionSet(\Dual)$ and  $\generalFunctionSet_{\star}(B)$
    for the sets respectively defined by
    \begin{align}
      \generalFunctionSet(\Dual)
      = \bset{ \generalFunction(\dual)}{ \dual \in \Dual \text{ and }
      \generalFunction\in \generalFunctionSet}
      \subset {\VV}
      \eqsepv
      \generalFunctionSet_{\star}(B)
      = \bset{ \generalFunction_{\star} b}{ b \in B \text{ and }
      \generalFunction\in \generalFunctionSet}
      \subset \Delta\np{\VV}
      \eqsepv
      \label{eq:set-pushforward}
    \end{align}
    with the simplified notations $\generalFunctionSet({\dual})=\generalFunctionSet(\na{\dual})$
    and $\generalFunctionSet_{\star}(b) =\generalFunctionSet_{\star}(\na{b})$,
    for $\dual \in \Dual$ and $b \in \Delta\np{\DUAL}$.
    
  \item Given two subsets $\generalFunctionSet'$ and $\generalFunctionSet''$ of
    $\Mappings{\YY}{\YY}$ we introduce the subset
    $\generalFunctionSet'\circ \generalFunctionSet''$
    defined by
    \begin{equation}
      \generalFunctionSet'\circ \generalFunctionSet'' =
      \bset{ \generalFunction'\circ\generalFunction''}{ \generalFunction'\in
        \generalFunctionSet' \text{ and } \generalFunction''\in \generalFunctionSet''}
      \subset \Mappings{\YY}{\YY}
      \eqfinp
      \label{eq:set-mappings-composition}
    \end{equation}
  \item For any sequence $\nseqa{\generalFunctionSet_{\kindex}}{\kindex \in \NN}$, with
    $\generalFunctionSet_{\kindex} \subset \Mappings{\YY}{\YY}$ for all
    $\kindex \in \NN$, we introduce for any $\kindex \in \NN$ the subsets
    $\generalFunctionSet_{0:k}$ defined by
    \begin{equation}
      \forall k\in \NN \eqsepv
      \generalFunctionSet_{0:k} = \generalFunctionSet_k\circ \generalFunctionSet_{k-1}
      \circ \cdots \circ  \generalFunctionSet_0  \subset \Mappings{\YY}{\YY}
      \eqfinp
      \label{eq:F-0k}
    \end{equation}
  \end{itemize}
  \label{eq:set-notations}
\end{subequations}

For a fixed value of $\controls \in \UU$, and $\observer \in \OO$, for all $\timeindex
\in \timesetNoHorizon$,
we have obtained in Lemma~\ref{lem:tau_as_pushforward} that
$\beliefdynamics_{\timeindex}\np{\cdot, \controls, \observer} = \Renormalization
\circ \np{\pushforwardtransition^{\controls,\observer}_{\timeindex}}_{\star}$.
Now, we introduce the sets
\begin{subequations}
  \begin{align}
    \setsBeliefDynamics_{\timeindex}
    &= \bset{ \beliefdynamics_{\timeindex}\np{\cdot, \controls, \observer}}{ \controls \in
      \UU, \observer \in \OO} \subset \Mappings{\BB}{\BB}
      \eqsepv
      \forall \timeindex \in \timesetNoHorizon,
      \label{eq:def_setsBeliefDynamics}
    \\
    \setsBeliefDynamics
    &=
    \bigcup_{\timeindex \in \timesetNoHorizon} \setsBeliefDynamics_{0:\timeindex}
        \eqfinv
       \label{eq:def_setsBeliefDynamics_total}
    \\
    \setsPushForward_{\timeindex}
    &=\bset{  \pushforwardtransition^{\controls,\observer}_{\timeindex}}{ \controls \in
      \UU, \observer \in \OO} \subset \Mappings{\barXX}{\barXX}
      \eqsepv
      \forall \timeindex \in \timesetNoHorizon,
      \label{eq:def_setsPushForward}
    \\
    \setsPushForward
    &=
 \bigcup_{\timeindex \in \timesetNoHorizon}
      \setsPushForward_{0:\timeindex}
      \eqfinv
      \label{eq:def_setsPushforward_total}
  \end{align}
\end{subequations}
where the composition of sets of mappings used in Equations~\eqref{eq:def_setsBeliefDynamics_total}--\eqref{eq:def_setsPushforward_total} is defined in
Equations~\eqref{eq:set-mappings-composition}--\eqref{eq:F-0k}.
Moreover, we
call $\setsPushForward$, defined by Equation~\eqref{eq:def_setsPushforward_total}, the
\emph{set of pushforwards of the \dpomdp} defined by
Problem~\eqref{eq:dpomdp_general_formulation}.

\begin{lemma}
  \label{lem:equ_BBR_FF}
  Let $\belief_0 \in \Delta\np{\XX}$.
  The set $\BBR_{\ic{1,\horizon}}\np{\belief_0}$ of reachable
  beliefs from the initial belief $\belief_0$, as defined in Equation~\eqref{eq:def_BBR_t_tprime}, satisfies
 \begin{equation}
   \BBR_{\ic{1, \horizon}} \np{\belief_0} =
   \setsBeliefDynamics\np{\belief_0} =
   \Renormalization \circ \np{\setsPushForward}_{\star} \np{\belief_0}
   \eqfinv
   \label{eq:equ_BBR_FF}
  \end{equation}
  where the two sets of mappings $\setsBeliefDynamics$ and $\setsPushForward$ are defined
  in Equations~\eqref{eq:def_setsBeliefDynamics_total}--\eqref{eq:def_setsPushforward_total}.
\end{lemma}

\begin{proof}
  The proof is detailed in Appendix~\ref{sect:technical_lemmatas_pushforward}
  (page~\pageref{par:lemmata}).
\end{proof}

Lemma~\ref{lem:equ_BBR_FF} is illustrated in
Figure~\ref{fig:illustration_composition_belief_dynamics_as_pushforward}.
A direct application of Lemma~\ref{lem:equ_BBR_FF} is that there is an equivalence between
studying the cardinality of $\BBR_{\ic{1, \horizon}} \np{\belief_0}$ and studying the
cardinality of $\np{\setsPushForward}_{\star} \np{\belief_0}$.

\begin{figure}[htbp]
\begin{minipage}{0.45\textwidth}
  \begin{center}
    \includegraphics[scale=0.9]{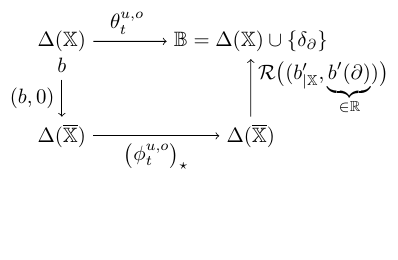}
  \caption{Illustration of the beliefs dynamics as pushforward measures}
\label{fig:illustration_belief_dynamics_as_pushforward}
\end{center}
\end{minipage}
\hspace{0.5cm}
\begin{minipage}{0.45\textwidth}
  \begin{center}
    \includegraphics[scale=0.9]{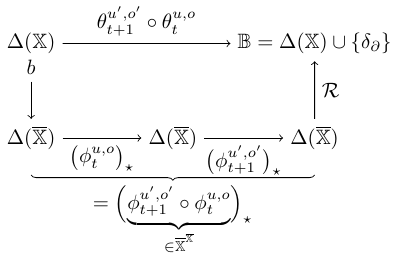}
  \caption{Illustration of the composition of beliefs dynamics as pushforward measures}
  \label{fig:illustration_composition_belief_dynamics_as_pushforward}
\end{center}
\end{minipage}
\end{figure}

We now present the postponed sketch of proof of
Theorem~\ref{th:reachable_belief_bounds}. A detailed proof can be found
in Appendix~\ref{subsect:bounds_cardinal_sets_pushfowards}.

\begin{proof}[Sketch of proof of Theorem~\ref{th:reachable_belief_bounds}]
  Let $\belief_0 \in \Delta\np{\XX}$ be given.

\noindent
$\bullet$  By Lemma~\ref{lem:equ_BBR_FF}, we have that
$\BBR_{\ic{1,\horizon}}(\belief_0) =
 \Renormalization \circ \np{\setsPushForward}_{\star} \np{\belief_0}$.

\noindent
$\bullet$  The first inequality
  $\cardinal{\BBR_{\ic{1,\horizon}}(\belief_0)} \leq \left( 1 + \cardinal{\XX}
  \right)^{\cardinal{\supp \np{\belief_0}}}$ comes from the fact that
  $\bcardinal{\np{\setsPushForward}_{\star} \np{\belief_0}}$ is bounded
  by the number of mappings from $\supp \np{\belief_0}$ to $\barXX$,
  as shown in Lemma~\ref{lem:restriction_counting_pushforward}.

\noindent
$\bullet$  Meanwhile, the second inequality
  $\bcardinal{\BBR_{\ic{1,\horizon}}(\belief_0)} \leq 1 +
  \cardinal{\supp\np{\belief_0}} \cardinal{\UU}^{\cardinal{\timeset}}$
  comes from the fact that, for any time and control
  $\np{\timeindex, \controls} \in \timesetNoHorizon \times \UU$,
  and for any belief $\belief \in \Delta\np{\XX}$, we have that
  $\sum_{\observer \in \OO}
  \bcardinal{\supp\bp{\np{\pushforwardtransition_{\timeindex}^{\controls,
          \observer}}_{\star}\belief}} \le \bcardinal{\supp\bp{\belief}}$
  by Lemma~\ref{lem:hY}. Therefore, for a given sequence of controls
  $\controls_{0:\timeindex} \in \UU^{\timeindex+1}$, there can be at most
  $\cardinal{\supp\np{\belief_0}}$ resulting beliefs (see Lemma~\ref{lem:second-bound}).
  As there are at most $\cardinal{\UU}^{\cardinal{\timeset}}$ such sequences
  $\controls_{0:\timeindex}$, $\timeindex \in \timesetNoHorizon$,
  this leads to
  $\bcardinal{\BBR_{\ic{1,\horizon}}(\belief_0)} \leq 1+ \cardinal{\supp\np{\belief_0}}
  \cardinal{\UU}^{\cardinal{\timeset}}$.
\end{proof}

We now present the subclass of \fullmdpomdp es (\Mdpomdp). 

\section{\Mdpomdp\ and complexity of \DP}
\label{sect:mdpomdp}

In this section, we introduce a subclass of \dpomdp s, \Mdpomdp s. First, we define this
subclass in \S\ref{subsect:def_mdpomdp}. Second,
in \S\ref{subsect:bounds_BBR_mdpomdp}, we present an improved bound on the cardinality of
the set of reachable beliefs for \Mdpomdp s compared to \dpomdp s. Third, in
\S\ref{subsect:reaching_pomdp_bound}, we show that the improved bound is tight.

\subsection{Definition of \MdpomdpFunctionSet\ and \Mdpomdp}
\label{subsect:def_mdpomdp}

Let us first define \separatedMappingSet s.

\begin{definition}
  \label{def:separatedMappingSet}
  Let $\YY_1$ and $\YY_2$ be two sets. A set
  $\generalFunctionSet \subset \Mappings{\YY_1}{\YY_2}$ of mappings from $\YY_1$ to
  $\YY_2$ is called a \emph{\separatedMappingSet} if
  \[
    \forall \np{\generalFunction_1, \generalFunction_2} \in
    \generalFunctionSet \times \generalFunctionSet \eqsepv
    \forall y \in \YY_1 \eqsepv
    \Bp{
    \generalFunction_1(y) = \generalFunction_2(y) \; \Longrightarrow \;
    \generalFunction_1 = \generalFunction_2}
    \eqfinp
  \]
\end{definition}
A \separatedMappingSet\ $\generalFunctionSet \subset \Mappings{\YY_1}{\YY_2}$ is hence a
set of mappings where all ordered pairs of mappings are either different everywhere, or equal
everywhere.
Otherwise stated, all the \emph{evaluation mappings} on set $\generalFunctionSet$ (i.e. the
mappings $\generalFunctionSet \to \YY_2, \generalFunction \mapsto \generalFunction\np{y}$,
for a fixed $y \in \YY_1$) are injective for all $y \in \YY_1$. For example, let
$\YY_1 = \ic{1, n}$ and
$\YY_2 = \RR$. Then, $\generalFunctionSet \subset \RR^{\YY_1}$ is identified with
$G \subset \RR^n$, and $\generalFunctionSet$ is a \separatedMappingSet\ if and only if the
projections of $G$ along each axis are injective.

In the special case where $\YY_1 = \YY_2 = \barXX$, with the extended set
$\barXX = \XX \cup \na{\cemetery}$ defined in Equation~\eqref{eq:def_bar_XX}, we want to
extend the above notion of \separatedMappingSet\ to tackle the added point $\cemetery$
in a specific way.
We thus introduce the notion of \cemeterySeparation\ for an ordered pair of self-mappings on the
set $\barXX$ and the notion of \MdpomdpFunctionSet.

\begin{definition}
  \label{def:cemeterySeparation_mapping_pair}
  An ordered pair
  $\np{\generalFunction_1, \generalFunction_2} \in \Mappings{\barXX}{\barXX}$ of
  self-mappings on the set $\barXX$ is \emph{\cemeterySeparated} if the restriction of the ordered pair
  $\np{\generalFunction_{1}, \generalFunction_{2}}$ to the set
  $\generalFunction_1^{-1}(\XX) \cap \np{\generalFunction_2}^{-1}(\XX)$ is separated.
  Moreover, a set $\generalFunctionSet$ of
  self-mappings on the set $\barXX$ is called a \emph{\MdpomdpFunctionSet} if all ordered pairs of
  mappings $\np{\generalFunction_1, \generalFunction_2} \in \generalFunctionSet^{2}$ are
  \cemeterySeparated.
\end{definition}

\begin{definition}
  \label{def:mdpomdp}
  A \Mdpomdp\ is a \dpomdp\ such that the set~$\setsPushForward$ of pushforwards
  of the \dpomdp, defined in Equation~\eqref{eq:def_setsPushforward_total}, 
  is a \MdpomdpFunctionSet.
\end{definition}
Otherwise stated, for a \Mdpomdp, if two sequences of controls and observations lead to
the same state when starting in state $\states$, then applying the two sequences of
controls to another state $\states'$ either leads to the same state (with the same
sequence of observations), or at least one
sequence of controls leads to the cemetery point~$\cemetery$ (as we encounter a
different sequence of observations).

We now present a link between the notion of \separatedMappingSet\ and the notion of
\Mdpomdp. This allows us to propose a sufficient condition in order to ensure that a
\dpomdp\ is a \Mdpomdp.

\begin{proposition}
  \label{prop:dynamics_of_dpomdp_separated_ensure_mdpomdp}
  If the set
  $\bigcup_{\timeindex \in \timesetNoHorizon}
  \dynamics^{\UU^{\timeindex+1}}_{0:\timeindex} =
  \nset{\dynamics_{0:\timeindex}^{\controls_{0:\timeindex}}}{\forall \timeindex \in
    \timesetNoHorizon, \forall \controls_{0:\timeindex} \in \UU^{\timeindex+1}}$ of the
  composition of the evolution mappings of Problem~\eqref{eq:dpomdp_general_formulation}
  is a \separatedMappingSet\ (see Definition~\ref{def:separatedMappingSet}), then
  Problem~\eqref{eq:dpomdp_general_formulation} is a \Mdpomdp.
\end{proposition}

\begin{proof}
  The detailed
  proof is found in Appendix~\ref{sect:complements-det-pomdp}.
\end{proof}
Note that the observation mappings
$\nseqa{\observerfunct_{\timeindex}}{\timeindex \in \timesetNoHorizon}$ do not play any
role in Proposition~\ref{prop:dynamics_of_dpomdp_separated_ensure_mdpomdp}.

Now that we have defined the subclass of \Mdpomdp s, we present a bound on the cardinality of
the set of reachable beliefs for this subclass.

\subsection{Complexity analysis of \Mdpomdp}
\label{subsect:bounds_BBR_mdpomdp}

We now present the main interest of \Mdpomdp\ when compared to \dpomdp, namely that the
bound on cardinality of the set of reachable beliefs is lowered from
$\left( 1 + \cardinal{\XX} \right)^{\cardinal{\supp \np{\belief_0}}}$ to
$1 + \bp{2^{\cardinal{\supp \np{\belief_{0}} } } - \cardinal{\supp \np{\belief_0}}}
\cardinal{\XX}$.

\begin{theorem}
  \label{th:bound_belief_space_mdpomdp}
  Consider a \Mdpomdp\ optimization problem given by
  Problem~\eqref{eq:dpomdp_general_formulation} which satisfies the finite sets
  Assumption~\ref{assumpt:pomdp_finite_sets}.
  For any initial belief $\belief_0 \in \Delta\np{\XX}$, the cardinality of the set
  $\BBR_{\ic{1,\horizon}}(\belief_0)$ of
  reachable beliefs starting from~$\belief_0$ satisfies
  the following bound
  \begin{equation}
    \label{eq:mdpomdp_reachable_space_bound}
    \bcardinal{\BBR_{\ic{1,\horizon}}(\belief_0)}
    \leq
    1 + \bp{2^{\cardinal{\supp \np{\belief_{0}} } } - \cardinal{\supp \np{\belief_0}}}
    \cardinal{\XX}
    \eqsepv \forall \belief_0 \in \Delta\np{\XX}
    \eqfinp
  \end{equation}
\end{theorem}

\begin{proof}
  The proof is detailed in Appendix~\ref{sect:complements-det-pomdp}.
\end{proof}

We have therefore an improved complexity bound of the \DPA\ for
\Mdpomdp\ compared with standard \dpomdp.
\begin{corollary}
  \label{cor:complexity_bounds_mdpomdp}
  Consider a \Mdpomdp\ optimization problem given by
  Problem~\eqref{eq:dpomdp_general_formulation} which satisfies the finite sets
  Assumption~\ref{assumpt:pomdp_finite_sets}.
 Then, the \DPA\ recalled in Proposition~\ref{prop:alg_DP_belief_solves_dpomdp}
solves (numerically) 
Problem~\eqref{eq:dpomdp_general_formulation} with complexity
\begin{equation}
    O \left( \min \left( 1 + \bp{2^{\cardinal{\supp \np{\belief_{0}} } } - \cardinal{\supp
            \np{\belief_0}}} \cardinal{\XX}, 1 + \cardinal{\supp\np{\belief_0}}
        \cardinal{\UU}^{\cardinal{\timeset}} \right) \cardinal{\timeset} \cardinal{\UU}
      \cardinal{\OO}
    \right)
    \eqfinp
 \end{equation}
\end{corollary}

\begin{proof}
  By Proposition~\ref{prop:alg_DP_belief_solves_dpomdp}, the \DPA\
  solves Problem~\eqref{eq:dpomdp_general_formulation} and its complexity is
  $O \bp{\cardinal{\timeset} \cardinal{\BBR_{\ic{1,\horizon}}(\belief_0)} \cardinal{\UU}
    \cardinal{\OO}}$. Then, by Theorem~\ref{th:bound_belief_space_mdpomdp}, we
  have that 
  $\bcardinal{\BBR_{\ic{1,\horizon}}(\belief_0)} \leq 1 + \bp{2^{\cardinal{\supp
        \np{\belief_{0}} } } - \cardinal{\supp \np{\belief_0}}} \cardinal{\XX}$ and, by
  Theorem~\ref{th:reachable_belief_bounds}, we have that
  $ \bcardinal{\BBR_{\ic{1,\horizon}}(\belief_0)} \leq 1 + \cardinal{\supp\np{\belief_0}}
  \cardinal{\UU}^{\cardinal{\timeset}} $.
\end{proof}

As the bound presented in Theorem~\ref{th:bound_belief_space_mdpomdp}
depends on the states that can be reached when starting from states
in the support of the initial belief, we can obviously improve the bound
when the support of the belief belongs to a subset of $\XX$ stable by
the dynamics $\nseqa{\dynamics_{\timeindex}}{\timeindex \in \timeset}$.

\begin{corollary}
  Assuming that Problem~\eqref{eq:dpomdp_general_formulation} is a \Mdpomdp, that
  Assumption~\ref{assumpt:pomdp_finite_sets} holds, that
  $\cardinal{\supp \np{\belief_0}} > 1$, that the evolution mappings
  $\nseqa{\dynamics_{\timeindex}}{\timeindex \in \timesetNoHorizon}$ of
  Problem~\eqref{eq:dpomdp_general_formulation} satisfy the property that there
  exists a subset $A \subset \XX$ such that, for all time
  $\timeindex \in \timesetNoHorizon$,
  $\dynamics_{\timeindex}\np{A, \UU} \subset A$. 
  Then, the bound presented in
  Theorem~\ref{th:bound_belief_space_mdpomdp} can be improved as follows:
  \begin{equation}
    \label{eq:bound_mdpomdp_stability_improvement}
    \supp \np{\belief_0} \subset A \implies
    \bcardinal{\BBR_{\ic{1,\horizon}}(\belief_0)}
    \leq
    1 + \bp{2^{\cardinal{\supp \np{\belief_{0}} } } - \cardinal{\supp \np{\belief_0}}}
    \cardinal{A}
    \eqsepv \forall \belief_0 \in \Delta\np{\XX}
    \eqfinp
  \end{equation}  
\end{corollary}

Now that we have a better complexity bound than with non-separated \dpomdp s, the
question is whether it is tight or not. We now show that it is.

\subsection[Existence of \Mdpomdp s with tight complexity bound]{Existence of \Mdpomdp s
  with tight complexity bound}
\label{subsect:reaching_pomdp_bound}
In Theorem~\ref{th:bound_belief_space_mdpomdp}, we have given an improved bound on the
cardinality of the set of reachable beliefs for \Mdpomdp\ compared with standard \dpomdp.
We now prove that the bound is tight.

\begin{proposition}
  \label{prop:reaching_the_bound}
  There exists a \mdpomdp\ such that equality is obtained in
  Equation~\eqref{eq:mdpomdp_reachable_space_bound}, that is,
  \begin{equation}
    \bcardinal{\BBR_{\ic{1,\horizon}}(\belief_0)}
    =
    1 + \bp{2^{\cardinal{\supp \np{\belief_{0}} } } - \cardinal{\supp \np{\belief_0}}}
    \cardinal{\XX}
    \eqfinp
    \label{eq:mdpomdp_reachable_space_bound_attained}
  \end{equation}
\end{proposition}

\begin{proof}
  We exhibit a simple \mdpomdp\ for which the set~$\BBR_{\ic{1,\horizon}}(\belief_0)$
of reachable beliefs     satisfies
  Equation~\eqref{eq:mdpomdp_reachable_space_bound_attained}. Following the framework of
Sect.~\ref{subsect:formulation_dpomdp}, let $\XX = \na{ \states_1, \states_2, \states_3}$
  consists of
  three distinct states, $\OO = \na{ \bar{\observer}_{1}, \bar{\observer}_{2}}$ of two distinct observations,
  and $\UU = \na{ \bar{\controls}_{1}, \bar{\controls}_{2}}$ of two distinct controls. The
  evolution mappings are defined as
  $\forall \states \in \XX$, $\dynamics(\states, \bar{\controls}_{1}) = \states$, and
  $\forall i \in \na{1, 2, 3}$, $\dynamics(\states_i, \bar{\controls}_{2}) =
  \states_{\modulo\np{i,3}+1}$, where $\modulo \np{i, 3}$ is the remainder of the Euclidean
  division of the natural number~$i$ by~$3$.
  Finally, the observation mapping is given by
  $\observerfunct\np{\states, \controls} = \bar{\observer}_{2}$
  if~$\states = \states_3$ and~$\controls = \bar{\controls}_{1}$,
  and by~$\observerfunct\np{\states, \controls} = \bar{\observer}_{1}$
  otherwise.

  We show in Figure~\ref{fig:representation_pushforward_mdpomdp_bound_reached} the
  mappings $\pushforwardtransition^{\np{\controls, \observer}}$ defined in
  Equation~\eqref{eq:def_Cal_T_dpomdp} for this simple case, and we illustrate the
  dynamics and observation mappings in
  Figure~\ref{fig:representation_dynamics_observations_mdpomdp_bound_reached}.


\begin{figure}[!htp]
  \begin{minipage}[t]{0.56\textwidth}
    \includegraphics[scale=0.8]{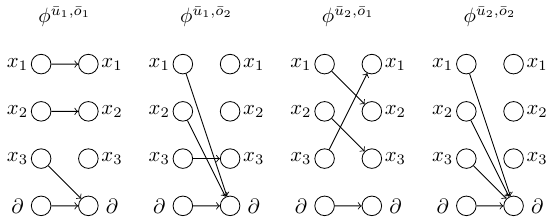}
  \caption{Representation of the $\pushforwardtransition^{\np{\controls, \observer}}$
    mappings in the case of \S\ref{subsect:reaching_pomdp_bound}
  }
  \label{fig:representation_pushforward_mdpomdp_bound_reached}
\end{minipage}
\hspace{0.5cm}
\begin{minipage}[t]{0.35\textwidth}
  \includegraphics[scale=0.8]{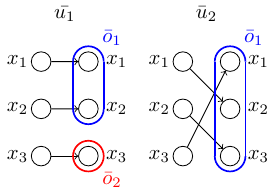}
  \caption{Representation of the dynamics and the observations depending on the control in
    the case of \S\ref{subsect:reaching_pomdp_bound}
  }
  \label{fig:representation_dynamics_observations_mdpomdp_bound_reached}
\end{minipage}
\end{figure}

By adding a cost function $\costfunct$, a horizon $\horizon > 0$ and admissibility
constraints $\admcontrolset: \XX \rightrightarrows \UU$, the resulting problem has all
the ingredients of a \dpomdp\ (as presented in Sect.~\ref{subsect:formulation_dpomdp}), where
Assumption~\ref{assumpt:pomdp_finite_sets} holds.

We now prove that the resulting \dpomdp\ is a \Mdpomdp. For that purpose, we enumerate
all the possible results of the dynamics before applying
Proposition~\ref{prop:dynamics_of_dpomdp_separated_ensure_mdpomdp}.
Let us consider a sequence 
$\np{\controls_{1}, \dots, \controls_{\timeindex}} \in \UU^{\timeindex}$
of controls.
By denoting $\dynamics^{\controls_{1:\timeindex}}$ the compositions of dynamics (i.e.
$\dynamics^{\controls_{1:\timeindex}}(\states) = \dynamics^{\controls_{\timeindex}} \circ
\dots \circ \dynamics^{\controls_1} \np{\states}$), we have that, for all $i \in \ic{1, 3}$,
 $ \dynamics^{\controls_{1:\timeindex}}(\states_i) = \states_{\modulo(i + \gamma
    \np{\controls_{1:\timeindex}} - 1, 3) + 1}$ ---
where $\gamma$ is the function that counts the number of times $\bar{\controls}_{2}$
appears in a sequence of controls (the function~$\gamma$ is defined as $
  \gamma: \UU^{\timeindex} \to \NN $, $
  \controls_{1:\timeindex} \mapsto
  \bcardinal{\nset{ \controls_i, i \in \ic{1, \timeindex}}{ \controls_i =
      \bar{\controls}_{2}}}
$).

The set
$\nset{\dynamics^{\controls_{1:\timeindex}}}{\controls_{1:\timeindex} \in
  \UU^{\timeindex}}$ is thus such that, for all sequences of controls
$\np{\controls_{1:\timeindex}, \controls'_{1:\timeindex'}} \in \UU^{\timeindex} \times
\UU^{\timeindex'}$, if there exists a state $\states \in \XX$ such that
$\dynamics^{\controls_{1:\timeindex}}(\states) =
\dynamics^{\controls'_{1:\timeindex'}}(\states)$, then we get that, for any state $\states' \in \XX$,
$\dynamics^{\controls_{1:\timeindex}}(\states') =
\dynamics^{\controls'_{1:\timeindex'}}(\states')$. Hence, the set
$\cup_{\timeindex \in \timesetNoHorizon}\dynamics_{0:\timeindex}^{\UU^{\timeindex+1}}$ is a
\separatedMappingSet. By
Proposition~\ref{prop:dynamics_of_dpomdp_separated_ensure_mdpomdp}, the optimization
problem is hence a \Mdpomdp.

We now choose an initial belief $\belief_{0}$ such that $\supp\np{\belief_{0}} =
\na{\states_1, \states_2}$, for which we can compute explicitly the reachable
beliefs (see Definition~\ref{def:reachable_space}).
We can apply
Theorem~\ref{th:bound_belief_space_mdpomdp} with such initial belief. Therefore, according
to Equation~\eqref{eq:mdpomdp_reachable_space_bound}, there can be at most $7$ reachable
beliefs (including $\cemeteryBelief$). In
Table~\ref{tab:resulting_support_proof_attained_bound}, we enumerate all possible supports
of the reachable beliefs when starting with belief~$\belief_0$.

\begin{table}[htbp]
  \centering
  \begin{tabular}{|c|c|c|c|}
    \hline
    Mapping applied & Support of resulting belief \\
    \hline
    $\pushforwardtransition^{\bar{\controls}_{1}, \bar{\observer}_{1}}$ & $\na{\states_1, \states_2}$ \\
    $\pushforwardtransition^{\bar{\controls}_{2}, \bar{\observer}_{1}}$ &  $\na{\states_2, \states_3}$ \\
    $\pushforwardtransition^{\bar{\controls}_{2}, \bar{\observer}_{1}} \circ \pushforwardtransition^{\bar{\controls}_{2}, \bar{\observer}_{1}}$ & $\na{\states_3, \states_1}$\\
    $\pushforwardtransition^{\bar{\controls}_{1}, \bar{\observer}_{2}} \circ \pushforwardtransition^{\bar{\controls}_{2}, \bar{\observer}_{1}}$ & $\na{\states_3}$ \\
    $\pushforwardtransition^{\bar{\controls}_{2}, \bar{\observer}_{1}} \circ \pushforwardtransition^{\bar{\controls}_{1}, \bar{\observer}_{2}} \circ \pushforwardtransition^{\bar{\controls}_{2}, \bar{\observer}_{1}}$ & $\na{\states_1}$ \\
    $\pushforwardtransition^{\bar{\controls}_{2}, \bar{\observer}_{1}} \circ \pushforwardtransition^{\bar{\controls}_{2}, \bar{\observer}_{1}} \circ \pushforwardtransition^{\bar{\controls}_{1}, \bar{\observer}_{2}} \circ \pushforwardtransition^{\bar{\controls}_{2}, \bar{\observer}_{1}}$ & $\na{\states_2}$ \\
   $\pushforwardtransition^{\bar{\controls}_{1}, \bar{\observer}_{2}}$ & $\na{\cemetery}$ \\

    \hline
  \end{tabular}
  \caption{Resulting support when applying given mappings to the initial belief
    $\belief_0$ with $ \supp \np{\belief_0} = \na{\states_1, \states_2}$}
  \label{tab:resulting_support_proof_attained_bound}
\end{table}

We have therefore $7$ different supports for the reachable beliefs, hence at least $7$
beliefs in the set of reachable beliefs starting from $\belief_0$. As
Equation~\eqref{eq:mdpomdp_reachable_space_bound} states that there can be at most $7$
reachable beliefs, we obtain that we have exactly $7$ reachable beliefs and thus
Equation~\eqref{eq:mdpomdp_reachable_space_bound_attained} is obtained.
\end{proof}

\begin{remark}
Note that, whereas the proof of Proposition~\ref{prop:reaching_the_bound} was made with a
\Mdpomdp\ with $\cardinal{\XX} = 3$, we can generate a \Mdpomdp\ such that equality is
obtained in Equation~\eqref{eq:mdpomdp_reachable_space_bound} for a set of any cardinality
$\cardinal{\XX} = n$, $n \geq 3$. We need once again that
$\XX = \nseqa{\states_i}{i \in \ic{1, n}}$ consists of $n$ distinct states,
$\OO = \na{ \bar{\observer}_{1}, \bar{\observer}_{2}}$ of two distinct observations and
$\UU = \na{ \bar{\controls}_{1}, \bar{\controls}_{2}}$ of two distinct controls. Then, the
dynamics is given by
$\forall \states \in \XX \eqsepv \dynamics(\states, \bar{\controls}_{1}) = \states$, and
$\forall i \in \ic{1, n}$,
$\dynamics(\states_i, \bar{\controls}_{2}) = \states_{\modulo\np{i,n}+1}$.
Finally, the observation mapping is given by
$\observerfunct\np{\states, \controls} = \bar{\observer}_{2}$
if~$\states = \states_n$ and~$\controls = \bar{\controls}_{1}$,
and by~$\observerfunct\np{\states, \controls} = \bar{\observer}_{1}$
otherwise.
\end{remark}

Now that we have presented the subclass of \Mdpomdp s, we give a numerical
illustration.

\section{Numerical application on an example of \Mdpomdp}
\label{sect:dpomdp_illustration}

In this section, we present a simple one-dimensional illustration of \Mdpomdp. We consider
that we empty a tank while minimizing an associated cost, as illustrated in
Figure~\ref{fig:illustration_bathtub}. The state is one-dimensional and consists in the
volume of water present in the tank. The control is also one-dimensional and is the amount
of water that the decision-maker removes during one time step. The decision-maker has
access at time $\timeindex$ to partial observation, as she only knows that the volume of
water in the tank is between two quantized levels.

\subsection{A partially observed tank as a \Mdpomdp}
More precisely, the problem is the following.

\begin{figure}[htbp]
  \begin{minipage}{0.6\linewidth}
\begin{itemize}
\item The state $\states$ consists of a discrete volume of water in the tank,
  with \( 0 \leq \states^{(1)} \leq \states^{(2)} \leq \cdots \leq \states^{(\inistatediscretization)} \) and 
  \\
  $\states \in \XX =
  \na{\states^{(1)}, \states^{(2)}, \dots , \states^{(\inistatediscretization)}} \subset
  \RR_+$ of finite cardinality $\inistatediscretization$.
\item The observation $\observer$ consists of a discrete level of water in the tank, with
\( 0 \leq \observer^{(1)} \leq \observer^{(2)} \leq \cdots \leq \observer^{(\observerdiscretization)}\)
and 
  \\$\observer \in \OO = \na{\observer^{(1)}, \observer^{(2)}, \dots,
    \observer^{(\observerdiscretization)}} \subset
  \RR_+$ of finite cardinality $\observerdiscretization$.
\item The control
  $\controls$ consists of a discrete volume of water to be removed, with
\( 0 \leq \controls^{(1)} \leq \controls^{(2)} \leq \cdots \leq \controls^{(\controlsdiscretization)}\)
and
\\$\controls \in \UU = \na{\controls^{(1)}, \controls^{(2)}, \dots,
    \controls^{(\controlsdiscretization)}} \subset \RR_+$ of finite cardinality
  $\controlsdiscretization$.
\item The unitary cost of water at each time $\timeindex \in \timesetNoHorizon$ is given
  by $\cost_{\timeindex} \in \RR$.
\end{itemize}
\end{minipage}
\hspace{0.3cm}
  \begin{minipage}{0.35\linewidth}
\begin{center}
  \includegraphics{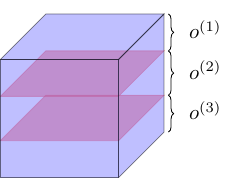}
\end{center}
\caption{Illustration of the water tank ``quantum'' of observation
  ($\observerdiscretization = 3$)}
\label{fig:illustration_bathtub}
\end{minipage}
\end{figure}

\subsubsubsection{Optimization problem}
We now adapt Problem~\eqref{eq:dpomdp_general_formulation} to the tank case
presented above:
\begin{subequations}
    \begin{align}
        \min_{\statesva, \controlsva, \observerva}
        & \espe \Bc{\sum_{\timeindex=0}^{\horizon-1} \cost_{\timeindex}
          \controlsva_{\timeindex} }
          \label{eq:obj_tank}\\
       s.t. ~
       & \PP_{\statesva_{0}} = \knownproba_0 \eqfinv
         \label{eq:initialisation_tank}\\
       & \statesva_{\timeindex+1} = \statesva_{\timeindex} - \controlsva_{\timeindex}
         \eqsepv \forall \timeindex \in \timeset \setminus \na{\horizon} \eqfinv
         \label{eq:dynamics_tank} \\
       & \controlsva_{\timeindex} \in \nset{ \controls^{(i)} \in \UU}{\controls^{(i)} \leq
          \statesva_{\timeindex}}
         \eqsepv \forall \timeindex \in \timeset \setminus \na{\horizon} \eqfinv
         \label{eq:admissibility_tank}\\
       & \observerva_{\timeindex} = \max \nset{ \observer^{(j)} \in \OO} { \statesva_{\timeindex}
           \geq \observer^{(j)} }
         \eqsepv \forall \timeindex \in \timeset \eqfinv
         \label{eq:observation_tank}\\
       & \sigma(\controlsva_{\timeindex}) \subset \sigma\left( \observerva_0, \dots,
           \observerva_{\timeindex}, \controlsva_{0}, \dots, \controlsva_{\timeindex-1}
       \right)
       \eqsepv \forall \timeindex \in \timeset \setminus \na{\horizon} \eqfinp
       \label{eq:non_anticipativity_tank}
    \end{align}
    \label{eq:formulation_dpomdp_tank}
\end{subequations}

Equation~\eqref{eq:obj_tank} represents the objective function of the tank problem, i.e.
Equation~\eqref{eq:dpomdp_gen_obj} of
Problem~\eqref{eq:dpomdp_general_formulation}. The instantaneous cost function at time
$\timeindex$ is defined as
$\costfunct_{\timeindex}(\controls_{\timeindex}) = \cost_{\timeindex}
\controls_{\timeindex}$, and hence only depends on the controls.
The evolution mapping corresponds to emptying the tank and is given by
$\dynamics: \np{\states, \controls} \mapsto \states - \controls$, which gives
Equation~\eqref{eq:dynamics_tank}.
The observation mapping $\observerfunct$ is given by a piecewise constant function which
does not depend on the controls $\controls$:
$\observerfunct(\states) = \max \nset{\observer^{(i)}}{\states \geq \observer^{(i)}}$.
This leads to equation \eqref{eq:observation_tank}, which
is the implementation of \eqref{eq:dpomdp_gen_observation_def}.
The admissibility set of the tank problem is given by
$\admcontrolset(\statesva_{\timeindex}) = [0, \statesva_{\timeindex}]$ (see
Equation~\eqref{eq:admissibility_tank}). It ensures that we cannot remove more water than
what is in the tank.

Problem~\eqref{eq:formulation_dpomdp_tank} has the same form as
Problem~\eqref{eq:dpomdp_general_formulation}. It is therefore a \dpomdp\ and all the
relevant results presented in \S\ref{sect:dp_for_dpomdp} hence apply.

\subsubsubsection{The partially observed tank problem as a \Mdpomdp}
The tank \dpomdp\ is a \Mdpomdp\ as a direct consequence of
Corollary~\ref{cor:affine_dynamics_mdpomdp}, in
Appendix~\ref{sect:complements-det-pomdp}. Indeed,
Corollary~\ref{cor:affine_dynamics_mdpomdp} states that if the evolution mappings
$\dynamics_{\timeindex}$ of a \dpomdp\ are linear, then it is a \Mdpomdp. As the evolution
function $\dynamics$ of the partially observed tank is indeed linear, the tank \dpomdp\ is
a \Mdpomdp.

\subsubsubsection{Associated beliefs dynamics $\beliefdynamics$}
Let $(\belief, \controls, \observer) \in \BB \times \UU \times \OO$, with
$\BB = \Delta\np{\XX} \cup \na{\cemeteryBelief}$, as defined in
Equation~\eqref{eq:def_BB}. As the evolution mappings and observation mappings are
stationary, the belief dynamics are also stationary.

By Equation~\eqref{eq:dynamics_tank}, we have
$\np{\dynamics^{\controls}}^{-1} \np{\nextstates} = \nextstates + \controls$.
As the observation mapping~$\observerfunct$ does not depend on the
control~$\controls$, $\intervalStates\np{\observer}$ is 
the set of states compatible with the observation~$\observer$.
Hence, 
the function~$\ProbaObservation$ in~\eqref{eq:def_proba_observation} is here
\begin{equation*}
  \ProbaObservation: \BB \times \UU \times \OO \to \left[ 0, 1 \right],
  (\belief, \controls, \observer) \mapsto
  \sum_{\states\in \intervalStates \np{\observer} - \controls
  }
  \belief(\states)
  \eqfinv
\end{equation*}
and Equation~\eqref{eq:beliefdynamics} gives
\[ \beliefdynamics(\belief, \controls, \observer) (\nextstates) =
\begin{cases}
  \frac{
  {\displaystyle
  \belief(\nextstates + \controls)}
  }
  {
  {\displaystyle
  \sum_{ \states' \in \intervalStates \np{\observer} - \controls
  }
  \belief(\states')
  }
  }
  &\text{if} \quad
    \nextstates \in
    \intervalStates \np{\observer} - \controls
    \eqfinv
  \\
  0
  &\text{if} \quad \nextstates \not\in
    \intervalStates \np{\observer} - \controls
    \eqfinp
\end{cases}
\]

\subsubsubsection{Bellman equations for the partially observed tank problem}
As Problem~\eqref{eq:formulation_dpomdp_tank} is a \dpomdp\ and the finite sets
Assumption~\ref{assumpt:pomdp_finite_sets} holds, we can apply
Proposition~\ref{prop:dynamics_belief_and_bellman}.
Equations~\eqref{eq:Bellman_DPOMDP_general_final}
and~\eqref{eq:Bellman_DPOMDP_general_bis} are here
\begin{subequations}
\begin{align}
    \valuefct_{\horizon}
    &: \BBR_{\horizon}(\belief_0) \to \RR
    \eqsepv \belief \mapsto 0
    \label{eq:Bellman_DPOMDP_tank_final}
    \\
    \valuefct_{\timeindex}
    &: \BBR_{\timeindex}(\belief_0) \to \RR
    \eqsepv  \belief \mapsto
    \min_{\controls \leq \min_{\states \in \supp \np{\belief}}\states}
    \Bp{
      \cost_{\timeindex}
      \controls
      + \sum_{\quad \observer \in \OO \quad}
      \sum_{\states - \controls \in \intervalStates \np{\observer}}  \belief(\states)
      \valuefct_{\timeindex+1}
      \bp{
        \beliefdynamics\np{ \belief, \controls, \observer}
      }
    }
    \eqfinp
    \label{eq:Bellman_DPOMDP_tank_bis}
\end{align}
\label{eq:Bellman_dpomdp_tank}
\end{subequations} 
Indeed, the intersection
$\beliefadmcontrolset_{\timeindex}(\belief) = \bigcap_{\states \in \supp\np{\belief}}
\admcontrolset_{\timeindex}(\states)$
is $\nset{\controls^{(i)} \in \UU}{\controls \leq \min_{\states \in \supp
      \np{\belief}}\states}$,
as the admissibility set is given by
Equation~\eqref{eq:admissibility_tank}, and as
\[\nset{ \controls^{(i)} \in \UU}{\controls^{(i)} \leq
    \states^{(j)}} \cap \nset{ \controls^{(i)} \in \UU}{\controls^{(i)} \leq
    \states^{(k)}} = \nset{ \controls^{(i)} \in \UU}{\controls^{(i)} \leq
    \min\bp{\states^{(j)}, \states^{(k)}}}
  \eqfinp
\]

\subsection{Numerical results}

We now present numerical results for the tank problem described by
Problem~\eqref{eq:formulation_dpomdp_tank}.

\subsubsubsection{Presentation of the instances}
We take the following data:
\begin{itemize}
\item $\XX = \ic{0, 300}$,
\item $\UU = \ic{0, 9}$,
\item $\OO = \na{0, 1, 20, 40, 60, 80, 100, 120, 140, 160, 180, 200, 220, 240,
    260, 280, 300}$,
\item $\timeset = \ic{0, 100}$,
\item $\supp \np{\belief_0} = \ic{260, 300}$, with a randomly generated probability
  distribution over that support, detailed in
  Figure~\ref{fig:value_b_0_bathtub}.
\end{itemize}

\begin{figure}[htbp!]
  \centering
  \includegraphics{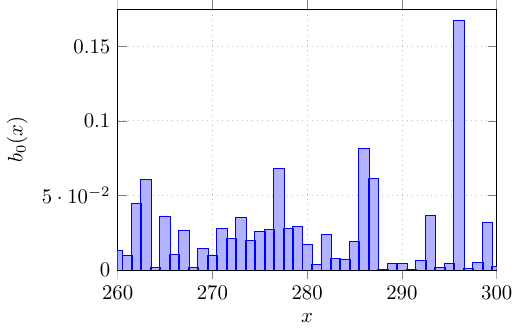}
  \caption{Probability distribution used as the initial belief $\belief_0$ for the
    numerical applications\hfill\label{fig:value_b_0_bathtub}}
\end{figure}

When considering the initial belief $\belief_0$ presented in
Figure~\ref{fig:value_b_0_bathtub} and a ``true'' (unknown) initial state of
$\states_0=290$ (used to simulate the observation process depending on the policy), we
obtain the trajectory of the tank water volume represented in Figure~\ref{fig:bathtub_trajectory_1}.

Moreover, we have a set of reachable beliefs $\reachablebeliefspace_{\ic{0,100}}$ such that
$\cardinal{\reachablebeliefspace_{\ic{0,100}}} = 64,400$. We therefore do not display
value functions, as they are defined on sets with too large cardinality.


We also made a second numerical application where the observation $\OO$ is changed to:
\begin{itemize}
\item $\OO = \na{1, 6, 11, 51, 101, 151, 201, 251}$
\end{itemize}

When considering the new observations set and the same initial belief and initial state,
we obtain the trajectory of the tank water volume represented in
Figure~\ref{fig:bathtub_trajectory_2}.

\begin{figure}[htbp]
  \begin{minipage}{0.48\linewidth}
    \includegraphics[scale=0.37]{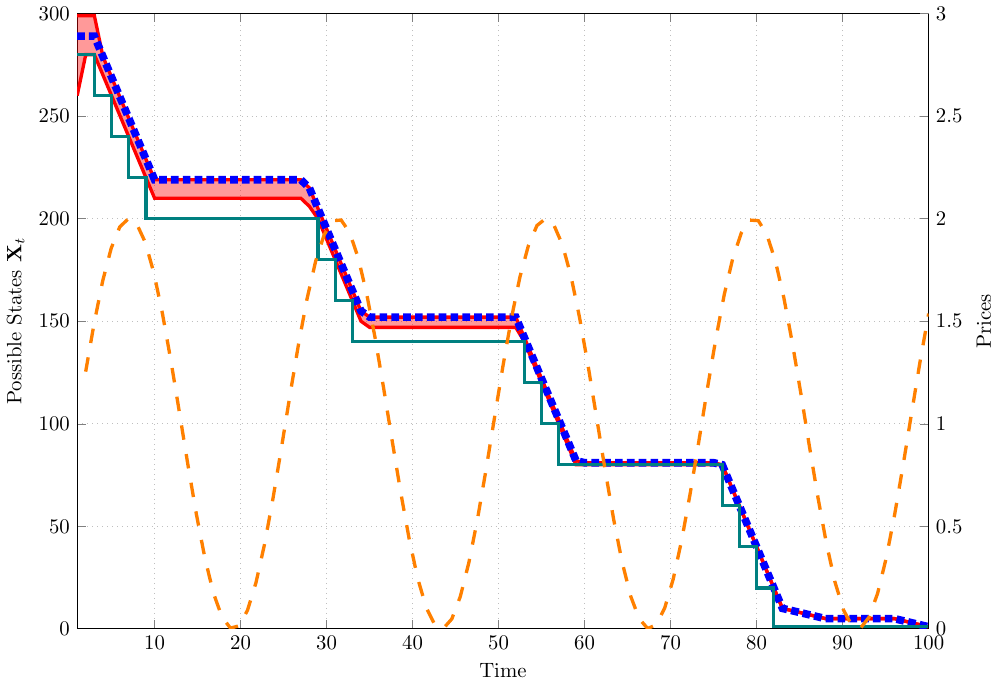}
  \caption{Representation of a trajectory of the volume of water in the tank when applying
    the optimal controls and considering the first set of observations. A vertical slice
    at time $\timeindex$ of the red area represents the support of the belief held at time
    $\timeindex$, the dotted blue curve represents the trajectory of the ``true'' state,
    the piecewise constant green curve is the observation we have access to at time
    $\timeindex$, and the dashed orange curve represents the periodic costs.
    \hfill\label{fig:bathtub_trajectory_1}}
\end{minipage}
\hspace{0.3cm}
\begin{minipage}{0.48\linewidth}
  \includegraphics[scale=0.37]{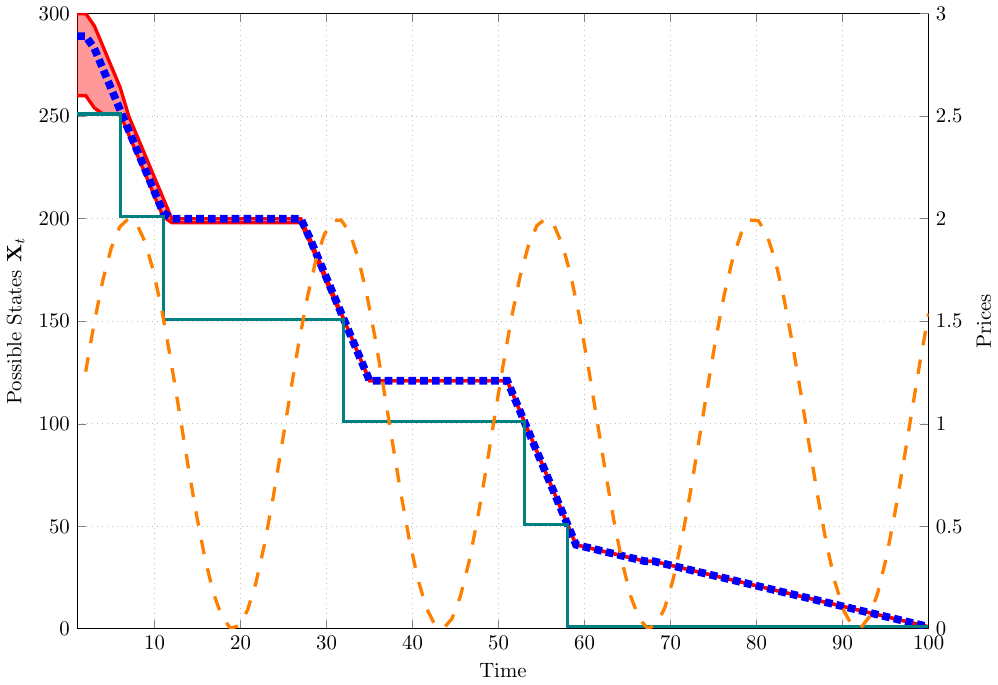}
  \caption{Representation of a trajectory of the volume of water in the tank when applying
    the optimal controls and considering the second set of observations. A vertical slice
    at time $\timeindex$ of the red area represents the support of the belief held at
    time~$\timeindex$, the dotted blue curve represents the trajectory of the ``true''
    state, the piecewise constant green curve is the observation we have access to at
    time~$\timeindex$, and the dashed orange curve represents the periodic costs.
    \hfill\label{fig:bathtub_trajectory_2}}
\end{minipage}
\end{figure}

Figures~\ref{fig:bathtub_trajectory_1} and~\ref{fig:bathtub_trajectory_2} both illustrate
some properties of \dpomdp s.
\begin{enumerate}
\item In both cases, we see that the size of the support of the beliefs decreases
  with time (the thickness of the vertical red slices is nonincreasing). 
\item We remark that such a decrease is due to the observations. Indeed, in
  Problem~\eqref{eq:formulation_dpomdp_tank}, the observation mapping ensures that the
  support of the beliefs must belong to intervals
  $\intervalStates\np{\observer_t}$ when we observe~$\observer_{\timeindex}$.
  Thus, the
  supports of the beliefs are reduced along the limit of those intervals, as is more
  easily seen in Figure~\ref{fig:bathtub_trajectory_2} between
  time~$\timeindex=1$ to~$\timeindex=6$.
  Indeed, at each time step in those periods, we remove some water, and we
  see that the lower part of the support remains at the observation value until time
  $t=7$. At that time, we change observation and we see that the upper bound of the
  support gets just beneath the previous observation, i.e. at $\states=249$.
\item
  We remark that, as could be expected, the optimal policy consists of removing
  water when costs are high, and stopping when costs are low.
\item
  We remark that, despite having fewer observations in the second case, the
  optimal trajectory in the second case reaches a Dirac (``deterministic'') belief (i.e. such that
  $\cardinal{\supp \np{\belief}} = 1$) much sooner in
  Figure~\ref{fig:bathtub_trajectory_2} compared to Figure~\ref{fig:bathtub_trajectory_1}
  (at time $\timeindex=33$ for the second case and time $\timeindex=53$ for the first
  case). Having more observations hence does not guarantee to remove ambiguities at a
  faster rate.
\end{enumerate}
We now present the computation time of the \DPA\ and compare it to
another algorithm, \SARSOP.

\subsubsubsection{Comparison with \SARSOP}
In this paragraph, we focus on the comparison with the algorithm \SARSOP, first introduced
in~\citep{sarsop_kurniawati_2008}. We used the Julia implementation of this algorithm,
with the POMDPs package API. The following results were obtained on a computer equipped
with a Core i7-8665U and 32~GB of memory, using Julia v$1.7.3$, POMDPs v$0.9.3$ and
\SARSOP\ v$0.5.5$.

We must first warn the reader that \SARSOP\ is an algorithm that solves an
infinite horizon \pomdp. We hence reformulate the finite horizon \dpomdp\ as an infinite
time \pomdp\ by extending the state with the time variable. Such reformulation leads to a
much bigger problem in terms of data and size of the state space, which heavily penalizes
\SARSOP. Hence, the reformulation prevents any
fair comparison of computation times. We still present some computation times in
Table~\ref{tab:computation_time_comparison_sarsop_mdpomdp}.

Note that, for each instance where the computation did not stop (i.e. those without a
``$>$'' symbol in the computation time column) due to hitting the memory limit of the
computer, \SARSOP\ and the \DPA\ have found the same values.


\begin{table}[htbp!]
  \centering
\begin{tabular}{|c|c|c|c|c|c|c|}
  \hline
  $\cardinal{\XX}$ & $\cardinal{\UU}$ & $\cardinal{\OO}$
  & $\cardinal{\supp \np{\belief_0}}$ & $\horizon$ & \SARSOP
  & \DPA\ \\
   & & & & & computation time (s) & computation time (s) \\
  \hline
  $11$ & $2$ & $3$ & $2$ & $20$ & $0.376$ & $0.002$ \\ 
  $21$ & $2$ & $5$ & $2$ & $25$ & $0.16$ & $0.003$ \\ 
  $51$ & $5$ & $5$ & $2$ & $100$ & $24.9$ & $0.20$ \\ 
  $51$ & $5$ & $5$ & $4$ & $100$ & $27.2$ & $1.20$ \\ 
  $51$ & $5$ & $5$ & $6$ & $100$ & $29.4$ & $3.03$ \\ 
  $101$ & $5$ & $5$ & $2$ & $200$ & $359$ & $0.96$ \\ 
  $101$ & $5$ & $5$ & $10$ & $200$ & $1930$ & $32.2$ \\ 
  $101$ & $10$ & $5$ & $10$ & $200$ & $1069$ & $78.2$ \\ 
  $201$ & $5$ & $5$ & $10$ & $200$ & $3506$ & $62.1$ \\ 
  $201$ & $10$ & $5$ & $10$ & $200$ & $15618$ & $309$ \\ 
  $201$ & $5$ & $5$ & $20$ & $200$ & $3652$ & $225$ \\ 
  $201$ & $10$ & $6$ & $20$ & $200$ & $33562$ & $497$ \\ 
  $301$ & $5$ & $6$ & $10$ & $200$ & $4638$ & $86.8$ \\ 
  $301$ & $10$ & $6$ & $10$ & $300$ & $>38000$  & $762$ \\ 
   & &  &  & &($>19217$s of iterations) & \\ 
  \hline
\end{tabular}
\caption{Computation time of different instances of both \SARSOP\ and
  the \DPA\
\label{tab:computation_time_comparison_sarsop_mdpomdp}
}
\end{table}

\section{Conclusion}
In this paper, we have presented a subclass of \pomdp s, \Mdpomdp s, which has properties
that contribute to push back the curse of dimensionality for \DP. Indeed, we have shown
that the conditions on the dynamics for \Mdpomdp\ improve the bound on the cardinality of
the set of the reachable beliefs: the bound is reduced from
$\bp{1 + \cardinal{\XX}}^{\cardinal{\supp \np{\belief_0}}}$ (in the case of \dpomdp, see
Theorem~\ref{th:reachable_belief_bounds}) to
$2^{\cardinal{\supp \np{\belief_0}}} \cardinal{\XX}$
(Theorem~\ref{th:bound_belief_space_mdpomdp}), as presented in
Table~\ref{tab:recap_bounds_dpomdp_infinite_horizon}.
This tighter bound guarantees that \DP~algorithms
efficiently solve \Mdpomdp\ problems, especially when considering small supports of the
initial state distributions. Moreover, the bound is tight (see
Proposition~\ref{prop:reaching_the_bound}).

The \Mdpomdp~class is, therefore, an interesting framework for some problems as only a
fraction of the number of beliefs needs to be considered, in comparison with \dpomdp\ or
\pomdp. The \Mdpomdp s are therefore tractable with larger instances than regular \pomdp s
or \dpomdp s.

\begin{table}[htbp!]
  \begin{center}
  \begin{tabular}{|c|c|c|}
    \hline
    Class    & Infinite horizon bound & Finite horizon bound \\
    \hline
    \hline
    \dpomdp   & $\left(1 + \cardinal{\XX} \right)^{\cardinal{\XX}}$ & $\min \bp{\left(1 + \cardinal{\XX} \right)^{\cardinal{\XX}},
     \bp{\cardinal{\UU} \cardinal{\OO} }^{\cardinal{\timeset}}}$\\
    & \citep{littman_thesis} & \\
    \hline
    \hline
    \dpomdp  & $\left(1 + \cardinal{\XX} \right)^{\cardinal{\supp \np{\belief_0}}}$
             & $\min \bp{\left(1 + \cardinal{\XX} \right)^{\cardinal{\supp \np{\belief_0}}},
      1+ \cardinal{\supp\np{\belief_0}} \cardinal{\UU}^{\cardinal{\timeset}}}$ \\
    improved bounds & Theorem~\ref{th:reachable_belief_bounds} & Theorem~\ref{th:reachable_belief_bounds} \\
    \hline
    Separated & $1 + \bp{2^{\cardinal{\supp \np{\belief_{0}} } } -
                \cardinal{\supp \np{\belief_0}}} \cardinal{\XX} $
              & $\min \Big( 1 + \bp{2^{\cardinal{\supp \np{\belief_{0}} } } -
                \cardinal{\supp \np{\belief_0}}} \cardinal{\XX}, $ \\
    \dpomdp  & & $\left. 1+ \cardinal{\supp\np{\belief_0}} \cardinal{\UU}^{\cardinal{\timeset}}\right)$\\
    & Theorem~\ref{th:bound_belief_space_mdpomdp} & Corollary~\ref{cor:complexity_bounds_mdpomdp} \\
    \hline
\end{tabular}
\end{center}
\caption{Summary of the bounds depending on the class of problem
}
\label{tab:recap_bounds_dpomdp_infinite_horizon}
\end{table}

\appendix

\section{Appendix}

First, in \S\ref{sect:technical_lemmatas_pushforward}, we present technical lemmata used to
prove bounds on the cardinality of the sets of reachable beliefs. Then, in
\S\ref{sect:complements-det-pomdp}, we present complementary results on \cemeterySeparated\
mappings sets.

\subsection{Technical lemmata}
\label{sect:technical_lemmatas_pushforward}
In this subsection, we present technical lemmata used in the proofs of
Theorem~\ref{th:reachable_belief_bounds}.
We first introduce in \S\ref{subsect:notations_technical_lemmatas}
the notions of forward and backward mappings. Second, in
\S\ref{subsect:tech_lematas_X_right_mappings}, we present properties on the composition and
pushforward measures by those forward and backward mappings. Third, in
\S\ref{subsect:bounds_cardinal_sets_pushfowards}, we present properties on the cardinality
of sets of forward and backward mappings, used notably in the proof of
Theorem~\ref{th:reachable_belief_bounds}.

\subsubsection{Forward and backward mappings}
\label{subsect:notations_technical_lemmatas}

For any subset $\Primal \subset \XX$, we introduce the notion of $\DualP$-forward and
$\DualP$-backward mappings. Given a mapping $\fonctionprimalbis:\PRIMAL\to \DUALP$ and a
subset $\DualP\subset \DUALP$, we define a mapping
$\Mapforward{\fonctionprimalbis}{\DualP}: \overline{\PRIMAL} \to \overline{\DUALP}$
(recall that $\overline{\PRIMAL} = \PRIMAL \cup \na{\cemetery}$,
in~\eqref{eq:def_bar_XX}), called a \emph{$\DualP$-forward mapping}, as follows
\begin{align}
  \Mapforward{\fonctionprimalbis}{\DualP}:
  \primal \in \overline{\PRIMAL}
  & \mapsto
    \begin{cases}
      \fonctionprimalbis(\primal)
      & \text{if}\quad \primal \in \PRIMAL \quad\text{and}\quad
        \fonctionprimalbis(\primal)\in \DualP
        \eqfinv
      \\
      \cemetery
      & \text{if}\quad \primal=\cemetery\quad\text{or}\quad
        \fonctionprimalbis(\primal)\not\in \DualP
        \eqfinp
    \end{cases}
    \label{eq:definition_fonctionprimal_dual_forward}
\end{align}
We call $\Mapforward{\fonctionprimalbis}{\DualP}:\overline{\PRIMAL} \to \overline{\DUALP}$
an $\DualP$-forward mapping as the subset $\DualP$ appearing in Equation~\eqref{eq:definition_fonctionprimal_dual_forward} is a subset of the codomain of $\fonctionprimalbis$

We also introduce the \emph{$\DualP$-backward mapping}
$\Mapbackward{\fonctionprimalbis}{\DualP}:\overline{\PRIMAL} \to \overline{\DUALP}$,
defined by
\begin{align}
  \Mapbackward{\fonctionprimalbis}{\Primal}:
  \primal \in \overline{\PRIMAL}
  & \mapsto
    \begin{cases}
      \fonctionprimalbis(\primal)
      & \text{if}\quad \primal \in \DualP
        \eqfinv
      \\
      \cemetery
      & \text{otherwise.}
    \end{cases}
    \label{eq:definition_fonctionprimal_dual_backward}
\end{align}
We call
$\Mapbackward{\fonctionprimalbis}{\DualP}:\overline{\PRIMAL} \to \overline{\DUALP}$ an
$\DualP$-backward mapping as
the subset $\DualP$ appearing in Equation~\eqref{eq:definition_fonctionprimal_dual_backward} is a subset of the domain of $\fonctionprimalbis$.

It is straightforward to check that we have
\begin{subequations}
  \begin{align}
    \forall \DualP \subset \DUALP \eqsepv
    &     \Mapforward{\fonctionprimalbis}{\DualP} =
      \Mapbackward{\fonctionprimalbis}{\fonctionprimalbis^{-1}(\Primal)}
      \eqfinv
      \label{eq:forward-to-backward}
    \\
    \forall \DualP \subset \DUALP \eqsepv
    & \Mapforward{\fonctionprimalbis}{\Primal}=
      \Mapforward{\fonctionprimalbis}{\Primal \cap \image\np{\fonctionprimalbis}}
      \eqfinv
      \label{eq:forward_equal_forward_cap_image}
  \end{align}
\end{subequations}
where, for any mapping $\fonctionprimalbis$, $\image\np{\fonctionprimalbis}$ is
the image of mapping $\fonctionprimalbis$, that is
$\image\np{\fonctionprimalbis} = \fonctionprimalbis\np{\XX}$. Using~\eqref{eq:forward-to-backward}, we obtain that a forward mapping
can be rewritten as a backward mapping.
The reverse is not always true as illustrated by the following example.
Consider $\PRIMAL=\na{\primal_1, \primal_2}$, the constant mapping $\fonctionprimalbis: \na{\primal_1, \primal_2}\mapsto \primal_1$ and
$\Primal=\na{x_1}$. Then, $\Mapbackward{\fonctionprimalbis}{\DualP}$ is given by $\Mapbackward{\fonctionprimalbis}{\DualP}\np{\primal_1}=\primal_1$ and
$\Mapbackward{\fonctionprimalbis}{\DualP}\np{\primal_2}=\cemetery$ and it cannot be equal to $\Mapforward{\fonctionprimalbis}{\DualP'}$ for any
$\DualP'\subset \PRIMAL$. Indeed, if it were the case, we would have $\Mapforward{\fonctionprimalbis}{\DualP'}\np{\primal_1}=
\Mapbackward{\fonctionprimalbis}{\DualP}\np{\primal_1}=\primal_1$ and this would imply
$\primal_1 \in \DualP'$. Thus we would also have $\Mapforward{\fonctionprimalbis}{\DualP'}\np{\primal_2}=\primal_1$ as $\fonctionprimalbis(\primal_2)=\primal_1 \in
\DualP'$ and finally we would obtain
$\Mapforward{\fonctionprimalbis}{\DualP'}\np{\primal_2}=\primal_1\not= \cemetery
= \Mapbackward{\fonctionprimalbis}{\DualP}\np{\primal_2}$, hence leading to a contradiction.

\subsubsection{Results on pushforward measures by forward and backward mappings sets}
\label{subsect:tech_lematas_X_right_mappings}

We now present properties of the composition of pushforward measures by forward and
backward mappings.

\begin{definition}
  \label{def:BackwardAndForwardMappingSets}
  Let $\mathbb{M} \subset \Mappings{\PRIMAL}{\PRIMAL}$ be a subset of self mappings on the
  set $\PRIMAL$. We say that $\GG \subset \Mappings{\overline\PRIMAL}{\overline\PRIMAL}$
  is an \emph{\MBackwardMappings\ set} (resp. an \emph{\MForwardMappings\ set}) if it
  satisfies the following property
  \begin{subequations}
    \begin{align}
      \GG \subset
      \bset{ \Mapbackward{\fonctionprimalbis}{\DualP}
      }{\fonctionprimalbis \in \mathbb{M} \text{ and } \DualP \subset \PRIMAL}
      \eqsepv
           \label{MBackwardMappingsDef_a}
      \\
      \big( \text{resp.}\quad
      \GG \subset
      \bset{ \Mapforward{\fonctionprimalbis}{\DualP}
      }{\fonctionprimalbis \in \mathbb{M} \text{ and } \DualP \subset \PRIMAL} \big)
      \eqfinv
    \end{align}
      \label{MBackwardMappingsDef}
  \end{subequations}
  where $\Mapbackward{\fonctionprimalbis}{\DualP}$ (resp.
  $\Mapforward{\fonctionprimalbis}{\DualP}$) is defined in
  Equation~\eqref{eq:definition_fonctionprimal_dual_backward} (resp.
  Equation~\eqref{eq:definition_fonctionprimal_dual_forward}). When
  $\mathbb{M}= \Mappings{\PRIMAL}{\PRIMAL}$, an \MBackwardMappings\ set (resp. an
  \MForwardMappings\ set) is just named an \emph{\BackwardMappings\ set} (resp. an
  \emph{\ForwardMappings\ set}).
\end{definition}

We obtain the following properties.
\begin{itemize}
\item If $\GG$ is an \MForwardMappings\ set, then $\GG$ is an \MBackwardMappings\ set (using
  Equality~\eqref{eq:forward-to-backward}).
\item \BackwardMappings\ sets are stable by composition, as we easily obtain that
  \begin{equation}
    \Mapbackward{\fonctionprimalbis'}{\DualP'} \circ
    \Mapbackward{\fonctionprimalbis}{\DualP}
    = \Mapbackward{(\fonctionprimalbis'\circ \fonctionprimalbis)}
    {\Primal \cap \fonctionprimalbis^{-1}(\Primal')}
    \eqfinp
    \label{eq:stable_composition_backward_mappings}
  \end{equation}
\item Let $\GG$ be an \BackwardMappings\ set and consider, for any
  $\Primal \subset \PRIMAL$, the subset $\BackwardSet{\Primal}$ of~$\GG$ defined by
  \begin{equation}
    \BackwardSet{\Primal} =
    \bset{g \in \GG}{ \exists \fonctionprimalbis \in \Mappings{\XX}{\XX}, g = \Mapbackward{\fonctionprimalbis}{\DualP}}
    \eqfinp
  \end{equation}
  Then, for any belief $\belief_0 \in \Delta\np{\PRIMAL}$, we have
  \begin{equation}
    \bp{\Renormalization\circ \np{\BackwardSet{\Primal\cap \supp\np{\belief_0}}}_{\star}}
    \np{\belief_0}
    =
    \bp{ \Renormalization\circ \np{\BackwardSet{\Primal}}_{\star}} \np{\belief_0}
    \eqfinp
    \label{eq:Backward-support-b0}
  \end{equation}
  Equation~\eqref{eq:Backward-support-b0} is a consequence of the following
  Lemma~\ref{lem:ext_pushforward_renorm_result}. Indeed,
  the expression of
  $\bp{ \Renormalization\circ \np{\BackwardSet{\Primal}}_{\star}} \np{\belief_0}$ given by
  Equation~\eqref{eq:R-circ-backward} only depends on the restriction of the measure
  $\belief_0$ to the subset $\Primal$ -- which coincides with the restriction of the measure
  $\belief_0$ to the subset $\Primal\cap\supp\np{\belief_0}$ -- as the measure $\belief_0$
  is null outside its support.
\end{itemize}

\begin{lemma}
  \label{lem:ext_pushforward_renorm_result}
  Let $\Primal$ be a subset of $\PRIMAL$.
  The mappings $\Renormalization\circ \np{ \Mapbackward{\fonctionprimalbis}{\DualP}}_{\star}$
  and $\Renormalization\circ \np{ \Mapforward{\fonctionprimalbis}{\DualP}}_{\star}$
  in $\Mappings{\Delta(\overline{\PRIMAL})}{\BB}$ ---
  where the pushforward measure is defined in Equation~\eqref{eq:def_push_forward},
  and the mapping $\Renormalization$ is defined in
  Equation~\eqref{eq:definition_renormalisation} ---
  have the following expressions: for all $\measureOnBarXX\in \Delta(\overline{\PRIMAL})$,
  \begin{subequations}
    \begin{align}
      \bp{\Renormalization\circ \np{\Mapforward{\fonctionprimalbis}{\DualP}}_{\star}}
      \np{\measureOnBarXX}
      &=
        \begin{cases}
          \Bc{\overline{\dualP} \in {\overline{\DUALP}}
          \mapsto
          \frac{\displaystyle{\measureOnBarXX}\bp{
          \fonctionprimalbis^{-1}(\overline{\dualP})}\findi{\DualP}\np{\overline{\dualP}}}
          {\displaystyle{\measureOnBarXX}\bp{ \fonctionprimalbis^{-1}(\DualP)}}}
          & \text{if}\quad {\measureOnBarXX}\bp{ \fonctionprimalbis^{-1}(\DualP)}\not=0
            \eqfinv
          \\
          \delta_{\cemetery}
          & \text{otherwise,}
        \end{cases}
        \label{eq:R-circ-forward}
        \intertext{and} 
      \bp{\Renormalization\circ \np{\Mapbackward{\fonctionprimalbis}{\DualP}}_{\star}}\np{
      \measureOnBarXX}
      &=
        \begin{cases}
          \Bc{\overline{\dualP} \in {\overline{\DUALP}}
          \mapsto
          \frac{\displaystyle{\measureOnBarXX}\bp{ \fonctionprimalbis^{-1}(
          \overline{\dualP})\cap \Primal}}
          {\displaystyle{\measureOnBarXX}\bp{ \fonctionprimalbis^{-1}(\DUALP)\cap \Primal}}}
          & \text{if}\quad {\measureOnBarXX}\bp{
            \fonctionprimalbis^{-1}(\DUALP)\cap \Primal}\not=0
            \eqfinv
          \\
          \delta_{\cemetery}
          & \text{otherwise.}
        \end{cases}
        \label{eq:R-circ-backward}
    \end{align}
  \end{subequations}
  \label{eq:R-circ-forward-backward}
\end{lemma}

\begin{proof}
  For any probability measure $\measureOnBarXX$ on the finite set $\overline{\PRIMAL}$, it
  is straightforward, using the definition of pushforward measure in
  Equation~\eqref{eq:def_push_forward}, to obtain that the pushforward of the measure
  $\measureOnBarXX$ through the mapping~$\Mapforward{\fonctionprimalbis}{\DualP}$, as
  defined in Equation~\eqref{eq:definition_fonctionprimal_dual_forward}, is given by
  \begin{align}
    \np{\Mapforward{\fonctionprimalbis}{\DualP}}_{\star}\measureOnBarXX
    : \overline{\DUALP}
    &\to \RR_{+}
      \nonumber \\
    \overline{y}
    &\mapsto
      \measureOnBarXX\bp{ \np{\Mapforward{\fonctionprimalbis}{\DualP}}^{-1}(\overline{y})}
      =
      \begin{cases}
        \measureOnBarXX\bp{ \fonctionprimalbis^{-1}(\overline{y})}
        & \text{if}\quad \overline{y}\in \DualP
          \eqfinv
        \\
        \Bp{1- \measureOnBarXX\bp{ \fonctionprimalbis^{-1}(\DualP)}}
        & \text{if}\quad \overline{y}=\cemetery
          \eqfinv
        \\
        0
        & \text{if}\quad \overline{y}\not=\cemetery \text{ and } \overline{y}\not\in \DualP
          \eqfinp
      \end{cases}
      \label{eq:pushforward-fY}
  \end{align}
  Thus, we obtain that
  \begin{equation}
    \forall \dualP \in \DUALP \eqsepv
    \bp{\np{\Mapforward{\fonctionprimalbis}{\DualP}}_{\star}
      \measureOnBarXX}_{\vert_{\DUALP}}
    (\dualP)
    = \measureOnBarXX\bp{ \fonctionprimalbis^{-1}({\dualP})}\findi{\DualP}{({\dualP})}
    \eqfinv
  \end{equation}
  and that
  \begin{equation}
    \bp{\np{\Mapforward{\fonctionprimalbis}{\DualP}}_{\star}\measureOnBarXX}(\DUALP)
    = \sum_{\dualP \in \DUALP} \measureOnBarXX\bp{ \fonctionprimalbis^{-1}(\dualP)}
    \findi{\DualP}{({\dualP})}
    = \measureOnBarXX\bp{ \fonctionprimalbis^{-1}(\DualP)}
    \eqfinp
  \end{equation}
  Hence, using the definition of $\Renormalization$ in
  Equation~\eqref{eq:definition_renormalisation}, the result follows from
  Equation~\eqref{eq:R-circ-forward}. The proof of Equation~\eqref{eq:R-circ-backward} is
  very similar and left to the reader.
\end{proof}

The composition of self-mappings of the form
$\Renormalization\circ \np{\Mapforward{\fonctionprimalbis}{\DualP}}_{\star}$ can also be written
without resorting to multiple renormalizations. Instead, we only need to renormalize the
composition of the pushforward measures, as shown below.

\begin{lemma}
  \label{lem:composition_pushforward_equality}
  Assume that $\fonctionprimalbis$ and $\fonctionprimalbis'$ are self-mappings on the
  finite set $\PRIMAL$. Then, for any subsets $\Primal$ and $\Primal'$ of
  $\PRIMAL$, we have the following composition equalities
  \begin{subequations}
    \begin{align}
      \Renormalization\circ \np{\Mapforward{\fonctionprimalbis}{\DualP}}_{\star}
      \circ \Renormalization\circ \np{\Mapforward{\fonctionprimalbis'}{\DualP'}}_{\star}
      &=
        \Renormalization\circ \np{\Mapforward{\fonctionprimalbis}{\DualP} \circ
        \Mapforward{\fonctionprimalbis'}{\DualP'}}_{\star}
        \eqfinv
        \label{eq:lem_composition_pushforward}
      \\
      \Renormalization\circ \np{\Mapbackward{\fonctionprimalbis}{\DualP}}_{\star}
      \circ \Renormalization\circ \np{\Mapbackward{\fonctionprimalbis'}{\DualP'}}_{\star}
      & =
        \Renormalization\circ \np{\Mapbackward{\fonctionprimalbis}{\DualP} \circ
        \Mapbackward{\fonctionprimalbis'}{\DualP'}}_{\star}
        \eqfinp
        \label{eq:lem_composition_pushbackward}
    \end{align}
    \label{eq:lem_composition_pushforward_pushbackward}
  \end{subequations}
\end{lemma}
\begin{proof}
  We just prove Equation~\eqref{eq:lem_composition_pushforward} as the proof follows the
  same lines for Equation~\eqref{eq:lem_composition_pushbackward}. As a preliminary, we
  remark that the mapping
  $\Renormalization\circ \np{\Mapforward{\fonctionprimalbis}{\DualP}}_{\star}$ is defined on the
  nonnegative measures on the set $\overline{\PRIMAL}$ and not just on probability
  measures. Now, given $\mu \in \Delta\np{\overline{\PRIMAL}}$, we consider the
  nonnegative measure $\mu'=(\mu_{\vert_{\PRIMAL}},0)$. The two nonnegative measures $\mu$
  and $\mu'$ coincide on the set $\PRIMAL$. Thus using the expression of
  $\Renormalization\circ \np{\Mapforward{\fonctionprimalbis}{\DualP}}_{\star}$ in
  Equation~\eqref{eq:R-circ-forward} and the fact that $\Primal \subset \PRIMAL$, we
  obtain that
  $\Renormalization\circ \np{\Mapforward{\fonctionprimalbis}{\DualP}}_{\star}(\mu) = \Renormalization\circ
  \np{\Mapforward{\fonctionprimalbis}{\DualP}}_{\star} (\mu_{\vert_{\PRIMAL}},0)$.

  Now, let $\measureOnBarXX \in \Delta\np{\overline{\PRIMAL}}$ be given. We denote by
  $\measureOnBarXXprime\in \Delta\np{\overline{\PRIMAL}}$ the probability measure
  $\measureOnBarXXprime = \np{\Mapforward{\fonctionprimalbis'}{\DualP'}}_{\star}
  \measureOnBarXX$. We consider two cases: either $\measureOnBarXXprime(\PRIMAL) \neq 0$,
  or $\measureOnBarXXprime(\PRIMAL) = 0$.

  \textbf{First case}. We assume that $\measureOnBarXXprime(\PRIMAL) \neq 0$. Then,
  we successively have
  \begin{align*}
    \Renormalization\circ \np{\Mapforward{\fonctionprimalbis}{\DualP}}_{\star}
      \circ \Renormalization\circ \np{\Mapforward{\fonctionprimalbis'}{\DualP'}}_{\star} \measureOnBarXX
    & =
      \Renormalization\circ \np{\Mapforward{\fonctionprimalbis}{\DualP}}_{\star} \circ \Renormalization
      (\measureOnBarXXprime)
      \tag{by replacing $\np{\Mapforward{\fonctionprimalbis'}{\DualP'}}_{\star} \measureOnBarXX$ by
      $\measureOnBarXXprime$}
    \\
    &=
      \Renormalization\circ \np{\Mapforward{\fonctionprimalbis}{\DualP}}_{\star}
      \bp{\frac{1}{\measureOnBarXXprime(\PRIMAL)} \measureOnBarXXprime_{\vert_{\PRIMAL}}, 0}
      \tag{using $\Renormalization$ definition in~\eqref{eq:definition_renormalisation},
      with $\measureOnBarXXprime(\PRIMAL) \neq 0$}
    \\
    &=
      \Renormalization\circ \np{\Mapforward{\fonctionprimalbis}{\DualP}}_{\star}
      \bp{\frac{1}{\measureOnBarXXprime(\PRIMAL)}\np{\measureOnBarXXprime_{\vert_{\PRIMAL}}, 0}}
      \tag{factorizing by $\frac{1}{\measureOnBarXXprime\np{\overline{\PRIMAL}}}$}
    \\
    &=
      \Renormalization \Bp{\frac{1}{\measureOnBarXXprime(\PRIMAL)}
      \np{\Mapforward{\fonctionprimalbis}{\DualP}}_{\star}
      \bp{\measureOnBarXXprime_{\vert_{\PRIMAL}}, 0}}
      \tag{as $\np{\Mapforward{\fonctionprimalbis}{\DualP}}_{\star}$ is 1-positively
      homogeneous}
    \\
    &=
      \Renormalization\bp{\np{\Mapforward{\fonctionprimalbis}{\DualP}}_{\star}
      \bp{\measureOnBarXXprime_{\vert_{\PRIMAL}}, 0}}
      \tag{as $\Renormalization$ is 0-positively homogeneous}
    \\
    &=
      \Renormalization\bp{\np{\Mapforward{\fonctionprimalbis}{\DualP}}_{\star}
      \np{\measureOnBarXXprime}}
      \tag{using the preliminary part}
    \\
    &=
      \Renormalization \circ \np{\Mapforward{\fonctionprimalbis}{\DualP}}_{\star} \circ
      \np{\Mapforward{\fonctionprimalbis'}{\DualP'}}_{\star} \measureOnBarXX
      \tag{as $\measureOnBarXXprime =  \np{\Mapforward{\fonctionprimalbis'}{\DualP'}}_{\star}
      \measureOnBarXX$}
    \\
    &=
      \Renormalization\circ \np{{\Mapforward{\fonctionprimalbis}{\DualP}}\circ
      \Mapforward{\fonctionprimalbis'}{\DualP'}}_{\star}
      (\measureOnBarXX) 
  \end{align*}
  as \( \np{\Mapforward{\fonctionprimalbis}{\DualP}}_{\star} \circ \np{\Mapforward{\fonctionprimalbis'}{\DualP'}}_{\star}\)
  \(= \np{\Mapforward{\fonctionprimalbis}{\DualP} \circ \Mapforward{\fonctionprimalbis'}{\DualP'}}_{\star} \)
  by definition~\eqref{eq:def_push_forward} of a pushforward measure.

  \textbf{Second case}. We assume that $\measureOnBarXXprime(\PRIMAL)=0$. Then, we have
  that $\measureOnBarXXprime= \delta_{\cemetery}$ as
  $\measureOnBarXXprime \in \Delta\np{\overline{\PRIMAL}}$, and we obtain
  \begin{align*}
    \Renormalization\circ \np{\Mapforward{\fonctionprimalbis}{\DualP}}_{\star}
    \circ \Renormalization\circ \np{\Mapforward{\fonctionprimalbis'}{\DualP'}}_{\star} \measureOnBarXX
    & =
      \Renormalization\circ \np{\Mapforward{\fonctionprimalbis}{\DualP}}_{\star} \circ \Renormalization
      (\delta_{\cemetery})
      \tag{by replacing $\np{\Mapforward{\fonctionprimalbis'}{\DualP'}}_{\star} \measureOnBarXX$ by
      $\measureOnBarXXprime=\delta_{\cemetery}$}
    \\
    & =
      \Renormalization\circ \np{\Mapforward{\fonctionprimalbis}{\DualP}}_{\star}
      (\delta_{\cemetery})
      \tag{as $\Renormalization (\delta_{\cemetery})= \delta_{\cemetery}$ }
    \\
    &= \Renormalization\circ \np{\Mapforward{\fonctionprimalbis}{\DualP}}_{\star}
      \circ \np{\Mapforward{\fonctionprimalbis'}{\DualP'}}_{\star} \measureOnBarXX
      \tag{by replacing $\delta_{\cemetery}=\measureOnBarXXprime$ by
      $\np{\Mapforward{\fonctionprimalbis'}{\DualP'}}_{\star}\measureOnBarXX$}
    \\
    &=
      \Renormalization\circ \np{{\Mapforward{\fonctionprimalbis}{\DualP}}\circ
      \Mapforward{\fonctionprimalbis'}{\DualP'}}_{\star}
      (\measureOnBarXX)
      \eqfinp
  \end{align*}
  Hence, in both cases, we obtain Equation~\eqref{eq:lem_composition_pushforward}.
\end{proof}

Now that we have exposed technical lemmata on the composition and renormalization of
\ForwardMappings\ and \BackwardMappings, we present lemmata on the cardinality of sets of
pushforward measures, notably the cardinality of pushforward measures by \ForwardMappings\
and \BackwardMappings.

\subsubsection{Results on the cardinality of sets of pushforward measures}
\label{subsect:bounds_cardinal_sets_pushfowards}

First, we bound the cardinality of the set of pushforward of a given nonnegative
measure thanks to the following Lemma~\ref{lem:restriction_counting_pushforward}.

\begin{lemma}
  \label{lem:restriction_counting_pushforward}
  Let $\mathbb{J} \subset \Mappings{\VV}{\YY}$ be a subset of mappings from the set $\VV$
  to the set $\YY$. Assume that the sets $\VV$ and $\YY$ are both finite. Then, for
  any {nonnegative} measure $\mu$ on the set $\VV$, we have that
  \begin{equation}
    \cardinal{\mathbb{J}_{\star} \np{\mu}}
    \le
    \cardinal{\YY}^{\cardinal{\supp \np{\mu}}}
    \eqfinv
    \label{eq:bounds_on_push_forward}
  \end{equation}
  where we recall that $\cardinal{\mathbb{J}_{\star} \np{\mu}}$ denotes the cardinal of
  the set $\bcardinal{ \nset{ j_{\star}\mu}{ j \in \mathbb{J}}}$ as exposed in
  Equation~\eqref{eq:set-pushforward}.
\end{lemma}

\begin{proof}
  Let $\mu$ be a given {nonnegative} measure on $\VV$. For any $j \in \mathbb{J}$, we
  denote by $j_{\vert \supp \np{\mu}}$ the restriction of the mapping $j$ to the subset
  $\supp \np{\mu}\subset \VV$. For any $\dual \in \YY$, we have that
  \begin{align*}
    j_{\star} \mu \np{\dual}
    & = \mu \bp{ j^{-1}\np{\dual}}
      \tag{by definition~\eqref{eq:def_push_forward} of a pushforward measure}
    \\
    & = \mu \Bp{ \bp{j^{-1}\np{\dual} \cap {\supp \np{\mu}}} \cup
      \bp{j^{-1}\np{\dual} \cap \np{\supp \np{\mu}}^{c}}}
    \\
    & = \mu \bp{ j^{-1}\np{\dual} \cap \supp \np{\mu}} + \underbrace{\mu
      \bp{j^{-1}\np{\dual} \cap \np{\supp \np{\mu}}^{c}}}_{=0}
    \\
    & = \mu \bp{ j_{\vert \supp \np{\mu}}^{-1}\np{\dual}}
    \\
    & = \Bp{ \bp{j_{\vert \supp \np{\mu}}}_{\star} \mu} \np{\dual}
      \tag{by \eqref{eq:def_push_forward}}
      \eqfinp
  \end{align*}
  Thus, defining
  $ \mathbb{J}_{\vert \supp \np{\mu}}= \nset{j_{\vert
      \supp \np{\mu}}}{j\in \mathbb{J}}$,
  we get that
  \begin{equation*}
    \cardinal{ \nset{ j_{\star}\mu}{j \in
        \mathbb{J}}}
    =
    \cardinal{ \nset{ {\np{j_{\vert \supp \np{\mu}}}}_{\star}\mu}{
        j \in \mathbb{J}}}
    \le
    \cardinal{ \mathbb{J}_{\vert \supp \np{\mu}}}
    \le
    \cardinal{ \YY^{\supp \np{\mu}}}
    =
    \cardinal{\YY}^{\cardinal{\supp \np{\mu}}}
    \eqfinp
  \end{equation*}
  This ends the proof.
\end{proof}

We now bound the cardinality of sets of forward and backward mappings.

\begin{lemma}
  \label{lem:boundXhatX}
  Let $\nseqa{\generalFunctionSet_{\kindex}}{\kindex \in \NN}$ be a given sequence where,
  for each ${\kindex \in \NN}$, the set
  $\generalFunctionSet_{\kindex}\subset \Mappings{\overline{\PRIMAL}}{\overline{\PRIMAL}}$
  is a finite set of self-mappings on the set $\overline{\PRIMAL}$. The sets
  $\generalFunctionSet_{\kindex}$, for all ${\kindex \in \NN}$, are assumed to be either
  all \ForwardMappings\ sets or all \BackwardMappings\ sets.
  We define the sequence $\nseqa{\generalFunctionBisSet_{\kindex}}{\kindex \in \NN}$,
  where, for each ${\kindex \in \NN}$, the set
  $\generalFunctionBisSet_{\kindex} \subset
  \Mappings{\Delta\np{\overline{\PRIMAL}}}{\Delta\np{\overline{\PRIMAL}}}$ is a finite set
  of self-mappings (on the set $\Delta\np{\overline{\PRIMAL}}$) given by
  \begin{equation}
    \forall \kindex \in \NN \eqsepv
    \generalFunctionBisSet_{\kindex} = \Renormalization \circ
    (\generalFunctionSet_{\kindex})_{\star}
    \eqfinp
    \label{eq:def_set_calR_Pushforward}
  \end{equation}
  Then, for any
  $\belief_0\in \Delta(\PRIMAL)$, we have the following bound
  \begin{equation}
    \forall n \in \NN \eqsepv
    \BBcardinal{ \bigcup_{\kindex = 0}^{n} {\generalFunctionBisSet}_{0:\kindex}(\belief_0)}
    \le (1+\cardinal{\PRIMAL})^{\cardinal{\supp(\belief_0)}}
    \eqfinv
    \label{eq:composeFk}
  \end{equation}
  where
  ${\generalFunctionBisSet}_{0:\kindex} = \generalFunctionBisSet_{\kindex} \circ \dots
  \circ \generalFunctionBisSet_{0}$ is defined in Equation~\eqref{eq:set-notations}.
\end{lemma}

\begin{proof}
  For all $\kindex\in \NN$, we have
  \begin{align*}
    {\generalFunctionBisSet}_{0:\kindex}(\belief_0)
    &= ({\generalFunctionBisSet}_{\kindex} \circ {\generalFunctionBisSet}_{\kindex-1}
      \circ \cdots \circ {\generalFunctionBisSet}_{0})(\belief_0)
      \tag{by Equation~\eqref{eq:F-0k}}
    \\
    & =\bp{ \Renormalization \circ \np{{\generalFunctionSet}_{\kindex}}_{\star} \circ \Renormalization \circ
      \np{{\generalFunctionSet}_{\kindex-1}}_{\star} \circ \cdots
      \circ \Renormalization \circ \np{{\generalFunctionSet}_{0}}_{\star}}(\belief_0)
      \tag{by Equation~\eqref{eq:def_set_calR_Pushforward}}
    \\
    &= \bp{ \Renormalization \circ \np{{\generalFunctionSet}_{\kindex}}_{\star} \circ
      \np{{\generalFunctionSet}_{\kindex-1}}_{\star} \circ \cdots
      \circ \np{{\generalFunctionSet}_{0}}_{\star}}(\belief_0)
      \intertext{by Lemma~\eqref{lem:composition_pushforward_equality}, as the sets
      ${\generalFunctionSet}_{\kindex}$ are, by assumption, either all  \ForwardMappings\
      sets or all \BackwardMappings\ sets,}
    &= \bp{  \Renormalization \circ \np{{\generalFunctionSet}_{\kindex} \circ
      {\generalFunctionSet}_{\kindex-1} \circ \cdots
      \circ {\generalFunctionSet}_{0}}_{\star}}
      (\belief_0)
      \intertext{as \( \np{{\generalFunctionSet}_{\kindex}}_{\star} \circ
      \np{{\generalFunctionSet}_{\kindex-1}}_{\star} \circ \cdots
      \circ \np{{\generalFunctionSet}_{0}}_{\star} = \np{{\generalFunctionSet}_{\kindex} \circ
      {\generalFunctionSet}_{\kindex-1} \circ \cdots
      \circ {\generalFunctionSet}_{0}}_{\star} \)
by definition~\eqref{eq:def_push_forward} of a pushforward measure,}
    &=  \Renormalization \bp{ \np{\generalFunctionSet_{0:\kindex}}_{\star}(\belief_0)}
      \eqfinp\tag{by Equation~\eqref{eq:F-0k}}
  \end{align*}
  Thus, we have, for all $n \in \NN$, $
    \BBcardinal{ \bigcup_{\kindex=0}^n  {\generalFunctionBisSet}_{0:\kindex}(\belief_0)} \le
    \BBcardinal{\bp{\bigcup_{\kindex=0}^n {\generalFunctionSet}_{0:\kindex}}_{\star}(\belief_0)}
    $, and the conclusion follows from Lemma~\ref{lem:restriction_counting_pushforward}
    with $\mathbb{J} = \bigcup_{\kindex=0}^n {\generalFunctionSet}_{0:\kindex}$,
    $\YY=\VV=\overline{\PRIMAL}$, and $\mu= \belief_0$.
\end{proof}

Note that Lemma~\ref{lem:boundXhatX} can be easily extended to cases with sequences
$\nseqa{\generalFunctionSet_{\kindex}}{\kindex \in \NN}$ of mixes of both
\ForwardMappings\ sets and \BackwardMappings\ sets as forward mappings are also
backward mappings by Equation~\eqref{eq:forward-to-backward}.
However, in the rest of this paper, we just need to consider non mixed sequences and thus we only need
Lemma~\ref{lem:boundXhatX}.

We now present a lemma on the conservation of the cardinality of the support of a
measure through a composition of sets of mappings, if we have conservation of the
cardinality for each individual set.

\begin{lemma}\label{lem:composeF}
  Let $\nseqa{\generalFunctionBisSet_{\kindex}}{\kindex \in \NN}$ be a sequence of sets of
  self-mappings on the set $\BB$ (i.e. for all $\kindex \in \NN$,
  $\generalFunctionBisSet_{\kindex} \subset
  \Mappings{\BB}{\BB}$) --- 
  where we recall that the set  $\BB=\Delta(\XX) \cup \na{\cemeteryBelief}$ is
  given by Equation~\eqref{eq:def_BB} ---    
  and assume that, for all
  ${\kindex \in \NN}$, we have that
  \begin{equation}
    \forall b \in \BB\eqsepv
    \quad \sum_{\fonctionprimalbis \in {\generalFunctionBisSet}_{\kindex}}
    \cardinal{\supp\bp{\fonctionprimalbis(b)_{\vert_{\PRIMAL}}}}
    \le \cardinal{\supp\np{b_{\vert_{\PRIMAL}}}}
    \eqfinp
    \label{eq:Fk-bound}
  \end{equation}
  Then, for any $\belief_0\in \Delta(\PRIMAL)$, we have the following bound
  \begin{equation}
    \forall \kindex\in \NN\eqsepv
    \bcardinal{ {\generalFunctionBisSet}_{0:\kindex}(\belief_0)\setminus \na{\cemeteryBelief}}
    \le
    \cardinal{\supp(\belief_0)}
    \eqfinv
    \label{eq:composeFk_bis}
  \end{equation}
  where
  ${\generalFunctionBisSet}_{0:\kindex}(\belief_0) = \generalFunctionBisSet_{\kindex}
  \circ \dots \circ \generalFunctionBisSet_{0} \np{\belief_0}$ is defined in
  Equation~\eqref{eq:F-0k}.
\end{lemma}

\begin{proof}
  Let a belief $\belief_0\in \Delta(\PRIMAL)$ be given. As a preliminary result we prove,
  by forward induction on $\kindex \in \NN$, that
  \begin{equation}
    \forall \kindex \in \NN
    \eqsepv
    \sum_{\belief \in {\generalFunctionBisSet}_{0:\kindex}\np{\belief_0}}
    \bcardinal{\supp\np{\belief_{\vert_{\PRIMAL}}}}
    \le \cardinal{\supp\np{b_0}}
    \label{eq:F0k-induction}
    \eqfinp
  \end{equation}
  First, we consider the case $k=0$. As
  $\generalFunctionBisSet_{0:0} = \generalFunctionBisSet_0$ the result follows from
  Equation~\eqref{eq:Fk-bound} used for $k=0$ and $\belief=\belief_0$. Second, we
  consider $\kindex > 0$, and, assuming that Equation~\eqref{eq:F0k-induction} is
  satisfied for $\kindex$, we prove that it is also satisfied for~$\kindex{+}1$ as
  follows:
  \begin{align*}
    \sum_{\belief \in {\generalFunctionBisSet}_{0:\kindex+1}\np{\belief_0}}
    \bcardinal{\supp\np{\belief_{\vert_{\PRIMAL}}}}
    &=
      \sum_{\fonctionprimalbis \in {\generalFunctionBisSet}_{0:\kindex+1}}
      \bcardinal{\supp\bp{{\fonctionprimalbis(\belief_0)}_{\vert_{\PRIMAL}}}}
      \tag{by~\eqref{eq:set-pushforward} {, combined with $\fonctionprimalbis \in
      {\generalFunctionBisSet}_{0:\kindex+1} \subset \Mappings{\BB}{\BB}$}  
      }
    \\
    &=
      \sum_{\fonctionprimalbis' \in {\generalFunctionBisSet}_{\kindex+1},
      \fonctionprimalbis'' \in {\generalFunctionBisSet}_{0:\kindex}}
      \BBcardinal{\supp\Bp{{\fonctionprimalbis'\bp{
      \fonctionprimalbis''(\belief_0)}}_{\vert_{\PRIMAL}}}}
      \tag{as $ {\generalFunctionBisSet}_{0:\kindex+1}=
      {\generalFunctionBisSet}_{\kindex+1}\circ {\generalFunctionBisSet}_{0:\kindex}$}
    \\
    &=
      \sum_{\fonctionprimalbis'' \in {\generalFunctionBisSet}_{0:\kindex}}
      \Bp{\sum_{\fonctionprimalbis' \in {\generalFunctionBisSet}_{\kindex+1}}
      \BBcardinal{\supp\Bp{{\fonctionprimalbis'\bp{
      \fonctionprimalbis''(\belief_0)}}_{\vert_{\PRIMAL}}}}}
    \\
    &\le
      \sum_{\fonctionprimalbis'' \in {\generalFunctionBisSet}_{0:\kindex}}
      \BBcardinal{\supp\bp{{\fonctionprimalbis''(\belief_0)}_{\vert_{\PRIMAL}}}}
      \tag{using Equation~\eqref{eq:Fk-bound} for $k+1$ and
      $\belief=\fonctionprimalbis''\np{\belief_0}$}
    \\
    &=
      \sum_{\belief \in {\generalFunctionBisSet}_{0:\kindex}(\belief_0)}
      \bcardinal{\supp\bp{{\belief}_{\vert_{\PRIMAL}}}}
      \tag{by~\eqref{eq:set-pushforward}}
    \\
    &\le \cardinal{\supp\np{b_0}}
      \tag{by induction assumption \eqref{eq:F0k-induction} on $\kindex$}
      \eqfinp
  \end{align*}
  We conclude that Equation~\eqref{eq:F0k-induction} is satisfied for all
  $\kindex \in \NN$.

  Now, we turn to the proof of Equation~\eqref{eq:composeFk_bis}. We make the following
  observation: if $b \in \Delta\np{\PRIMAL}$, then we have that
  $\cardinal{\supp\np{b_{\vert_{\PRIMAL}}}} \ge 1$ and if $b=\cemeteryBelief$ then
  $\cardinal{\supp\np{b_{\vert_{\PRIMAL}}}}=0$. Thus, we have that
  \begin{align}
    \cardinal{ {\generalFunctionBisSet}_{0:\kindex}(\belief_0) \setminus
    \na{\cemeteryBelief}}
    &=  \sum_{b \in {\generalFunctionBisSet}_{0:\kindex}(\belief_0) \setminus
      \na{\cemeteryBelief}} 1
    \\
    &\le
      \sum_{b \in {\generalFunctionBisSet}_{0:\kindex}(\belief_0) \setminus \na{\cemeteryBelief}}
      \cardinal{\supp\np{b_{\vert_{\PRIMAL}}}}
      \tag{as $\cardinal{\supp\np{b_{\vert_{\PRIMAL}}}}\ge 1$ for $b \in
      {\generalFunctionBisSet}_{0:\kindex}(\belief_0) \setminus \na{\cemeteryBelief}$}
    \\
    &=\sum_{b \in {\generalFunctionBisSet}_{0:\kindex}(\belief_0)}
      \cardinal{\supp\np{b_{\vert_{\PRIMAL}}}}
      \tag{as $\cardinal{\supp\np{{\cemeteryBelief}_{\vert_{\PRIMAL}}}}=0$}
    \\
    &\le\cardinal{\supp\np{b_0}}\tag{by \eqref{eq:F0k-induction}}
      \eqfinv
  \end{align}
  which gives Equation~\eqref{eq:composeFk_bis}. That concludes the proof.
\end{proof}

We now present a technical lemma.

\begin{lemma}
  \label{lem:hY}
  Let $\fonctionprimalbis \in \Mappings{\YY}{\VV}$ be a mapping from the set $\YY$ to the
  set $\VV$ and assume that the sets $\YY$ and $\VV$ are both finite. Let
  $\DualV \subset \VV$ be a subset of $\VV$. We define the mapping\footnote{ Note that the
    mapping $\fonctionprimalbis_{\DualV}$ is slightly different from
    $\Mapforward{\fonctionprimalbis}{\DualV}$. Indeed
    $\Mapforward{\fonctionprimalbis}{\DualV}$ are defined for self-mappings, whereas
    $\fonctionprimalbis_{\DualV}$ is defined for an extended codomain (set of
    destinations).} $\fonctionprimalbis_{\DualV}: \YY \to \VV \cup \na{\cemetery_{\VV}}$
  taking values in the extended set $\overline{\VV}=\VV \cup \na{\cemetery_{\VV}}$ as
  follows
  \begin{align}
    \fonctionprimalbis_{\DualV} :
    \dual \in \YY
    & \mapsto
      \begin{cases}
        \fonctionprimalbis\np{\dual}
        & \text{if} \quad
          \fonctionprimalbis\np{\dual}\in \DualV
          \eqfinv
        \\
        \cemetery_{\VV}
        & \text{elsewhere}
          \eqfinp
      \end{cases}
      \label{eq:definition_fonctionprimal_dual}
  \end{align}
  Then, for any {nonnegative} measure $\mu$ on the set $\YY$, we have that
  \begin{equation}
    \BBcardinal{\supp\Bp{\bp{\np{\fonctionprimalbis_{\DualV}}_{\star}\mu}_{\vert_{\VV}}}}
    \le \bcardinal{\supp\bp{\mu_{\vert_{\fonctionprimalbis^{-1}\np{\DualV}}}}}
    \label{eq:hY-bound}
    \eqfinp
  \end{equation}
  Moreover, for any finite family $\nseqa{\DualV_i}{i\in I}$ of pairwise disjoints subsets
  of $\VV$, we have that
  \begin{equation}
    \sum_{i\in I}
    \BBcardinal{\supp\Bp{\bp{\np{\fonctionprimalbis_{\DualV_i}}_{\star}\mu}_{
          \vert_{\VV}}}}
    \le
    \bcardinal{\supp\bp{\mu_{\vert_{\fonctionprimalbis^{-1}\np{\sqcup_{i\in I}
              \DualV_i}}}}}
    \eqfinv
    \label{a:hYi-bound}
  \end{equation}
  where $\sqcup$ is the union of disjoints sets.
\end{lemma}

\begin{proof}
  We prove~Equation~\eqref{eq:hY-bound}. Let $\mu \in \Delta(\YY)$ be given. First, we
  note that, if the set
  $\supp\bp{\bp{\np{\fonctionprimalbis_{\DualV}}_{\star}\mu}_{\vert_{\VV}}}$ is empty, the
  result is obvious. Second, we assume that
  $\supp\bp{\bp{\np{\fonctionprimalbis_{\DualV}}_{\star}\mu}_{\vert_{\VV}}}\not=\emptyset$
  and consider
  $\dualV \in \supp\bp{\bp{\np{\fonctionprimalbis_{\DualV}}_{\star}\mu}_{\vert_{\VV}}}$.
  Thus, $\dualV$ is restricted to belong to $\VV$ and, {by the definition of the support of
  a pushforward measure}, it must satisfy
  $\mu \bp{ \fonctionprimalbis_{\DualV}^{-1}(\dualV)}\not=0$. This implies that
  $ \fonctionprimalbis_{\DualV}^{-1}(\dualV)\not=\emptyset$ and, using the definition of
  $\fonctionprimalbis_{\DualV}$ (in Equation~\eqref{eq:definition_fonctionprimal_dual}),
  we obtain that $\dualV$ must belong to $\DualV$. We conclude that there must exist
  $\dual \in \fonctionprimalbis_{\DualV}^{-1}(\dualV)$ such that $\mu(\dual)\not=0$ which,
  combined with the fact that the mapping $\fonctionprimalbis_{\DualV}^{-1}$ coincides
  with the mapping $\fonctionprimalbis^{-1}$ on $\DualV$, gives that
  $\dual \in \fonctionprimalbis^{-1}(\dualV)\cap \supp\np{\mu}$.

  Now, consider the
  set-valued mapping
   $ \Gamma : \supp\bp{\bp{\np{\fonctionprimalbis_{\DualV}}_{\star}\mu}_{\vert_{\VV}}}
    \rightrightarrows \YY
    \eqsepv
    \dualV \mapsto
      \fonctionprimalbis^{-1}(\dualV)\cap \supp\np{\mu} $.
  By construction, the set-valued mapping $\Gamma$ takes values in the subsets of
  $\supp\np{\mu_{\vert_{\fonctionprimalbis^{-1}\np{\DualV}}}}$, and we have just
  proved that it takes values in the nonempty subsets of
  $\mu_{\vert_{\fonctionprimalbis^{-1}\np{\DualV}}}$. Moreover, the set-valued mapping
  $\Gamma$ is injective as we easily obtain that
  $(\fonctionprimalbis\circ \Gamma)(\dualV)=\dualV$ for all
  $\dualV \in
  \supp\bp{\bp{\np{\fonctionprimalbis_{\DualV}}_{\star}\mu}_{\vert_{\VV}}}$. Thus, the
  image of $\Gamma$ is a partition of a subset of
  $\supp\np{\mu_{\vert_{\fonctionprimalbis^{-1}\np{\DualV}}}}$ and we conclude that
  \begin{equation*}
    \bcardinal{\supp\bp{\bp{\np{\fonctionprimalbis_{\DualV}}_{\star}\mu}_{\vert_{\VV}}}}
    = \bcardinal{\Gamma\bp{\supp\bp{\bp{\np{\fonctionprimalbis_{\DualV}}_{\star}\mu}_{
            \vert_{\VV}}}}}
    \le \cardinal{\supp\np{\mu_{\vert_{\fonctionprimalbis^{-1}\np{\DualV}}}}}
    \eqfinv
  \end{equation*}
  which gives Equation~\eqref{eq:hY-bound}.

  \medskip

  Now, we turn to the proof of Inequality~\eqref{a:hYi-bound}. We successively have
  \begin{align*}
    \sum_{i\in I}
    \BBcardinal{\supp\Bp{\bp{\np{\fonctionprimalbis_{\DualV_i}}_{\star}\mu}_{
    \vert_{\VV}}}}
    & \le
      \sum_{i\in I}
      \bcardinal{\supp\bp{\mu_{\vert_{\fonctionprimalbis^{-1}\np{ \DualV_i}}}}}
      \tag{by~\eqref{eq:hY-bound} for each $i\in I$}
    \\
    &=
      \bcardinal{\supp\bp{\mu_{\vert_{\sqcup_{i\in I} \fonctionprimalbis^{-1}\np{
      \DualV_i}}}}} \\
    \intertext{(as the family  of subsets
    $\nseqa{\fonctionprimalbis^{-1}\np{\DualV_i}}{i\in I}$ is composed of  pairwise
    disjoints subsets as it was the case for the family $\nseqa{\DualV_i}{i\in I}$)}
    &=
      \bcardinal{\supp\bp{\mu_{\vert_{ \fonctionprimalbis^{-1}\np{
      \sqcup_{i\in I} \DualV_i}}}}}
      \tag{as $\fonctionprimalbis^{-1}\np{\sqcup_{i\in I} \DualV_i}=
      \sqcup_{i\in I}\fonctionprimalbis^{-1}\np{ \DualV_i}$}
      \eqfinv
  \end{align*}
  which concludes the proof.
\end{proof}

Lemma~\ref{lem:hY} shows that the cardinality of the support of a measure
decreases when the measure is transported by a pushforward measure induced by a mapping of
the form given by Equation~\eqref{eq:definition_fonctionprimal_dual}. A similar result
\begin{equation*}
  \forall \timeindex \in \timeset \eqsepv
  \forall \belief \in \BB \eqsepv
  \forall \controls \in \UU\eqsepv
  \sum_{\observer \in \OO}
  \bcardinal{\supp\bp{\beliefdynamics_{\timeindex} (\belief, \controls, \observer)}}
  \leq
  \bcardinal{\supp\np{\belief}}
  \eqfinv
\end{equation*}
is given in {\cite[Lemma 6.2]{littman_thesis}} but only
for the mappings $\nseqp{\beliefdynamics_{\timeindex}}{\timeindex \in \timeset}$
defined in Equation~\eqref{eq:beliefdynamics},
and with a proof not explicitly connected to pushforward measures.

We now present the final lemma necessary to prove Theorem~\ref{th:reachable_belief_bounds}.

\def\DualV{V}
\begin{lemma}
  \label{lem:second-bound}
  Let $\nseqa{\fonctionprimalbis^k}{\kindex \in \NN}$ be a sequence of self-mappings on
  the set $\overline\XX$ and, for all $\kindex \in \NN$, let
  $\nseqa{\DualP_i^k}{i\in I_k}$ be a finite family of two by two disjoints subsets of
  $\PRIMAL$
 . Let
  $\nseqa{\generalFunctionSet_{\kindex}}{\kindex \in \NN}$ be the sequence of
  self-mappings on the set~$\overline\XX$, of the following form
  \begin{equation}
    \forall \kindex \in \NN \eqsepv \generalFunctionSet_{\kindex} = \bset{
      \Mapforward{\fonctionprimalbis^k}{\DualP_i^k}}{ i \in I_k}
    \subset \Mappings{\overline{\PRIMAL}}{\overline{\PRIMAL}} 
    \eqsepv
    \label{eq:GGk_definition}
  \end{equation}
  where
  $\Mapforward{\fonctionprimalbis^k}{\DualP_i^k}: \overline{\PRIMAL} \to
  \overline{\PRIMAL}$ are built following
  Equation~\eqref{eq:definition_fonctionprimal_dual_forward}. Consider the sequence
  $\nseqa{\generalFunctionBisSet_k}{\kindex \in \NN}$ of sets of self-mappings on the set
  $\BB=\Delta(\XX) \cup \na{\cemeteryBelief}$, given, for all $\kindex \in \NN$, by
  $\generalFunctionBisSet_k = \Renormalization\circ (\generalFunctionSet_k)_{\star}$ and the
  associated sequence $\nseqp{\generalFunctionBisSet_{0:\kindex}}{\kindex \in \NN}$ as
  defined in Equation~\eqref{eq:set-notations}. Then, given
  $\belief_0\in \Delta(\PRIMAL)$, we have
  \begin{equation}
    \forall \kindex\in \NN\eqsepv
    \bcardinal{ {\generalFunctionBisSet}_{0:\kindex}(\belief_0)
      \setminus \na{\cemeteryBelief}}
    \le
    \cardinal{\supp(\belief_0)}
    \eqfinp
    \label{eq:composeFk_tri}
  \end{equation}
\end{lemma}
\begin{proof}
  The proof relies on Lemma~\ref{lem:hY} from which we obtain that the
  mappings~$\generalFunctionBisSet_k$ satisfy Equation~\eqref{eq:Fk-bound} for all
  $\kindex \in \NN$, and on Lemma~\ref{lem:composeF}.

  First, as a preliminary fact, we have that, for all $\mu \in \Delta\np{\barXX}$,
  $\supp \Bp{\bp{\Renormalization \np{\mu}}_{\vert \XX}} = \supp \np{ \mu_{\vert \XX}}$. Indeed,
  by~\eqref{eq:definition_renormalisation}, if $\mu\np{\XX} = 0$, then
  $
  \supp \Bp{\bp{\Renormalization \np{\mu}}_{\vert \XX}} = \supp \bp{\np{\cemeteryBelief}_{\vert \XX}}
  = \emptyset = \supp\bp{\mu_{\vert \XX}}
  $;
  whereas if $\mu\np{\XX} \neq 0$, then we have
  $
    \supp \Bp{\bp{\Renormalization \np{\mu}}_{\vert \XX}} = \supp \bp{\np{\frac{\mu_{\vert
            \XX}}{\mu\np{\XX}}, 0}_{\vert \XX}} = \supp \bp{\frac{\mu_{\vert
            \XX}}{\mu\np{\XX}}} = \supp \np{ \mu_{\vert \XX}}
  $.

  \medskip

  Second, we show that the
  mappings~$\generalFunctionBisSet_k$ satisfy Equation~\eqref{eq:Fk-bound} for all $\kindex
  \in \NN$.
  For that purpose, we fix $k\in \NN$, and $b \in \BB$, and we successively have
  \begin{align*}
    \sum_{\fonctionprimalbis \in {\generalFunctionBisSet}_{\kindex}}
    \bcardinal{\supp\bp{\fonctionprimalbis(b)_{\vert_{\PRIMAL}}}}
    &=
      \sum_{i \in I_k}
      \BBcardinal{\supp\Bp{\bp{\bp{\Renormalization \circ
      \np{\Mapforward{\fonctionprimalbis^k}{\DualP^k_i}}_{\star}}(b)}_{\vert_{\PRIMAL}}}}
      \tag{by definition of ${\generalFunctionBisSet}_{\kindex} = \Renormalization\circ (\generalFunctionSet_k)_{\star}$ and
      $\generalFunctionSet_k$ in~\eqref{eq:GGk_definition}}
    \\
    &=
      \sum_{i \in I_k}
      \bcardinal{\supp\Bp{\bp{\np{\Mapforward{\fonctionprimalbis^k}{\DualP^k_i}}_{\star}(b)
      }_{\vert_{\PRIMAL}}}}
    \tag{as, by the preliminary fact, $\forall \mu \in \Delta\np{\barXX}$,
     $\supp \Bp{\bp{\Renormalization \np{\mu}}_{\vert \XX}} = \supp \np{ \mu_{\vert \XX}} $}
\\
    &\le
      \bcardinal{\supp\bp{{\belief}_{\vert_{\fonctionprimalbis^{-1}\np{\sqcup_{i\in I_k}
      \DualP_i^k}}}}}
      \intertext{(by~\eqref{a:hYi-bound} in Lemma~\ref{lem:hY}, applied with  $\YY = \VV = \PRIMAL$ and
      $\DualV = \Primal$, $V_i = \DualP^k_i$ for $i\in I = I_k$)}  
    &\le
      \bcardinal{\supp\bp{{\belief}_{\vert_{\PRIMAL}}}}
      \tag{as $\fonctionprimalbis^{-1}\np{\sqcup_{i\in I_k} \DualP_i^k}\subset \PRIMAL$}
      \eqfinp
  \end{align*}
  Third, as the assumptions given in Equation~\eqref{eq:Fk-bound} are satisfied,
  the result follows by Lemma~\ref{lem:composeF}.
\end{proof}

\subsubsection{Proof of Lemmata~\ref{lem:tau_as_pushforward} and~\ref{lem:equ_BBR_FF}, and of
  Theorem~\ref{th:reachable_belief_bounds}}
\label{par:lemmata}
We now present the postponed proof of
Lemma~\ref{lem:tau_as_pushforward}, presented page~\pageref{lem:tau_as_pushforward}.

\begin{proof}[Proof of Lemma~\ref{lem:tau_as_pushforward}]
  Fix
  $\np{\controls, \observer} \in \UU \times \OO$,
  $\timeindex \in \timeset \setminus \na{\horizon}$, and $\belief \in \BB$,
  and then denote by
  $\Primal\subset \PRIMAL$ the subset
  $\Primal = \bp{\observerfunct_{\timeindex+1}^{\controls}}^{-1}\np{\observer}$.
  We need to prove Equation~\eqref{eq:equality_dynamics_belief_pushforward}, that is,
  to prove that we have
  \( 
  \beliefdynamics_{\timeindex}\np{\belief, \controls, \observer} = \Renormalization
  \circ \np{\pushforwardtransition^{\controls,\observer}_{\timeindex}}_{\star} \np{\belief}
  \).

  Using Equation~\eqref{eq:def_proba_observation}, and the definition of the subset $\Primal$, we have that
  \begin{equation}
    \ProbaObservation_{\timeindex+1}(\belief, \controls, \observer)
    = \belief \bp{ \np{
        \observerfunct_{\timeindex+1}^{\controls} \circ \dynamics_{\timeindex}^{\controls}}^{-1}
      \np{\observer}}
    = \belief \bp{ \np{\dynamics_{\timeindex}^{\controls}}^{-1}(\Primal)}
    \label{eq:Qreformule}
    \eqfinp
  \end{equation}
  Now, using the expression of $\beliefdynamics_{\timeindex}$ in
  Equation~\eqref{eq:beliefdynamics} combined with Equation~\eqref{eq:Qreformule} and the
  definition of $\Primal$, we obtain, for all $\primal \in \overline{\PRIMAL}$,
  that
  \begin{align}
    \beliefdynamics_{\timeindex}(\belief, \controls, \observer) \np{\primal}
    &=
      \begin{cases}
        \displaystyle
        \frac{\belief\bp{\np{\dynamics_{\timeindex}^{\controls}}^{-1} \np{\primal}}
          \findi{\Primal}\np{\primal}}
        {  \belief \bp{ \np{\dynamics_{\timeindex}^{\controls}}^{-1}(\Primal)}}
        &\text{ if }
        \belief \bp{ \np{\dynamics_{\timeindex}^{\controls}}^{-1}(\Primal)}
        \neq 0 \eqfinv\\
        0
        &\text{ otherwise.}
      \end{cases}
  \end{align}
  Then, Equation~\eqref{eq:equality_dynamics_belief_pushforward} follows from
  Lemma~\ref{lem:ext_pushforward_renorm_result} applied with the mapping
  $\fonctionprimalbis = \dynamics_{\timeindex}^{\controls}$ and with the subset
  $X=\bp{\observerfunct_{\timeindex+1}^{\controls}}^{-1}\np{\observer}$, as we have
  \begin{equation}
    \label{eq:link_pushForward_forwardMapping}
    \pushforwardtransition_{\timeindex}^{\controls, \observer} =
    \Mapforward{\dynamics_{\timeindex}^{\controls}}{
    \np{\observerfunct_{\timeindex+1}^{\controls}}^{-1}\np{\observer}}
    \eqfinv
  \end{equation}
  where $\Mapforward{\dynamics_{\timeindex}^{\controls}}{
    \np{\observerfunct_{\timeindex+1}^{\controls}}^{-1}\np{\observer}}$ is defined in
  Equation~\eqref{eq:definition_fonctionprimal_dual_forward}.
  This ends the proof.
\end{proof}

We now present the postponed proof of Lemma~\ref{lem:equ_BBR_FF},
presented page~\pageref{lem:equ_BBR_FF}.

\begin{proof}[Proof of Lemma~\ref{lem:equ_BBR_FF}]
  As a preliminary resukt, we prove that, for all times
  $\timeindex \in \timesetNoHorizon$, we have that 
  \begin{equation}
    \setsBeliefDynamics_{0:\timeindex} = \Renormalization\circ
    \np{\setsPushForward_{0:\timeindex}}_{\star}
    \eqfinp
    \label{eq:compose_TT_R_compose_FF}
  \end{equation}
  First, using the definitions of the sets $\setsBeliefDynamics_{\timeindex}$ and
  $\setsPushForward_{\timeindex}$ in Equations~\eqref{eq:def_setsBeliefDynamics}
  and~\eqref{eq:def_setsPushForward}, and applying Lemma~\ref{lem:tau_as_pushforward} with
  the notation~\eqref{eq:set-pushforward}, we obtain that
  \begin{equation}
    \setsBeliefDynamics_{\timeindex} = \Renormalization\circ
    \np{\setsPushForward_{\timeindex}}_{\star}
    \eqsepv
    \forall \timeindex \in \timesetNoHorizon
    \eqfinp
    \label{eq:equ_TT_t_R_circ_FF_t}
  \end{equation}
  Second, for all times $\np{\timeindex, \timeindex'} \in \bp{\timesetNoHorizon}^2$ and for
  all ordered pairs of controls and observations $\np{\controls, \controls'} \in \UU^2$ and
  $\np{\observer, \observer'} \in \OO^2$, we can apply
  Lemma~\ref{lem:composition_pushforward_equality} on the mappings
  $\pushforwardtransition_{\timeindex}^{\controls, \observer}$ and
  $\pushforwardtransition_{\timeindex'}^{\controls', \observer'}$. Indeed, by
  Equation~\eqref{eq:link_pushForward_forwardMapping}, the mappings
  $\pushforwardtransition_{\timeindex}^{\controls, \observer}$ and
  $\pushforwardtransition_{\timeindex'}^{\controls', \observer'}$ are $X$-forward
  mappings. We hence have, by Equation~\eqref{eq:lem_composition_pushforward_pushbackward}, that
  $\Renormalization \circ \pushforwardtransition_{\timeindex}^{\controls, \observer} \circ
  \Renormalization \circ \pushforwardtransition_{\timeindex'}^{\controls', \observer'} =
  \Renormalization \circ \pushforwardtransition_{\timeindex}^{\controls, \observer} \circ
  \pushforwardtransition_{\timeindex'}^{\controls', \observer'}$. Combined with
  Equation~\eqref{eq:equ_TT_t_R_circ_FF_t}, this leads to Equation~\eqref{eq:compose_TT_R_compose_FF}.

  Now, let $\belief_{0} \in \Delta \np{\XX}$. We prove
  by induction on $\timeindex$ that we have
  \begin{equation}
    \label{eq:inductionB}
    \BBR_{\timeindex+1} \np{\belief_0}
    =
    \setsBeliefDynamics_{0:\timeindex} \np{\belief_0}
    =
    \Renormalization \circ
    \bp{\setsPushForward_{0:\timeindex}}_{\star}\np{\belief_0}
    \eqsepv
    \forall \timeindex \in \timesetNoHorizon
    \eqfinp
  \end{equation}

  First, by Definition~\ref{def:reachable_space} of the set of reachable
  beliefs, we have that 
  \def\myop#1#2{\mathrel{\overset{\mbox{\normalfont\tiny\sffamily #1}}{#2}}}
  \begin{equation*}
    \BBR_{1} \np{\belief_0}
    \myop{\eqref{eq:def_BBR_t}}{=}
    \beliefdynamics_{0} \bp{\na{\belief_0}, \UU, \OO}
    \myop{\eqref{eq:def_setsBeliefDynamics}}{=}
    \setsBeliefDynamics_{0}\np{\belief_0}
    \myop{\eqref{eq:equ_TT_t_R_circ_FF_t}}{=}
    \Renormalization \circ
    \bp{\setsPushForward_{0}}_{\star}\np{\belief_0}
    \eqfinv
  \end{equation*}
  i.e. Equation~\eqref{eq:inductionB} stands at time $0$.
  Second, assuming Equation~\eqref{eq:inductionB} is true for $\timeindex \in
  \timesetNoHorizon$, we successively have
  \begin{equation*}
    \BBR_{\timeindex+2} \np{\belief_0}
    \myop{\eqref{eq:def_BBR_t}}{=}
    \beliefdynamics_{\timeindex+1} \bp{\BBR_{\timeindex+1} \np{\belief_0}, \UU, \OO}
    \myop{\eqref{eq:def_setsBeliefDynamics}}{=}
    \setsBeliefDynamics_{\timeindex+1}\bp{\BBR_{\timeindex+1} \np{\belief_0}}
    \myop{\eqref{eq:inductionB}}{=}
      \setsBeliefDynamics_{\timeindex+1} \circ
      \setsBeliefDynamics_{0:\timeindex} \np{\belief_0}
    \myop{\eqref{eq:set-notations}}{=}
      \setsBeliefDynamics_{0:\timeindex+1} \np{\belief_0}
    \myop{\eqref{eq:compose_TT_R_compose_FF}}{=}
    \Renormalization \circ
    \bp{\setsPushForward_{0:\timeindex+1}}_{\star}\np{\belief_0}
    \eqfinv
  \end{equation*}
  giving Equation~\eqref{eq:inductionB} for $\timeindex+1$.

  Finally, Equation~\eqref{eq:equ_BBR_FF} comes from the
  definition~\eqref{eq:def_BBR_t_tprime} of~$\BBR_{\ic{1,\horizon}}$,
the
definitions~\eqref{eq:def_setsBeliefDynamics_total}--\eqref{eq:def_setsPushforward_total}
of the sets $\setsBeliefDynamics$ and $\setsPushForward$, 
    and the previously established Equation~\eqref{eq:inductionB}.
\end{proof}

We can now give the detailed proof of Theorem~\ref{th:reachable_belief_bounds},
presented page~\pageref{th:reachable_belief_bounds}.

\begin{proof}[Proof of Theorem~\ref{th:reachable_belief_bounds}]

  Let $\belief_0 \in \Delta\np{\XX}$ be given.

  First, we first prove the inequality
  $\cardinal{\BBR_{\ic{1,\horizon}}(\belief_0)} \leq \left( 1 + \cardinal{\XX}
  \right)^{\cardinal{\supp \np{\belief_0}}}$.
  Using Equation~\eqref{eq:equ_BBR_FF} in Lemma~\ref{lem:equ_BBR_FF},
  we have that 
  $\BBR_{\ic{1,\horizon}}(\belief_0) =
  \setsBeliefDynamics\np{\belief_0}$. We hence get that 
  \def\myop#1#2{\mathrel{\overset{\mbox{\normalfont\tiny\sffamily #1}}{#2}}}
  \begin{equation*}
    \cardinal{\BBR_{\ic{1,\horizon}} \np{\belief_0}}
    \myop{\eqref{eq:equ_BBR_FF}}{=}
    \cardinal{\setsBeliefDynamics\np{\belief_0}}
    \myop{\eqref{eq:def_setsBeliefDynamics_total}}{=}
    \BBcardinal{
      \bigcup_{i = 0}^{\horizon-1} \setsBeliefDynamics_{0:i} \np{\belief_0}}
    \myop{\eqref{eq:composeFk}}{\le}
    \left( 1 + \cardinal{\XX} \right)^{\cardinal{\supp \np{\belief_0}}}
    \eqfinp
  \end{equation*}
  The last inequality is given by Equation~\eqref{eq:composeFk}, obtained by applying
  Lemma~\ref{lem:boundXhatX}. As all the elements of~$\setsPushForward_{\timeindex}$ are
  of the form given in Equation~\eqref{eq:def_Cal_T_dpomdp}, the two sequences
  $\nseqa{\setsPushForward_{\timeindex}}{\timeindex \in \ic{0, \horizon-1}}$ and
  $\nseqa{\setsBeliefDynamics_{\timeindex}}{\timeindex \in \ic{0, \horizon-1}}$ satisfy the
  assumptions of Lemma~\ref{lem:boundXhatX} --- where the role of
  $\nseqa{\generalFunctionBisSet_{\kindex}}{\kindex \in \NN}$ is taken by
  $\nseqa{\setsBeliefDynamics_{\timeindex}}{\timeindex \in \ic{0, \horizon-1}}$ and the
  role of $\nseqa{\generalFunctionSet_{\kindex}}{\kindex \in \NN}$ is taken by
  $\nseqa{\setsPushForward_{\timeindex}}{\timeindex \in \ic{0, \horizon-1}}$ (the proof of
  Lemma~\ref{lem:tau_as_pushforward} states that set
  $\setsPushForward_{\timeindex}$ is an \ForwardMappings\ set).

  Second, we prove that we have
  \begin{equation}
    \bcardinal{\BBR_{\ic{1,\horizon}}(\belief_0)} \leq
    1+ \cardinal{\supp\np{\belief_0}} \cardinal{\UU}^{\cardinal{\timeset}}
    \label{eq:second-inequality-B}
    \eqfinv
  \end{equation}
  in order to obtain Inequality~\eqref{eq:bounds_dpomdp}. With the help of the
  representation of the beliefs evolution mappings given by
  Lemma~\ref{lem:tau_as_pushforward}, Inequality~\eqref{eq:second-inequality-B} is
  obtained as an application of Lemma~\ref{lem:second-bound},
  that we detail now.

  For each $\timeindex \in \timesetNoHorizon$
  and each $\controls_t \in \UU$ we introduce the sets $
    \setsBeliefDynamicsUpInd{\controls_{\timeindex}}_{\timeindex}
    = \bset{ \beliefdynamics_{\timeindex}\np{\cdot, \controls_{\timeindex}, \observer}}{
      \observer \in \OO}$ and $
    \setsPushForwardUpInd{\controls_t}_{\timeindex}
    =\bset{ \pushforwardtransition^{\controls_{\timeindex},
        \observer}_{\timeindex}}{ \observer \in \OO} $.
  Using set notations described in Equations~\eqref{eq:set-notations}, we obtain that
  $\setsBeliefDynamicsUpInd{\controls_t} _{\timeindex} = \Renormalization\circ
  \np{\setsPushForwardUpInd{\controls_t}_{\timeindex}}_{\star}$. Then, using the definition of
  $\BBR_{\timeindex}(\belief_0)$ in Equation~\eqref{eq:def_BBR_t}, we have that, for all
  time $\timeindex \in \timeset\setminus\na{0}$, 
  \begin{equation}
    \BBR_{\timeindex}(\belief_0)
    =
    \bigcup_{\controls_{0:\timeindex-1}\in \UU_{0:\timeindex-1}}
    \setsBeliefDynamicsUpInd{\controls_{\timeindex-1}}_{\timeindex-1} \circ
    \setsBeliefDynamicsUpInd{\controls_{\timeindex-2}}_{\timeindex-2} \circ \cdots \circ
    \setsBeliefDynamicsUpInd{\controls_0}_{0}(\belief_0)
    =
    \bigcup_{\controls_{0:\timeindex-1}\in \UU_{0:\timeindex-1}}
    \setsBeliefDynamicsUpInd{\controls_{0:\timeindex-1}}_{0:\timeindex-1}(\belief_0)
    \eqfinp
    \label{eq:BversusF_bis}
  \end{equation}
  For a fixed sequence~$\controls_{0:\timeindex}\in \UU_{0:\timeindex}$ of controls, the
  associated sequences of mappings
  $\nseqa{\setsBeliefDynamicsUpInd{\controls_t}_{\timeindex}}{\timeindex \in \timeset}$ and
  $\nseqa{\setsPushForwardUpInd{\controls_t}_{\timeindex}}{\timeindex \in \timeset}$
  satisfy the assumptions of Lemma~\ref{lem:second-bound} --- where the role of
  $\nseqa{\generalFunctionBisSet_{\kindex}}{\kindex \in \NN}$ is taken by
  $\nseqa{\setsBeliefDynamicsUpInd{\controls_t}_{\timeindex}}{\timeindex \in \ic{-1,
      \horizon}}$, the role of $\nseqa{\generalFunctionSet_{\kindex}}{\kindex \in \NN}$ is
  taken by
  $\nseqa{\setsPushForwardUpInd{\controls_t}_{\timeindex}}{\timeindex \in \ic{-1,
      \horizon}}$ and the role of the family of disjoint sets
  $\nseqa{\DualP_i^k}{i\in I_k}$ is taken by the family
  $\nseqa{\np{\observerfunct_{\timeindex}^{\controls}}^{-1}\np{\observer}}{\observer \in
    \OO, \timeindex \in \ic{-1, \horizon}}$ (the proof of
  Lemma~\ref{lem:tau_as_pushforward}
  states that the set
  $\setsPushForward_{\timeindex}$ is an \ForwardMappings\ set). We hence get that
  \begin{equation}
    \forall \timeindex\in \timesetNoHorizon \eqsepv
    \bcardinal{ {\setsBeliefDynamicsUpInd{\controls_{0:\timeindex}}}_{0:\timeindex}(\belief_0)
      \setminus \na{\cemeteryBelief}}
    \le
    \cardinal{\supp(\belief_0)}
    \eqfinp
    \label{eq:F-fixed-controls}
  \end{equation}
  Finally, we obtain
  \begin{align*}
    \bcardinal{\BBR_{\ic{1,\horizon}}(\belief_0)}
    &= \BBcardinal{\bigcup_{\timeindex=1}^{\horizon}\bp{\BBR_{\timeindex}\np{\belief_0}}}
      \tag{using Equation~\eqref{eq:def_BBR_t_tprime}}
    \\
    &\leq 1+ \BBcardinal{\bigcup_{\timeindex=1}^{\horizon}\bp{\BBR_{\timeindex}\np{\belief_0}
      \setminus \na{\cemeteryBelief}}}
      \tag{by removing $\cemeteryBelief$ from $\BBR_{\timeindex}\np{\belief_0}$ for all
      $\timeindex$}
    \\
    &= 1+ \BBcardinal{ \bigcup_{\timeindex=0}^{\horizon-1}
      \bigcup_{\controls_{0:\timeindex}\in \UU_{0:\timeindex}}
      \bp{\setsBeliefDynamicsUpInd{\controls_{0:\timeindex}}_{0:\timeindex}(\belief_0)
      \setminus \na{\cemeteryBelief}}}
      \tag{using~Equation~\eqref{eq:BversusF_bis}}
    \\
    &\leq 1 + \sum_{\timeindex=0}^{\horizon-1}
      \sum_{\controls_{0:t}\in \UU_{0:t}}
      \bcardinal{
      \bp{
      \setsBeliefDynamicsUpInd{\controls_{0:\timeindex}}_{0:\timeindex}(\belief_0)\setminus
      \na{\cemeteryBelief}}
      }
      \tag{as $\cardinal{A \cup B} \leq \cardinal{A} + \cardinal{B}$}
    \\
    &\leq 1+ \sum_{\timeindex=0}^{\horizon-1}
      \sum_{\controls_{0:t}\in \UU_{0:t}}
      \cardinal{ \supp(\belief_0)}
      \tag{using Equation~\eqref{eq:F-fixed-controls}}
    \\
    &\leq 1+ \sum_{\timeindex=0}^{\horizon-1}
      \cardinal{\UU}^{\timeindex+1}
      \cardinal{ \supp(\belief_0)}
      \tag{as $\UU_{0:t} = \UU^{\timeindex+1}$}
    \\
    &\leq 1+
      \cardinal{\UU}\Bp{\frac{\cardinal{\UU}^{\horizon}-1}{\cardinal{\UU} - 1}}
      \cardinal{ \supp(\belief_0)}
    \tag{as $\sum_{i=0}^N x^i = \frac{x^{N+1}-1}{x-1}$ for $x\neq 1$}
    \\
    &\leq 1 +  \cardinal{\UU}^{\cardinal{\timeset}} \cardinal{ \supp(\belief_0)}
      \tag{as $\cardinal{\timeset} = \horizon + 1$ and $\cardinal{\UU} > 1$}
      \eqfinp
  \end{align*}
  We have established the Inequality~\eqref{eq:second-inequality-B}, and this concludes the
  proof.
\end{proof}

\subsection{Complementary result on \MdpomdpFunctionSet s}
\label{sect:complements-det-pomdp}
In this subsection, we present complementary results on \MdpomdpFunctionSet s by applying
the framework presented in Appendix~\ref{sect:technical_lemmatas_pushforward}. We notably
apply the notion of forward and backward mappings, presented in
Equations~\eqref{eq:definition_fonctionprimal_dual_forward}
and~\eqref{eq:definition_fonctionprimal_dual_backward}, and the notion of pushforward
measures, defined in Equation~\eqref{eq:def_push_forward}. First,
in~\S\ref{subsect:properties_of_mdpomdp_function_sets}, we present and prove the lemmata
used in the proofs of Proposition~\ref{prop:dynamics_of_dpomdp_separated_ensure_mdpomdp}
and Theorem~\ref{th:bound_belief_space_mdpomdp} presented in Sect.~\ref{sect:mdpomdp}. Second,
in~\S\ref{subsect:example_mdpomdp}, we present a few examples of \Mdpomdp s.

\subsubsection{Properties of \MdpomdpFunctionSet s}
\label{subsect:properties_of_mdpomdp_function_sets}

\begin{lemma}
  \label{lem:separated_mappings_to_cemetery_separated_mapping}
  Let $\GG$ be an \MBackwardMappings\ set
  as in Definition~\ref{def:BackwardAndForwardMappingSets}.
  If $\mathbb{M}$ is a \separatedMappingSet,
  then $\GG$ is a \MdpomdpFunctionSet.
\end{lemma}

\begin{proof}
  Let $\generalFunction_1$ and $\generalFunction_2$ be two mappings in $\GG$. In order to
  prove that $\GG$ is a \MdpomdpFunctionSet, using
  Definition~\ref{def:cemeterySeparation_mapping_pair}, we need to prove that the
  restrictions of the two mappings $\generalFunction_1$ and $\generalFunction_2$ on the
  subset $A=\generalFunction_1^{-1}(\XX) \cap \generalFunction_2^{-1}(\XX)$ are separated.
  Using the property of the set $\GG$, there exist $m^1 \in \mathbb{M}$ (resp.
  $m^2 \in \mathbb{M}$) and $\Primal_1\subset \XX$ (resp. $\Primal_2\subset \XX$) such
  that $\generalFunction_1=\Mapbackward{m^1}{\Primal_1}$ (resp.
  $\generalFunction_2=\Mapbackward{m^2}{\Primal_2}$). Combined with the definition of
  $\Mapbackward{m^1}{\Primal_1}$ in
  Equation~\eqref{eq:definition_fonctionprimal_dual_backward}, this gives that
  $\generalFunction_1^{-1}(\XX) = (m^1)^{-1}(\Primal_1)$ (resp.
  $\generalFunction_2^{-1}(\XX) = (m^2)^{-1}(\Primal_2)$). We therefore obtain the
  equality $A=(m^1)^{-1}(\Primal_1) \cap (m^2)^{-1}(\Primal_2)$.

  First, if the set $A$ is empty, it is immediate to prove that $\generalFunction_1$ and
  $\generalFunction_2$ are \cemeterySeparated. Second, assuming that $A$ is not empty and
  using again the fact that $\generalFunction_1=\Mapbackward{m^1}{\Primal_1}$, we obtain
  that $\generalFunction_1$ coincides with $m^1$ on the set $A$, and in the same way we
  obtain that $\generalFunction_2$ coincides with $m^2$ on the set $A$.

  Now, as $m^1$ and $m^2$ belong to a \separatedMappingSet, they are separated mappings,
  and therefore their restrictions to $A$ are also separated. We conclude that the
  restrictions of $\generalFunction_1$ and $\generalFunction_2$ on the subset
  $A=\generalFunction_1^{-1}(\XX) \cap {\generalFunction_2}^{-1}(\XX)$ are separated. This
  ends the proof.
\end{proof}

A direct consequence of Lemma~\ref{lem:separated_mappings_to_cemetery_separated_mapping}
is the following Corollary~\ref{cor:separated_familly_and_sets_leads_to_union_separated}.
\begin{corollary}
  \label{cor:separated_familly_and_sets_leads_to_union_separated}
  Let $\nseqa{\mathbb{M}_\kindex}{\kindex \in \NN}$ be a sequence of sets of
  self-mappings on the set~$\overline\XX$.
  Let $\nseqa{\generalFunctionSet_{\kindex}}{\kindex \in \NN}$ be a sequence
  of sets of self-mappings on the set $\overline\XX$, such that, for all $k\in \NN$,
  $\generalFunctionSet_{\kindex}$ is an \MkBackwardMappings{k}\ set.
  If the set
  $\cup_{\kindex \in \NN} \bp{\mathbb{M}_k\circ \mathbb{M}_{k-1}
    \circ \cdots \circ  \mathbb{M}_0}$ of mappings is
  a \separatedMappingSet, then the set
  $\cup_{\kindex \in \NN} \bp{\generalFunctionSet_k\circ \generalFunctionSet_{k-1}
    \circ \cdots \circ  \generalFunctionSet_0}$ is
  a \cemeterySeparated\ mapping set.
\end{corollary}

\begin{proof}
  Let $\generalFunctionSet_1$ and $\generalFunctionSet_2$ be respectively an
  \MkBackwardMappings{1}\ set and an \MkBackwardMappings{2}\ set.
  Then, we have that
  \begin{align}
    \generalFunctionSet_1 \circ \generalFunctionSet_2
    &= \bset{ g_1 \circ g_2 }{g_1 \in  \generalFunctionSet_1 \text{ and } g_2 \in \generalFunctionSet_2}
      \nonumber
      \tag{by Notation~\eqref{eq:set-mappings-composition} for composition}
    \\
    &\subset \bset{ \Mapbackward{m^1}{\Primal_1} \circ \Mapbackward{m^2}{\Primal_2}}
      {m^1 \in \mathbb{M}_1\eqsepv m^2 \in \mathbb{M}_2\eqsepv \Primal_1\subset  \PRIMAL
      \eqsepv \Primal_2 \subset  \PRIMAL}
      \tag{by property~\eqref{MBackwardMappingsDef_a} of a \MBackwardMappings\ set}
      \nonumber
    \\
    &\subset \bset{
      \Mapbackward{(m^1\circ m^2)}
      {\Primal_2 \cap (m^2)^{-1}(\Primal_1)}}
      {m^1 \in \mathbb{M}_1\eqsepv m^2 \in \mathbb{M}_2\eqsepv \Primal_1\subset  \PRIMAL
      \eqsepv \Primal_2 \subset  \PRIMAL}
      \tag{by property~\eqref{eq:stable_composition_backward_mappings}}
      \nonumber
    \\
    &\subset \bset{ m_{\Primal}}{ m \in \mathbb{M}_1\circ \mathbb{M}_2 \text{ and }  \Primal\subset  \PRIMAL}
      \eqfinp\nonumber
  \end{align}
  We have obtained that $\generalFunctionSet_1\circ \generalFunctionSet_2$ is a
  \MeBackwardMappings{\mathbb{M}_1\circ \mathbb{M}_2} set. Thus, if
  $\mathbb{M}_1\circ \mathbb{M}_2$ is a \separatedMappingSet, then the set
  $\generalFunctionSet_1\circ \generalFunctionSet_2$ is a \cemeterySeparated\
  mapping set by using
  Lemma~\ref{lem:separated_mappings_to_cemetery_separated_mapping}.  The end of
  the proof follows by induction on the number of compositions of sets, and by
  straightforward arguments when considering unions of \BackwardMappings\ sets.
\end{proof}

Before presenting bounds on the  cardinality of a \MdpomdpFunctionSet, we present
Lemma~\ref{lem:cardinality_mdpomdp_pushforward_family_no_delta}.

\begin{lemma}
  \label{lem:cardinality_mdpomdp_pushforward_family_no_delta}
  Let $\mathbb{J} \subset \Mappings{\PRIMAL}{\DUAL}$ be a set of mappings from the finite
  set $\PRIMAL$ to the finite set $\DUAL$.
  Assume that
  for all ordered pairs of mappings $\np{j, j'} \in \mathbb{J}^2$, if there exists
  $\primal \in \PRIMAL$ such that $j(\primal) = j'(\primal)$, then $j = j'$. Then, we have
  that
  \begin{equation}
    \cardinal{\mathbb{J}} \leq \cardinal{\DUAL} \eqfinp
    \label{eq:bound-sep}
  \end{equation}
\end{lemma}

\begin{proof}
  Fix $\overline{\primal} \in \PRIMAL$ and consider the evaluation mapping
  $\gamma_{\overline{\primal}}:\mathbb{J} \to \DUAL$ defined by
  $\gamma_{\overline{\primal}}\np{ j}=j\np{\overline{\primal}}$ for all
  $j \in \mathbb{J}$. The image $\gamma_{\overline{\primal}}\np{\mathbb{J}}$ of the set
  $\mathbb{J}$ by the mapping $\gamma_{\overline{\primal}}$ is the subset
  $\nset{ j\np{\overline{\primal}}}{ j \in \mathbb{J}}$ of $\DUAL$. First, the codomain of
  the mapping $\gamma_{\overline{\primal}}$ being the finite set $\DUAL$, we immediately
  get that
  \begin{equation}
    \bcardinal{\gamma_{\overline{\primal}}\np{\mathbb{J}}} \le \cardinal{\DUAL}
    \eqfinp
    \label{eq:codomainbound}
  \end{equation}
  Second, the mapping $\gamma_{\overline{\primal}}$ is injective. Indeed, using the
  assumption on the set $\mathbb{J}$, two distinct mappings $j$ and $j'$ in the set
  $\mathbb{J}$ must satisfy
  $\gamma_{\overline{\primal}}\np{j} = j\np{\overline{\primal}} \not=
  j'\np{\overline{\primal}} = \gamma_{\overline{\primal}}\np{j'}$. Thus, we must have the
  equality
  $ \cardinal{\mathbb{J}}= \bcardinal{\gamma_{\overline{\primal}}\np{\mathbb{J}}} $ which,
  combined with Equation~\eqref{eq:codomainbound}, gives Inequality~\eqref{eq:bound-sep},
  and concludes the proof.
\end{proof}

We now use the previous Lemma~\ref{lem:cardinality_mdpomdp_pushforward_family_no_delta} to
bound the cardinality of a \MdpomdpFunctionSet.
\begin{lemma}
  \label{lem:mdpomdp}
  Let be given a 
  \MdpomdpFunctionSet\
  $\generalFunctionSet$ of self-mappings on the set $\overline{\PRIMAL}=\PRIMAL\cup{\na{\cemetery}}$. Moreover, assume
  that, for all $\generalFunction \in \generalFunctionSet$,
  $\generalFunction \np{\cemetery}=\cemetery$. For any subsets $\Primal$ and $\Primal'$ of
  the set $\overline{\PRIMAL}$, we define
  $\restrictionFunctionSubsetSet{\generalFunctionSet}{\Primal}{\Primal'}$ as
  \begin{equation}
    \restrictionFunctionSubsetSet{\generalFunctionSet}{\Primal}{\Primal'} =
    \bset{\generalFunction \in \generalFunctionSet}{
      \generalFunction^{-1}\np{\PRIMAL} = \Primal, \generalFunction\np{\Primal} \subset
      \Primal'}
    \eqfinp
    \label{eq:def_restrictionFunctionSubsetSet}
  \end{equation}
  Then, we have that 
  \begin{equation}
    \label{eq:bound-Tau-X}
    \bcardinal{\restrictionFunctionSubsetSet{\generalFunctionSet}{\Primal}{\Primal'}}
    \begin{cases}
      \le \cardinal{\Primal'}
      & \text{if}\; \Primal \subset \PRIMAL\eqsepv
      \\
      = 0
      & \text{if}\; \Primal \cap \na{\cemetery} \not=\emptyset
      \eqfinp
    \end{cases}
  \end{equation}
\end{lemma}

\begin{proof} Fix $\Primal \subset \overline{\PRIMAL}$ and
  $\Primal' \subset \overline{\PRIMAL}$. First, we consider the case where
  $\Primal \cap \na{\cemetery} \not=\emptyset$. As we have assumed that
  $\generalFunction\np{\cemetery}=\cemetery$, for all
  $\generalFunction\in \generalFunctionSet$, we obtain
  that $\generalFunction\in \generalFunctionSet \implies$
  $ \generalFunction^{-1}\np{\PRIMAL}\cap \na{\cemetery}=\emptyset$. Thus, we conclude
  that
  $\cardinal{\restrictionFunctionSubsetSet{\generalFunctionSet}{\Primal}{\Primal'}} =
  \cardinal{\emptyset}=0$. Second, we consider the case where $\Primal \subset \PRIMAL$
  and consider the mapping
  \begin{equation}
    \Gamma: \restrictionFunctionSubsetSet{\generalFunctionSet}{\Primal}{\Primal'} \to
    {\Primal'}^\Primal \eqsepv
    \generalFunction \mapsto  \generalFunction_{\vert_{\Primal}}
    \eqfinp
    \label{eq:def_mapping_gamma}
  \end{equation}

  The mapping $\Gamma$ is injective. Indeed, if two mappings in
  $\restrictionFunctionSubsetSet{\generalFunctionSet}{\Primal}{\Primal'}$ have the same
  restriction on $\Primal$, they coincide on $\overline{\PRIMAL}$ as they are both constant
  on the set $\overline{\PRIMAL}\setminus \Primal$ with value $\cemetery$. We therefore
  obtain that
  \begin{equation}
    \bcardinal{\restrictionFunctionSubsetSet{\generalFunctionSet}{\Primal}{\Primal'}} =
    \bcardinal{\Gamma\np{
        \restrictionFunctionSubsetSet{\generalFunctionSet}{\Primal}{\Primal'}}}
    \eqfinp
    \label{eq:Gamma-inj}
  \end{equation}
  Now, the set
  $\generalFunctionSet'=\Gamma
  \np{\restrictionFunctionSubsetSet{\generalFunctionSet}{\Primal}{\Primal'}}$
  is a subset of mappings from $\Primal$ to $\Primal'$.
  As $\generalFunctionSet$ is a \MdpomdpFunctionSet, we obtain that $\generalFunctionSet'$
  is a separated set of mappings from $\Primal$ to $\Primal'$. Indeed, consider an ordered pair of
  mappings $\np{\generalFunction'_1, \generalFunction'_2} \in \generalFunctionSet'^2$ and
  assume that there exists $\primal \in \Primal$ such that
  $\generalFunction'_1(\primal) = \generalFunction'_2(\primal)$. Using the definition of
  $\generalFunctionSet'$, we have that $\generalFunction'_1(\primal)$ and
  $\generalFunction'_2(\primal)$ are both non equal to $\cemetery$. Moreover, there exists
  $\generalFunction_1$ and $\generalFunction_2$ in
  $\restrictionFunctionSubsetSet{\generalFunctionSet}{\Primal}{\Primal'}$ such that
  $\generalFunction'_1 = \Gamma\np{\generalFunction_1}$ and
  $\generalFunction'_2 = \Gamma\np{\generalFunction_2}$. Using again the definition of
  $\generalFunctionSet' = \Gamma
  \np{\restrictionFunctionSubsetSet{\generalFunctionSet}{\Primal}{\Primal'}}$ we obtain
  that $\generalFunction_1\np{x}=\generalFunction_2\np{x}\not=\cemetery$. Now, as
  $\generalFunctionSet$ is a \MdpomdpFunctionSet, we obtain that the two mappings
  $\generalFunction_1$ and $\generalFunction_2$ coincide on $\Primal$ since they both do
  not take the value $\cemetery$ on $\Primal$. We conclude that their restrictions on
  $\Primal$, the mappings $\generalFunction'_1$ and $\generalFunction'_2$, coincide. Using
  Lemma~\ref{lem:cardinality_mdpomdp_pushforward_family_no_delta} in~\S\ref{sect:complements-det-pomdp}, we obtain that
  \begin{equation*}
    \bcardinal{\Gamma\np{\restrictionFunctionSubsetSet{\generalFunctionSet}{\Primal}{
          \Primal'}}}
    \le \cardinal{\Primal'}
    \eqfinv
  \end{equation*}
  which, combined with Equation~\eqref{eq:Gamma-inj}, gives
  Equation~\eqref{eq:bound-Tau-\Primal}. This concludes the proof.
\end{proof}

We now present the postponed proof of
Proposition~\ref{prop:dynamics_of_dpomdp_separated_ensure_mdpomdp}, presented in
page~\pageref{prop:dynamics_of_dpomdp_separated_ensure_mdpomdp}.

\begin{proof}[Proof of
Proposition~\ref{prop:dynamics_of_dpomdp_separated_ensure_mdpomdp}]
  The proof of Proposition~\ref{prop:dynamics_of_dpomdp_separated_ensure_mdpomdp} is
  a direct consequence of
  Corollary~\ref{cor:separated_familly_and_sets_leads_to_union_separated}.

  We assume that the set
  $\bigcup_{\timeindex \in \timeset} \dynamics^{\UU^{\timeindex+1}}_{0:\timeindex} =
  \nset{\dynamics_{0:\timeindex}^{\controls_{0:\timeindex}}}{\forall \timeindex \in
    \timesetNoHorizon, \forall \controls_{0:\timeindex} \in \UU^{\timeindex+1}}$ of the
  composition of the evolution mappings of Problem~\eqref{eq:dpomdp_general_formulation}
  is a \separatedMappingSet. We then prove that
  Problem~\eqref{eq:dpomdp_general_formulation} is a \Mdpomdp.

  First, for all time $\timeindex$ and for all ordered pairs
  $\np{\controls, \observer} \in \UU \times \OO$, we have
  $\pushforwardtransition^{\controls, \observer}_{\timeindex} =
  \Mapforward{\dynamics_{\timeindex}^{\controls}}{
    \np{\observerfunct_{\timeindex+1}^{\controls}}^{-1}\np{\observer}}$
  (see Equation~\eqref{eq:link_pushForward_forwardMapping}).
  Thus, by
  Equation~\eqref{eq:forward-to-backward}, there exists $X \subset \XX$ such that
  $\pushforwardtransition^{\controls, \observer}_{\timeindex} =
  \Mapbackward{\dynamics_{\timeindex}^{\controls}}{X}$. Hence,
  $\setsPushForward_{\timeindex}$ is of the same form as in
  Equation~\eqref{eq:GGk_definition}, with the role of set
  $\generalFunctionBisSet_{\kindex}$ taken by $\na{\dynamics_{\timeindex}^{\UU}}$.

  We hence have that
  $\setsPushForward = \bigcup_{\timeindex \in \timeset} \setsPushForward_{0:\timeindex}$
  is a \MdpomdpFunctionSet\ by
  Corollary~\ref{cor:separated_familly_and_sets_leads_to_union_separated} --- where the role
  of $\nseqa{\generalFunctionSet_{\kindex}}{\kindex \in \NN}$ is taken by
  $\nseqa{\setsPushForward_{\timeindex}}{\timeindex \in \timesetNoHorizon}$, and the role
  of $\nseqa{\generalFunctionBisSet_{\kindex}}{\kindex \in \NN}$ is taken by
  $\nseqa{\dynamics_{\timeindex}^{\UU}}{\timeindex \in \timesetNoHorizon}$.

  Therefore, as $\setsPushForward$ is a
  \MdpomdpFunctionSet, Problem~\eqref{eq:dpomdp_general_formulation} is a \Mdpomdp.
\end{proof}

We now present the postponed proof of
Theorem~\ref{th:bound_belief_space_mdpomdp}, presented 
page~\pageref{th:bound_belief_space_mdpomdp}.

\begin{proof}[Proof of Theorem~\ref{th:bound_belief_space_mdpomdp}]
  {Let $X \subset \XX$.}
  We start by giving preliminary bounds on
  $\BBcardinal{\bp{\Renormalization\circ \np{ \restrictionsetsPushForward{X}{\XX}}_{\star}}
    \np{\belief_0} \setminus \na{\cemeteryBelief}}$, where
  $\restrictionsetsPushForward{X}{\XX}$ is defined by
  Equation~\eqref{eq:def_restrictionFunctionSubsetSet},
  i.e.
  \[
    \restrictionFunctionSubsetSet{\setsPushForward}{X}{\XX} =
    \bset{\pushforwardtransition \in \setsPushForward}{
      \pushforwardtransition^{-1}\np{\XX} = X,
      \pushforwardtransition\np{\Primal} \subset \XX}
    \eqfinv
  \]
  where $\setsPushForward$ is defined in Equation~\eqref{eq:def_setsPushForward}.
  We consider three cases depending on the cardinality of
  the subset~$\Primal$.
  \begin{subequations}
    \label{eq:cardX}
    \begin{enumerate}
    \item When $\cardinal{\Primal}=0$, we have that $\Primal = \emptyset$ and
      $\bp{\Renormalization\circ \np{\restrictionsetsPushForward{\emptyset}{\PRIMAL}}_{\star}}
      \np{\belief_0} \setminus \na{\cemeteryBelief}=\emptyset$, and thus
      \begin{equation}
        \BBcardinal{\bp{ \Renormalization \circ \np{ \restrictionsetsPushForward{X}{\XX}}_{\star}}
          \np{\belief_0} \setminus \na{\cemeteryBelief} }=0
        \eqfinp
        \label{eq:cardXeq0}
      \end{equation}
    \item When $\cardinal{\Primal}=1$, we have that
      $\bp{\Renormalization\circ \np{\restrictionsetsPushForward{X}{\XX}}_{\star}} \np{\belief_0}
      \setminus \na{\cemeteryBelief} \subset \bset{\delta_{x} }{x\in \PRIMAL}$, as the
      only probability distributions of $\Delta\np{\XX}$ with support of cardinality at
      most~$1$ are the Dirac measures $\bset{\delta_{x} }{x\in \PRIMAL}$
      and thus
      \begin{equation}
        \BBcardinal{\bp{\Renormalization\circ \np{ \restrictionsetsPushForward{X}{\XX}}_{\star}}
          \np{\belief_0}\setminus \na{\cemeteryBelief}}
        \le \bcardinal{\bset{\delta_{x} }{x\in \PRIMAL}} = \cardinal{\PRIMAL}
        \eqfinp
        \label{eq:cardXeq1}
      \end{equation}
    \item For $\cardinal{\Primal}\ge 2$, we have by Lemma \ref{lem:mdpomdp} in
      Appendix~\ref{sect:technical_lemmatas_pushforward}, applied with $\GG = \FF$ (as
      $\FF$ is a \MdpomdpFunctionSet) that
      \begin{equation}
        \BBcardinal{\bp{\Renormalization\circ \np{ \restrictionsetsPushForward{X}{\XX}}_{\star}}
          \np{\belief_0}\setminus \na{\cemeteryBelief}}
        \le \bcardinal{\np{\restrictionsetsPushForward{X}{\XX}}_{\star}}
        \le \cardinal{\PRIMAL}
        \eqfinp
        \label{eq:cardXg2}
      \end{equation}
    \end{enumerate}
  \end{subequations}
  We have by Equation~\eqref{eq:equ_BBR_FF} that
  $\bcardinal{\BBR_{\ic{1,\horizon}}(\belief_0)} = \cardinal{\setsBeliefDynamics
    \np{\belief_0}}$.
  We now detail the cardinality of $\setsBeliefDynamics \np{\belief_0}$:
  \begin{align}
    \bcardinal{ \setsBeliefDynamics \np{\belief_0}\setminus \na{\cemeteryBelief}}
    &= \bcardinal{ \bp{\Renormalization\circ \np{\setsPushForward}_{\star}}\np{\belief_0}\setminus
      \na{\cemeteryBelief}}
      \nonumber
    \\
    &= \BBcardinal{ \bgp{\Renormalization\circ \Bp{ \bigcup_{\Primal \subset \PRIMAL}
      \restrictionsetsPushForward{X}{\XX}}_{\star}} \np{\belief_0}\setminus
      \na{\cemeteryBelief}}
      \tag{as $\bigcup_{\Primal \subset \PRIMAL}
      \restrictionsetsPushForward{X}{\XX} = \setsPushForward$}
      \nonumber
    \\
    &= \BBcardinal{\bigcup_{\Primal \subset \PRIMAL} \bp{\Renormalization\circ
      \np{ \restrictionsetsPushForward{X}{\XX}}_{\star}} \np{\belief_0}\setminus
      \na{\cemeteryBelief}}
      \nonumber
      \intertext{ as $\forall\np{\pushforwardtransition, \pushforwardtransition’} \in
      \bp{\setsPushForward}^2$, $\Renormalization\circ \bp{\pushforwardtransition \cup
      \pushforwardtransition’} = \Renormalization\circ\pushforwardtransition \cup {\cal
      R}\circ\pushforwardtransition’$,}
    \nonumber
    &= \BBcardinal{\bigcup_{\Primal \subset \supp\np{\belief_0}} \bp{\Renormalization\circ \np{
      \restrictionsetsPushForward{X}{\XX}}_{\star}} \np{\belief_0}\setminus
      \na{\cemeteryBelief}}
      \nonumber
      \intertext{
      as $\bp{ \Renormalization \circ \np{\setsPushForward_{\Primal \cap
      \supp\np{\belief_0} \to \PRIMAL}}_{\star}} \np{\belief_0} =
      \bp{ \Renormalization \circ \np{
      \restrictionsetsPushForward{X}{\XX}}_{\star}} \np{\belief_0}$
      by Equation~\eqref{eq:Backward-support-b0} in
      Lemma~\ref{lem:ext_pushforward_renorm_result},}
      \nonumber
    \\
    &\le
      \sum_{\Primal \subset \supp\np{\belief_0}} \BBcardinal{\bp{\Renormalization\circ \np{
      \restrictionsetsPushForward{X}{\XX}}_{\star}} \np{\belief_0}\setminus
      \na{\cemeteryBelief}}
      \nonumber
    \\
    &=
      \sum_{k \ge 0}
      \sum_{\substack{\Primal \subset \supp\np{\belief_0}
    \nonumber
    \\
    \cardinal{\Primal}=k}}
    \BBcardinal{\bp{\Renormalization\circ \np{ \restrictionsetsPushForward{X}{\XX}}_{\star}}
    \np{\belief_0}\setminus \na{\cemeteryBelief}}
    \label{eq:ineq_cardinal_BBR_b0_leq_sum_cardinal_restrictionsetsPushForward}
    \\
    &\le
      \cardinal{\PRIMAL}
      +
      \sum_{\substack{\Primal \subset \supp\np{\belief_0} \\ \cardinal{\Primal}\ge 2}}
    \cardinal{\PRIMAL}
    \tag{by Equations~\eqref{eq:cardX}}
    \nonumber
    \\
    &=
      \cardinal{\XX} + \bp{2^{\cardinal{\supp \np{\belief_0}}} -
      \cardinal{\supp \np{\belief_0}} - 1}
      \cardinal{\XX}
      \label{eq:bound_mdpomdp_proof}
      \eqfinv
  \end{align}
  where the last equality comes from the fact that
  $\bcardinal{\nset{\Primal \subset \supp\np{\belief_0}}{\cardinal{\Primal}\ge 2}}$ is
  given by
  \begin{multline*}
    \bcardinal{\nset{\Primal \subset \supp\np{\belief_0}}{\cardinal{\Primal}\ge 2}}
    =\\
    \underbrace{\bcardinal{\bset{\Primal \subset \PRIMAL}{\Primal \subset \supp\np{\belief_0}}}}_{
    2^{\cardinal{\supp \np{\belief_0}}}} -
    \underbrace{\bcardinal{
        \bset{\Primal \subset \supp\np{\belief_0}}{\cardinal{\Primal}=1}}}_{= \cardinal{\supp
      \np{\belief_0}}} -
    \underbrace{\bcardinal{\bset{\Primal \subset
          \supp\np{\belief_0}}{\cardinal{\Primal}=0}}}_{= 1}
    \eqfinp
  \end{multline*}

  We hence obtain that
  \def\myop#1#2{\mathrel{\overset{\mbox{\normalfont\tiny\sffamily #1}}{#2}}}
  \begin{equation*}
    \bcardinal{\BBR_{\ic{1,\horizon}}(\belief_0)}
    \myop{\eqref{eq:equ_BBR_FF}}{=}
    \cardinal{\setsBeliefDynamics \np{\belief_0}}
    \myop{\eqref{eq:bound_mdpomdp_proof}}{\leq}
    1 + \bp{2^{\cardinal{\supp \np{\belief_{0}} } } - \cardinal{\supp \np{\belief_0}}}
    \cardinal{\XX}
    \eqfinp
  \end{equation*}
  This ends the proof.
\end{proof}


\subsubsection{Example of \Mdpomdp s}
\label{subsect:example_mdpomdp}

A direct consequence of
Proposition~\ref{prop:dynamics_of_dpomdp_separated_ensure_mdpomdp} is that, if the
evolution mappings of a \dpomdp\ belong to a \separatedMappingSet, then the \dpomdp\ is a
\Mdpomdp. We now present an example of such evolution mappings. In the following, we use
the same notations as those presented in Problem~\eqref{eq:dpomdp_general_formulation}.

\begin{corollary}
  \label{cor:affine_dynamics_mdpomdp}
  Consider a \dpomdp~optimization problem given by
  Problem~\eqref{eq:dpomdp_general_formulation} which satisfies the finite sets
  Assumption~\ref{assumpt:pomdp_finite_sets}. Assuming that, for all time
  $\timeindex \in \timeset \setminus \na{\horizon}$, there exists a mappings~$g_{\timeindex}$
  such that, for all states $\states \in \XX \subset \RR^n$,
  \begin{equation}
    \dynamics_{\timeindex} \np{\states, \controls} = \states +
    g_{\timeindex}\np{\controls}
    \eqfinv
    \label{eq:affine_dynamics_def}
  \end{equation}
  then
  Problem~\eqref{eq:dpomdp_general_formulation} is a \Mdpomdp.
\end{corollary}

\begin{proof}
  We start by proving that the set of mappings
  $\cup_{\timeindex \in \timesetNoHorizon} \bp{\dynamics^{\UU^{\timeindex+1}}_{0:\timeindex}}$
  is a \separatedMappingSet. For that purpose, consider $\timeindex_1 \leq \timeindex_1'$ and $\timeindex_2 \leq \timeindex_2'$
  such that
  $\ic{\timeindex_1, \timeindex_1'} \subset \timesetNoHorizon$ and
  $\ic{\timeindex_2, \timeindex_2'} \subset \timesetNoHorizon$ and consider $\controls_{\timeindex_1:\timeindex_1'} \in
  \UU^{\timeindex_1'-\timeindex_1 + 1}$ and $\controls'_{\timeindex_2:\timeindex_2'} \in
  \UU^{\timeindex_2'-\timeindex_2 + 1}$ two
  sequences of controls, in order to obtain two mappings of the set $\cup_{\timeindex \in \timesetNoHorizon} \bp{\dynamics^{\UU^{\timeindex+1}}_{0:\timeindex}}$, $
  \dynamics_{\timeindex_1:\timeindex_1'}^{\controls_{\timeindex_1:\timeindex_1'}}: \XX
  \to \XX, \states \mapsto \states + \sum_{\timeindex \in \ic{\timeindex_1,
      \timeindex_1'}} g_{\timeindex} \np{\controls_{t}}$, and $
  \dynamics_{\timeindex_2:\timeindex_2'}^{\controls'_{\timeindex_2:\timeindex_2'}}:
  \XX \to \XX, \states \mapsto \states + \sum_{\timeindex \in \ic{\timeindex_2,
      \timeindex_2'}} g_{\timeindex} \np{\controls'_{t}}
  $.
  If there exists a state $\states \in \XX$ such that
  $\dynamics_{\timeindex_1:\timeindex_1'}^{\controls_{\timeindex_1:\timeindex_1'}}
  \np{\states} =
  \dynamics_{\timeindex_2:\timeindex_2'}^{\controls'_{\timeindex_2:\timeindex_2'}}
  \np{\states}$, then we have that
  $\sum_{\timeindex \in \ic{\timeindex_1, \timeindex_1'}}
  g_{\timeindex}\np{\controls_{t}} =
  \sum_{\timeindex \in \ic{\timeindex_2, \timeindex_2'}}
  g_{\timeindex} \np{\controls'_{t}}
  $ and thus the two mappings
  $\dynamics_{\timeindex_1:\timeindex_1'}^{\controls_{\timeindex_1:\timeindex_1'}}
  $ and $\dynamics_{\timeindex_2:\timeindex_2'}^{\controls'_{\timeindex_2:\timeindex_2'}}$
  coincide.
  Therefore, the set
  $\cup_{\timeindex \in \timeset} \bp{\dynamics^{\UU^{\timeindex+1}}_{0:\timeindex}}$
  of
  composition of the evolution mappings is a \separatedMappingSet, and we conclude by
  Proposition~\ref{prop:dynamics_of_dpomdp_separated_ensure_mdpomdp} that
  Problem~\eqref{eq:dpomdp_general_formulation} is a \Mdpomdp.
\end{proof}


\end{document}